\documentclass[12pt,reqno]{amsart}
\usepackage{amsmath,amsfonts,amssymb}
\allowdisplaybreaks
\usepackage[utf8]{inputenc}
\usepackage[english]{babel}
\usepackage[T1]{fontenc}
\usepackage{lmodern}
\usepackage{mathtools}
\usepackage{scrextend}
\usepackage{enumitem}
\usepackage{graphicx}
\usepackage{stmaryrd}

\usepackage[all]{xy}

\usepackage{bm}%\bm to bold math
\usepackage[normalem]{ulem}%\sout to strike through
\numberwithin{equation}{section}
\usepackage{xcolor}

\usepackage{centernot}

\usepackage[initials]{amsrefs}

\usepackage{tikz}
\usetikzlibrary{intersections}
\usetikzlibrary {calc}
\usetikzlibrary{positioning}
\usetikzlibrary{hobby}
\makeatletter
\tikzset{
	mybox/.cd,
	dir/.code args={(#1) to (#2)}{
		\def\mybox@start{#1}
		\def\mybox@target{#2}
	},
	left/.code args={#1:#2:#3}{
		\def\mybox@left@out{#1}
		\def\mybox@left@in{#2}
		\def\mybox@left@dist{#3}
	},
	right/.code args={#1:#2:#3}{
		\def\mybox@right@out{#1}
		\def\mybox@right@in{#2}
		\def\mybox@right@dist{#3}
	},
}
\newcommand\mybox[2][]{
	\def\mybox@left@out{45}
	\def\mybox@left@in{45}
	\def\mybox@right@out{45}
	\def\mybox@right@in{45}
	\def\mybox@left@dist{1cm}
	\def\mybox@right@dist{1cm}
	\pgfqkeys{/tikz/mybox}{#2}
	\draw[relative, line join=round,#1] (\mybox@start) to[out=\mybox@left@out, in=180+(-1)*(\mybox@left@in), distance=\mybox@left@dist] (\mybox@target)
	to[out=\mybox@right@in, in=180+(-1)*(\mybox@right@out), distance=\mybox@right@dist] (\mybox@start) -- cycle;
}
\makeatother

\title{Zooming in at the root of the stable tree}
\date{\today}
\usepackage[foot]{amsaddr}
\author{Michel Nassif}
\address{CERMICS, Ecole des Ponts, France}
\email{michel.nassif@enpc.fr}

\usepackage[text={500pt,620pt},centering]{geometry}
\usepackage[colorlinks=true, pdfstartview=FitV, linkcolor=blue, citecolor=red, urlcolor=blue,pagebackref=false]{hyperref}
\usepackage{cleveref}

\crefformat{equation}{(#2#1#3)}
\crefrangeformat{equation}{(#3#1#4)--(#5#2#6)}
\crefmultiformat{equation}{(#2#1#3)}%
{ and~(#2#1#3)}{, (#2#1#3)}{ and~(#2#1#3)}

\makeatletter%indent subsections in table of contents
\def\l@subsection{\@tocline{2}{0pt}{2.5pc}{5pc}{}}
\makeatother

\usepackage{amsthm}

\theoremstyle{plain}
\newtheorem{thm}{Theorem}[section]%reset counter at every new section
\newtheorem{corollary}[thm]{Corollary}%same counter as theorem
\newtheorem{proposition}[thm]{Proposition}
\newtheorem{lemma}[thm]{Lemma}%don't reset counter after every new theorem
\newtheorem*{thm*}{Theorem}
\newtheorem*{proposition*}{Proposition}

\theoremstyle{definition}
\newtheorem{remark}[thm]{Remark}

\newtheorem*{remark*}{Remark}

\usepackage{amsopn}

\newcommand{\norm}[1]{\left\lVert#1\right\rVert}
\newcommand{\real}{\mathbb{R}}
\newcommand{\dd}{\mathrm{d}}
\newcommand{\ex}[1]{\operatorname{\mathbb{E}}\left[#1\right]}
\newcommand{\n}{\operatorname{\mathbb{N}}}
\newcommand{\m}{\operatorname{\mathbb{N}^\mathrm{B}}}

\mathchardef\mhyphen="2D

\renewcommand{\epsilon}{\varepsilon}
\renewcommand{\phi}{\varphi}
\newcommand{\ghp}{d_\mathrm{GHP}}

\newcommand{\M}{\mathcal{M}}

\newcommand{\law}{\overset{\scriptscriptstyle (d)}{=}} %\scriptscriptstyle
\newcommand{\lawd}{\overset{(d)}{=}}

\newcommand{\ind}[1]{\mathbf{1}_{\left\{#1\right\}}}
\newcommand{\tree}{\mathsf{T}}
\newcommand{\rdtree}{\mathcal{T}}

\renewcommand{\H}{\mathfrak{h}}

\renewcommand{\root}{\emptyset}
\newcommand{\T}{\mathbb{T}}

\newcommand{\e}{\mathrm{e}}

\newcommand{\excm}[1]{\operatorname{\mathbb{N}}^{(#1)}}

\newcommand{\condex}[2]{\operatorname{\mathbb{E}}\left[#1\middle| #2\right]}
\newcommand{\height}{\eta}
\newcommand{\mass}{\tau}

\newcommand{\Tdown}{\mathsf{T}^{\downarrow}}
\newcommand{\Tint}[2]{\mathsf{T}_{[#1,#2)}}
\newcommand{\treeint}[3]{\rdtree_{[#1,#2),#3}}
\renewcommand{\for}[1]{\operatorname{\mathbb{P}}_{\! #1}}

\newcommand{\Z}{\mathbf{Z}}
\newcommand{\Bex}{B^{\mathrm{ex}}}

\newcommand{\D}{D[0,\infty)}
\newcommand{\f}{\mathfrak{f}}
\newcommand{\g}{\mathfrak{g}}
\newcommand{\expp}[1]{\exp\left\lbrace#1\right\rbrace}

\newcommand{\norma}{\operatorname{\mathsf{norm}_\gamma}}
\begin{document}

\begin{abstract}
	We study the shape of the normalized stable Lévy tree $\rdtree$ near its root. We show that, when zooming in at the root at the proper speed with a scaling depending on the index of stability, we get the unnormalized Kesten tree. In particular the limit is described by a tree-valued Poisson point process which does not depend on the initial normalization. We apply this to study the asymptotic behavior of additive functionals of the form
	\[\mathbf{Z}_{\alpha,\beta}=\int_{\mathcal{T}} \mu(\mathrm{d} x) \int_0^{H(x)} \sigma_{r,x}^\alpha \mathfrak{h}_{r,x}^\beta\,\mathrm{d} r\]
	as $\max(\alpha,\beta) \to \infty$, where $\mu$ is the mass measure on $\mathcal{T}$, $H(x)$ is the height of $x$ and $\sigma_{r,x}$ (resp. $\mathfrak{h}_{r,x}$) is the mass (resp. height) of the subtree of $\mathcal{T}$ above level $r$ containing $x$. Such functionals arise as scaling limits of additive functionals of the size and height on conditioned Bienaymé-Galton-Watson trees.
\end{abstract}

\keywords{Lévy trees, additive functionals, scaling limit}

\subjclass[2010]{60J80, 60G55, 60G52}

\maketitle

\section{Introduction}
 Stable trees are special instances of Lévy trees which were introduced by Le Gall and Le Jan \cite{le1998branching} in order to generalize Aldous' Brownian tree \cite{aldous1991continuum}. More precisely, stable trees are compact weighted rooted real trees depending on a parameter $\gamma \in (1,2]$, with $\gamma=2$ corresponding to the Brownian tree, which encode the genealogical structure of continuous-state branching processes with branching mechanism $\psi(\lambda) = \lambda^\gamma$. As such, they are the possible scaling limits of Bienaymé-Galton-Watson trees with critical offspring distribution belonging to the domain of attraction of a stable distribution with index $\gamma \in (1,2]$, see Duquesne \cite{duquesne2003limit} and Kortchemski \cite{kortchemski2013simple}. They also appear as scaling limits of various models of trees and graphs, see e.g. Haas and Miermont \cite{haas2012scaling}, and are intimately related to fragmentation and coalescence processes, see Miermont \cite{miermont2003fragmentations,miermont2005fragmentations} and Berestycki, Berestycki and Schweinsberg \cite{berestycki2007beta}. Stable trees can be defined via the normalized excursion of the so-called height process which is a local time functional of a spectrally positive Lévy process. We refer to Duquesne and Le Gall \cite{duquesne2002random} for a detailed account. See also Duquesne and Winkel \cite{duquesne2007growth}, Goldschmidt and Haas \cite{haas2015line}, Marchal \cite{marchal2008note} for alternative constructions.

 In the present paper, we study the shape of the normalized stable tree $\rdtree$ (\emph{i.e.} the stable tree conditioned to have total mass $1$) near its root. More precisely we show that, after zooming in at the root of $\rdtree$ and rescaling, one gets the Kesten tree, that is an infinite branch on which subtrees are grafted according to a Poisson point process. In particular, the (rescaled) subtrees near the root of $\rdtree$ are independent and the conditionning for the total mass to be equal to $1$ disappears when zooming in. This idea to zoom in at the root of the stable tree is closely related to the small time asymptotics -- present in the works of Miermont \cite{miermont2003fragmentations} and Haas \cite{haas2007fragmentation} -- of the self-similar fragmentation process $F^{-}(t)$ obtained from the stable tree by removing vertices located under height $t$. See Remark \ref{remark: haas} in this direction. As a consequence, we obtain the asymptotic behavior of additive functionals on $\rdtree$ of the form
 \begin{equation}\label{eq: the functional on a Levy tree}
 	\Z_{\alpha,\beta}= \int_\rdtree Z_{\alpha,\beta}(x)\, \mu(\dd x) \qquad \text{with}\quad  \forall x \in \rdtree, \ Z_{\alpha,\beta}(x)=\int_0^{H(x)} \sigma_{r,x}^\alpha \H_{r,x}^\beta\, \dd r,
 \end{equation}
 where $\mu$ is the mass measure on $\rdtree$ which is a uniform measure supported by the set of leaves, $H(x)$ is the height of $x \in \rdtree$, that is its distance to the root, and $\sigma_{r,x}$ (resp. $\H_{r,x}$) is the mass (resp. height) of the subtree of $\rdtree$ above level $r$ containing $x$.

Before stating our results, we first introduce some notations. Let $\T$ be the space of weighted rooted compact real trees, that is the set of compact real trees $(T,d)$ endowed with a distinguished vertex $\root$ called the root and with a nonnegative finite measure $\mu$. We equip the set $\T$
with the Gromov-Hausdorff-Prokhorov topology, see Section \ref{section: real trees} for a precise definition. 

Define a rescaling map $R_\gamma\colon \T \times (0,\infty)\to \T$ by
\begin{equation}\label{eq: definition R}
	R_\gamma\left((T,\root,d,\mu),a\right) = \left(T,\root,ad,a^{\gamma/(\gamma-1)}\mu\right).
\end{equation}
In words, $R_\gamma((T,\root,d,\mu),a)$ is the tree obtained from $(T,\root,d,\mu)$ by multiplying all distances by $a$ and all masses by $a^{\gamma/(\gamma-1)}$. Moreover, define for every $(T,\root,d,\mu)\in \T$
\begin{equation}\label{eq: definition norma}
	\norma(T) = R_\gamma(T,\mu(T)^{-1+1/\gamma}),
\end{equation}
which is the tree $T$ normalized to have total mass $1$ and where distances are rescaled accordingly. Denote by $\excm{1}$ the distribution of the normalized stable tree with total mass $1$, see Section \ref{section: stable tree} for a precise definition. Under $\excm{1}$, let $U$ be a uniformly chosen leaf, that is $U$ is a $\rdtree$-valued random variable with distribution $\mu$. Denote by $\rdtree_i, \, i \in I_U$ the trees grafted on the branch $\llbracket \root, U\rrbracket$ joining the root $\root$ to the leaf $U$, each one at height $h_i$ and with total mass $\sigma_i = \mu(\rdtree_i)$, see Figure \ref{fig: random leaf}. Fix $\f\colon (0,\infty)\to (0,\infty)$ (this represents the speed at which we zoom in) and define for every $\epsilon >0$ a point measure on $[0,\infty)^2\times \T$ by
\begin{equation}
	\mathcal{N}_\epsilon^\f(U) = \sum_{h_i \leqslant \f(\epsilon) H(U)} \delta_{\left(\epsilon^{-1} h_i,\epsilon^{-\gamma/(\gamma-1)}\sigma_i, \norma(\rdtree_i)\right)}.
\end{equation}
Finally, for any metric space $X$, we denote by $\M_p(X)$ the space of point measures on $X$ equipped with the topology of vague convergence. 
\begin{figure}[h]
	\centering
	\begin{tikzpicture}
		\draw[thick]  (0,2) -- (0,5);
		\filldraw[fill=black] (0,2) circle [radius=2pt] node[right]{$\root$} ;
		\filldraw[fill=black] (0,5) circle [radius=2pt] node[right]{$U$} ;
		\draw [<->,blue] (1,2) --  (1,4.2) node[midway,right] {$h_i$};
		\draw[thick] plot [smooth] coordinates {(0,2.5) (0.3,2.8) (0.4,3)};
		\draw[thick] plot [smooth] coordinates {(0,3) (-0.2,3.2) (-0.4,3.5)};
		\draw[thick] plot [smooth] coordinates {(0,3.2) (0.2,3.3) (0.4,3.6)};
		\draw[thick] plot [smooth] coordinates {(0,3.5) (-0.2,3.8) (-0.4,4)};
		\draw[thick] plot [smooth] coordinates {(0,3.7) (0.3,3.9) (0.5,3.95)};
		\draw[blue,thick] plot [smooth] coordinates {(0,4.2) (-0.2, 4.3) (-0.3,4.5) (-0.5,4.9)};
		\draw[blue,thick] plot [smooth] coordinates {(-0.3,4.5) (-0.2,4.7) (-0.2, 4.85)};
		\draw[blue,thick] plot [smooth] coordinates { (-0.2, 4.3) (-0.3,4.3) (-0.4,4.25)};
		\node[blue] at (-0.7,4.6) {$\rdtree_i$};
	\end{tikzpicture}
	\caption{The subtrees $\rdtree_i$ grafted on the branch $\llbracket \root, U\rrbracket$ at height $h_i$.}
	\label{fig: random leaf}
\end{figure}

Our first main result states that the measure $\mathcal{N}_\epsilon^\f(U)$ converges to a Poisson point process which is independent of the underlying tree $\rdtree$ and of $H(U)$. 
\begin{thm}\label{thm: zooming intro} Let $\rdtree$ be the normalized stable tree with branching mechanism $\psi(\lambda) = \lambda^\gamma$ where $\gamma \in (1,2]$. Conditionally on $\rdtree$, let $U$ be a $\rdtree$-valued random variable with distribution $\mu$ under $\excm{1}$. 
Let $(\tree_s', \, s\geqslant 0)$ be a Poisson point process with intensity $\m$ given by \eqref{eq: intensity measure}, independent of $(\rdtree,H(U))$. 
	Let $\Phi \colon [0,\infty)^2\times \T\to [0,\infty)$ be a measurable function such that there exists $C>0$ such that for every $h\geqslant 0$ and $T\in \T$, we have
	\begin{equation*}
		\left|\Phi(h,b,T)-\Phi(h,a,T)\right|\leqslant C|b-a|.
	\end{equation*}
	\begin{enumerate}[label=(\roman*),leftmargin=*]
		\item If $\lim_{\epsilon \to 0}\epsilon^{-1/2}\f(\epsilon)=0$ and $\lim_{\epsilon \to 0}\epsilon^{-1}\f(\epsilon) = \infty$, then we have the following convergence in distribution
		\begin{equation}\label{eq: cv ppp 1}
			\left(\rdtree, H(U), \langle\mathcal{N}_\epsilon^\f(U),\Phi\rangle\right) \xrightarrow[\epsilon \to 0]{(d)} \left(\rdtree, H(U), \sum_{s\geqslant 0} \Phi\left(s,\mu(\tree_s'),\norma\left(\tree_s'\right)\right)\right),
		\end{equation}
		in the space $\T\times [0,\infty)\times [0,\infty]$. In particular, we have the following convergence in distribution in $\T \times [0,\infty) \times \M_p([0,\infty)\times \T)$.
		\begin{equation}\label{eq: cv ppp 2}
			\left(\rdtree, H(U),  \sum_{h_i \leqslant \f(\epsilon) H(U)} \delta_{\left(\epsilon^{-1} h_i, R_\gamma(\rdtree_i,\epsilon^{-1})\right)}\right) \xrightarrow[\epsilon \to 0]{(d)} \left(\rdtree, H(U), \sum_{s\geqslant 0} \delta_{\left(s,\tree_s'\right)}\right).
		\end{equation}
		\item If $\f(\epsilon) = \epsilon$, then we have the following convergence in distribution
		\begin{equation}
			\left(\rdtree, H(U), \langle\mathcal{N}_\epsilon^\f(U),\Phi\rangle\right) \xrightarrow[\epsilon \to 0]{(d)} \left(\rdtree, H(U), \sum_{s\leqslant H(U)} \Phi\left(s,\mu(\tree_s'),\norma\left(\tree_s'\right)\right)\right)
		\end{equation}
		in the space $\T\times [0,\infty)\times [0,\infty]$.
	\end{enumerate}
\end{thm}
In other words, zooming in at the speed $\f(\epsilon) = \epsilon$ gives a \emph{finite} branch on which subtrees are grafted in a Poissonian manner, whereas zooming in at a slower speed gives an \emph{infinite} branch at the limit. Notice that the convergence \eqref{eq: cv ppp 1} is stronger than convergence in distribution for the vague topology \eqref{eq: cv ppp 2} as it holds for functions $\Phi$ with very few regularity assumptions: $\Phi(h,a,T)$ is only Lipschitz-continuous with respect to $a$ instead of (Lipschitz-)continuous with respect to $(h,a,T)$ with bounded support. In particular, this could allow to consider local time functionals of the tree.

As an application of this result, we study the asymptotic behavior as $\max(\alpha,\beta)\to \infty$ of additive functionals $\Z_{\alpha,\beta}$ on the stable tree $\rdtree$. Such functionals arise as scaling limits of additive functionals of the size and height on conditioned Bienaymé-Galton-Watson trees, see Delmas, Dhersin and Sciauveau \cite{delmas2018} or Abraham, Delmas and Nassif \cite{abraham2020global} where it is shown that $\Z_{\alpha,\beta}<\infty$ a.s. if (and only if) $\gamma \alpha + (\gamma-1)(\beta+1)>0$, see Corollary 6.10 therein. In the present paper, we only consider $\alpha,\beta\geqslant 0$ which guarantees in particular the finiteness of $\Z_{\alpha,\beta}$. For example, let us mention the total path length and the Wiener index which properly scaled converge respectively to $\Z_{0,0}$ and $\Z_{1,0}$. Fill and Janson \cite{fillsum} considered the case $\gamma=2$ and $\beta =0$ (\emph{i.e.} functionals of the mass on the Brownian tree) and proved that there is convergence in distribution as $\alpha \to \infty$ of $\Z_{\alpha,0}$ properly normalized to
\begin{equation*}\label{eq: FJ limit}
	\int_0^\infty \e^{-S_t}\, \dd t,
\end{equation*}
where $(S_t, \, t \geqslant 0)$ is a $1/2$-stable subordinator. Their proof relies on the connection between the normalized Brownian excursion which codes the Brownian tree and the three-dimensional Bessel bridge. Our aim is twofold: we extend their result to the non-Brownian stable case $\gamma \in (1,2)$ while also considering polynomial functionals depending on both the mass and the height. We use a different approach relying on the Bismut decomposition of the stable tree.

Going back to the connection with the fragmentation process $F^{-}(t) = \allowdisplaybreaks (F^{-}_1(t), F^{-}_2(t),\ldots)$, it is not hard to see thatthe additive functional $\Z_{\alpha,0}$ can be expressed in terms of $F^{-}$ as
\begin{equation*}
	\Z_{\alpha,0} = \sum_{i\geqslant 1} \int_0^\infty F_i^{-}(t)^{\alpha+1}\, \dd t.
\end{equation*}
Once this is established, one can argue that only the largest fragment $F_1^{-}$ contributes to the limit, the others being negligible, then use \cite[Corollary 17]{haas2007fragmentation} which implies that $1-F_1^{-}$ properly normalized converges in distribution to a $(1-1/\gamma)$-stable subordinator $S$, to get the convergence of $\Z_{\alpha,0}$ to $\int_0^\infty \e^{-S_t}\, \dd t$. In the present paper, we do not adopt this approach as it does not allow to consider functionals of the height (that is $\beta \neq 0$).

We distinguish two regimes according to the behavior of $\beta/\alpha^{1-1/\gamma}$. The regime $\beta/\alpha^{1-1/\gamma} \to c \in [0,\infty)$ is related to Theorem \ref{thm: zooming intro} and the result in that case can be stated as follows, see Theorem \ref{thm: subcritical case} for a more general statement.
\begin{thm}\label{thm: subcritical}
	Assume that $\alpha \to \infty$, $\beta \geqslant 0$ and $\beta/\alpha^{1-1/\gamma} \to c \in [0,\infty)$. Let $\rdtree$ be the normalized stable tree with branching mechanism $\psi(\lambda) = \lambda^\gamma$ where $\gamma\in (1,2]$ and denote by $\H$ its height. Then we have the following convergence in distribution under $\excm{1}$
		\begin{equation}
			 \alpha^{1-1/\gamma}\H^{-\beta}\Z_{\alpha,\beta}
			\xrightarrow[\alpha \to \infty]{(d)} \int_0^\infty \e^{-S_t-ct/\H}\, \dd t,
		\end{equation}
	where $(S_t, t \geqslant 0)$ is a stable subordinator with Laplace exponent $\phi(\lambda)=\gamma \lambda^{1-1/\gamma}$, independent of $\rdtree$.
\end{thm}

Let us briefly explain why we get a subordinator $S$ at the limit. It is well known that $\mu$ is supported on the set of leaves of $\rdtree$. Let $x\in \rdtree$ be a leaf and recall that $\sigma_{r,x}$ is the mass of the subtree above level $r$ containing $x$.  Since the total mass of the stable tree is $1$, the main contribution to $Z_{\alpha,\beta}(x)$ as $\alpha \to \infty$ comes from large subtrees $\rdtree_{r,x}$ with $r$ close to $0$. The height $\H_{r,x}$ of such subtrees is approximately $\H - r$. On the other hand, their mass is equal to $1$ minus the mass we discarded from the subtrees grafted on the branch $\llbracket \root,x\rrbracket$ at height less than $r$. By Theorem \ref{thm: zooming intro}, subtrees are grafted on $\llbracket \root,x\rrbracket$ according to a point process which is approximately Poissonian, at least close to the root $\root$. Thus the mass $\sigma_{r,x}$ is approximately $1-S_r$.

Theorem \ref{thm: subcritical case} is slightly more general: we prove joint convergence in distribution of $\alpha^{1-1/\gamma}\H^{-\beta}\Z_{\alpha,\beta}$ and $\alpha^{1-1/\gamma} \H^{-\beta}Z_{\alpha,\beta}(U)$, where $U\in \rdtree$ is a leaf chosen uniformly at random (\textit{i.e.} according to the measure $\mu$), to the same random variable. In other words, taking the average of $Z_{\alpha,\beta}(x)$ over all leaves yields the same asymptotic behavior as taking a leaf uniformly at random. This is due to the following observations: a) a uniform leaf $U$ is not too close to the root with high probability in the sense that its most recent common ancestor with $x^*$ has height greater than $\epsilon$, where $x^*$ is the heighest leaf of $\rdtree$, b) when taking the average over all leaves, the contribution of those leaves whose most recent common ancestor with $x^*$ has height less than $\epsilon$ is negligible, and c) for those $x \in \rdtree$ whose most recent common ancestor with $x^*$ has height greater than $\epsilon$, the main contribution to $Z_{\alpha,\beta}(x)$ comes from large subrees $\rdtree_{r,x}$ with $r \leqslant \epsilon$, these subtrees are common to all such leaves as $\rdtree_{r,x} = \rdtree_{r,x^*}$. This is made rigorous in Lemma \ref{lemma: convergence in proba}.

Let us make a connection with Theorem 1.18 of Fill and Janson \cite{fillsum}. Recall that the normalized Brownian tree with branching mechanism $\psi(\lambda) = \lambda^2$ is coded by $\sqrt{2}\Bex$ where $\Bex$ is the normalized Brownian excursion, see \cite{duquesne2002random}. Thanks to the representation formula of \cite[Lemma 8.6]{delmas2018}, we see that Fill and Janson's $Y(\alpha)=\sqrt{2}\Z_{\alpha-1,0}$. Thus, we recover their result in the Brownian case $\gamma=2$ when $\beta =0$ (in which case $c=0$).

Notice that as long as the exponent $\beta$ of the height does not grow too quickly, viz. $\beta/\alpha^{1-1/\gamma} \to 0$, the additional dependence on the height makes no contribution at the limit. On the other hand, in the regime $\beta/\alpha^{1-1/\gamma} \to \infty$, the height $\H_{r,x}^\beta$ dominates the mass $\sigma_{r,x}^\alpha$ so we get the convergence in probability of $\Z_{\alpha,\beta}$ with a different scaling and there is no longer a subordinator at the limit. See Theorem \ref{thm: supercritical case} for a more general statement.
\begin{thm}\label{thm: supercritical}
	Assume that $\beta \to \infty$, $\alpha \geqslant 0$ and $\alpha^{1-1/\gamma}/\beta\to 0$. Let $\rdtree$ be the normalized stable tree with branching mechanism $\psi(\lambda) = \lambda^\gamma$ where $\gamma\in (1,2]$. Then we have the following convergence in $\excm{1}$-probability
	\begin{equation}
		\lim_{\beta\to \infty} \beta \H^{-\beta}\Z_{\alpha,\beta} = \H.
	\end{equation}
\end{thm}
\begin{remark}
Assume that $\alpha,\beta \to \infty$ and $\beta/\alpha^{1-1/\gamma} \to c \in (0,\infty)$ so that Theorem \ref{thm: subcritical} applies. Then we have the convergence in distribution under $\excm{1}$
	\begin{equation*}
		\beta \H^{-\beta} \Z_{\alpha,\beta} = \frac{\beta}{\alpha^{1-1/\gamma}} \alpha^{1-1/\gamma}\H^{-\beta}\Z_{\alpha,\beta} \xrightarrow[\beta \to \infty]{(d)} c \int_0^\infty \e^{-S_t-ct/\H} \, \dd t = \H\int_0^\infty \e^{-S_{\H t/c}-t}\, \dd t.
	\end{equation*}
Now letting $c \to \infty$, the right-hand side converges to $\H\int_0^\infty \e^{-t}\, \dd t =\H$. Thus, one may view Theorem \ref{thm: supercritical} as a special case of Theorem \ref{thm: subcritical} by saying that, if $\beta \to \infty$ and $\beta/\alpha^{1-1/\gamma} \to c\in(0,\infty]$, then we have the convergence in distribution under $\excm{1}$
\begin{equation*}
	\beta\H^{-\beta}\Z_{\alpha,\beta}
	\xrightarrow[\beta \to \infty]{(d)} c\int_0^\infty \e^{-S_t-ct/\H}\, \dd t,
\end{equation*}
where the measure $c \e^{-ct/\H}\, \dd t$ on $[0,\infty)$ should be understood as $\H\delta_0$ if $c=\infty$.
\end{remark}

We conclude the introduction by giving a decomposition of a general (compact) Lévy tree used in the proof of Theorem \ref{thm: subcritical} which is of independent interest. Consider a Lévy tree $\rdtree$ under its excursion measure $\n$ associated with a branching mechanism $\psi(\lambda) = a\lambda + b\lambda^2+\int_0^\infty (\e^{-\lambda r} -1+\lambda r)\, \pi(\dd r)$ where $a,b \geqslant 0$ and $\pi$ is a $\sigma$-finite measure on $(0,\infty)$ satisfying $\int_0^\infty (r\wedge r^2)\,\pi(\dd r)<\infty$. We further assume that the Grey condition holds $\int^\infty \dd \lambda/\psi(\lambda)<\infty$ which is equivalent to the compactness of the Lévy tree. We refer to \cite[Section 1]{duquesne2002random} for a complete presentation of the subject. For every $x \in \rdtree$ and every $0\leqslant r< r'\leqslant H(x)$, we let $\rdtree_{[r,r'),x} = (\rdtree_{r,x}\setminus \rdtree_{r',x})\cup \{x_{r'}\}$ where $x_{r'}$ is the unique ancestor of $x$ at height $H(x_{r'}) =r'$ and $\rdtree_{r,x}$ is the subtree of $\rdtree$ above level $r$ containing $x$. The following result states that, when $x \in \rdtree$ and $0\eqqcolon r_0< r_1< \ldots < r_n <r_{n+1}\coloneqq H(x)$ are chosen “uniformly” at random under $\n$, then the random trees $\treeint{r_{i-1}}{r_{i}}{x}$, $1\leqslant i\leqslant n+1$ are independent and distributed as $\rdtree$ under $\n[\sigma\bullet]$, see Figure \ref{fig: decomposition n+1 subtrees}. In particular, this generalizes \cite[Lemma 6.1]{abraham2020global} which corresponds to $n=1$.
\begin{figure}[h]
	\centering
	\begin{tikzpicture}[>=stealth]
		\draw[thick, name path=tree] plot [smooth] coordinates {(0,0) (1,1)  (2,3) (2,4) (1.8,5)};
		\draw[->] (4,0) -- (4,5.5) node[right]{$r$};
		\draw ($(4,0)+(-2pt,0)$) -- ($(4,0)+(2pt,0)$) node[right]{$0$};
		\draw [blue,name path=tree1](0.74,0.65) circle [radius=1cm]node[label={[label distance=1cm]180:$\treeint{0}{r_1}{x}$}]{};
		\draw [red,name path=tree2]($(1.6,2.1)+(2pt,3pt)$) circle [radius=0.8cm]node[label={[label distance=1cm]180:$\treeint{r_1}{r_2}{x}$}]{};
		\draw [green,name path=tree3]($(2,3.3)+(3pt,1pt)$) circle [radius=0.4cm]node[label={[label distance=1cm]180:$\treeint{r_2}{r_3}{x}$}]{};
		\draw [magenta,name path=tree4]($(2,4.4)+(-1pt,-1pt)$) circle [radius=0.64cm]node[label={[label distance=1cm]180:$\treeint{r_3}{H(x)}{x}$}]{};
		\draw ($(4,5)+(-2pt,0)$) -- ($(4,5)+(2pt,0)$) node[right]{$H(x)$};
		\node [name intersections={of=tree and tree2,by={a,b}}]  at (b){};
		\draw [densely dotted] (b)--($(4,0)!(b)!(4,4)$) node[right]{$r_1$};
		\node [name intersections={of=tree and tree3,by={c,d}}]  at (c){};
		\draw [densely dotted] (c)--($(4,0)!(c)!(4,4)$) node[right]{$r_2$};
		\draw [densely dotted] (d)--($(4,0)!(d)!(4,4)$) node[right]{$r_3$};
		\node at (0.5,-1){$\rdtree$};
		\filldraw[fill=black](0,0) circle [radius=2pt] node[below left]{$\root$};
		\filldraw[fill=black](1.8,5) circle [radius=2pt] node[above]{$x$};
	\end{tikzpicture}
	\caption{The decomposition of $\rdtree$ under $\n$ into $n+1$ (with $n=3$) subtrees along the ancestral line of a uniformly chosen leaf $x$.}
	\label{fig: decomposition n+1 subtrees}
\end{figure}
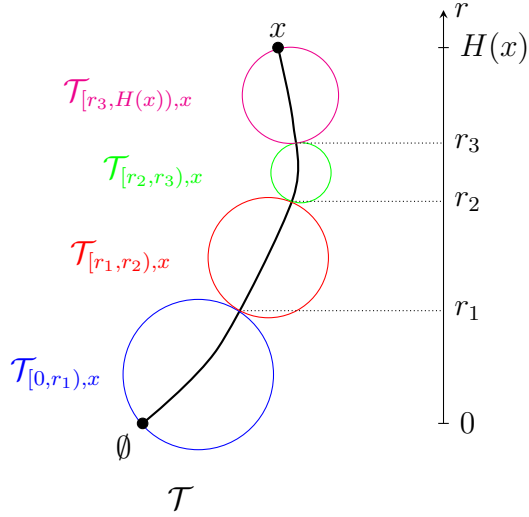
\begin{thm}\label{thm: decomposition n+1 intro}
	Let $\rdtree$ be the Lévy tree with a general branching mechanism $\psi$ satisfying the Grey condition $\int^\infty \dd \lambda/\psi(\lambda) <\infty$ under its excursion measure $\n$. Then for every $n\geqslant 1$ and every nonnegative measurable functions $f_i, \ 1\leqslant i \leqslant n+1$ defined on $[0,\infty)\times\T$, we have with $r_0 = 0$ and $r_{n+1} = H(x)$
\begin{equation*}
	\n\left[\int_\rdtree \mu(\dd x) \int_{0< r_1< \ldots < r_n < H(x)} \prod_{i=1}^{n+1} f_i \left(r_i-r_{i-1},\treeint{r_{i-1}}{r_{i}}{x}\right) \,\prod_{i=1}^n \dd r_i\right]=\prod_{i=1}^{n+1} \n\left[\int_\rdtree \mu(\dd x) f_i(H(x),\rdtree)\right].
\end{equation*}
In particular, for every nonnegative measurable functions $g_i$, $1\leqslant i \leqslant n+1$ defined on $\T$, we have
\begin{equation*}
	\n\left[\int_\rdtree \mu(\dd x) \int_{0< r_1< \ldots < r_n < H(x)} \prod_{i=1}^{n+1} g_i \left(\treeint{r_{i-1}}{r_{i}}{x}\right) \,\prod_{i=1}^n \dd r_i\right] = \prod_{i=1}^{n+1} \n\left[\sigma g_i(\rdtree)\right].
\end{equation*}
\end{thm}

A consequence of this decomposition is the following result giving the joint distribution of $\rdtree_y$, the subtree of $\rdtree$ above vertex $y\in \rdtree$, and $H(y)$ when $y$ is chosen according to the length measure $\ell(\dd y)$ on the stable tree $\rdtree$ (which roughly speaking is the Lebesgue measure on the branches of $\rdtree$). In particular, this generalizes \cite[Proposition 1.6]{abraham2020global}.
\begin{corollary}
	Let $\rdtree$ be the normalized stable tree with branching mechanism $\psi(\lambda) = \lambda^\gamma$ where $\gamma \in (1,2]$. Let $f$ and $g$ be nonnegative measurable functions defined on $\T$ and $[0,\infty)$ respectively. We have
	\begin{equation}
		\excm{1}\left[\int_\rdtree f(\rdtree_y) g(H(y))\, \ell(\dd y)\right] = \n \left[\ind{\sigma <1} (1-\sigma)^{-1/\gamma} G(1-\sigma) f(\rdtree)\right]
	\end{equation}
	where 
	\begin{equation*}
		G(a) = \excm{1}\left[\int_\rdtree \mu(\dd x) g\left(a^{1-1/\gamma}H(x)\right)\right] , \quad \forall a >0.
	\end{equation*}
\end{corollary}

The paper is organized as follows. In Section \ref{section: real trees} we define the space of real trees and the Gromov-Hausdorff-Prokhorov topology. In Section \ref{section: stable tree}, we introduce the stable tree, recall some of its properties and prove Theorem \ref{thm: decomposition n+1 intro} as well as some other useful results. In Section \ref{section: zooming}, we prove Theorem \ref{thm: zooming intro}. Sections \ref{section: subcritical} and \ref{section: supercritical} deal with the asymptotic behavior of $\Z_{\alpha,\beta}$ when $\beta/\alpha^{1-1/\gamma} \to c \in [0,\infty)$ and $\beta /\alpha^{1-1/\gamma} \to \infty$ respectively. Finally, we gather some technical proofs in Section \ref{section: technical}.

\section{Real trees and the Gromov-Hausdorff-Prokhorov topology}\label{section: real trees}
\subsection{Real trees}
We recall the formalism of real trees, see \cite{evans2007probability}. A metric space $(T,d)$ is a real tree if the following two properties hold for every $x,y \in T$.
\begin{enumerate}[label=(\roman*),leftmargin=*]
	\item (Unique geodesics). There exists a unique isometric map $f_{x,y} \colon [0,d(x,y)] \to T$ such that $f_{x,y}(0) = x$ and $f_{x,y}(d(x,y)) = y$.
	\item (Loop-free). If $\phi$ is a continuous injective map from $[0,1]$ into $T$ such that $\phi(0) = x$ and $\phi(1) = y$, then we have
	\[\phi([0,1]) = f_{x,y}\left([0,d(x,y)]\right).\]
\end{enumerate}
A weighted rooted real tree $(T, \root, d , \mu)$ is a real tree $(T,d)$ with a distinguished vertex $\root \in T$ called the root and equipped with a nonnegative finite measure $\mu$. Let us consider a weighted rooted real tree $(T,\root,d,\mu)$. The range of the mapping $f_{x,y}$ described above is denoted by $\llbracket x,y\rrbracket$ (this is the line segment between $x$ and $y$ in the tree). In particular, $\llbracket\root,x\rrbracket$ is the path going from the root to $x$ which we will interpret as the ancestral line of vertex $x$. We define a partial order on the tree by setting $x \preccurlyeq y$ ($x$ is an ancestor of $y$) if and only if $x \in \llbracket\root,y\rrbracket$. If $x,y \in T$, there is a unique $z \in T$ such that $\llbracket\root,x\rrbracket \cap \llbracket\root, y\rrbracket = \llbracket\root,z \rrbracket$. We write $z = x \wedge y$ and call it the most recent common ancestor to $x$ and $y$. For every vertex $x \in T$, we define its height by $H(x) = d(\root, x)$. The height of the tree is defined by $\H(T) = \sup_{x\in T}H(x)$. Note that if $(T,d)$ is compact, then $\H(T) < \infty$.

Let $x \in T$ be a vertex. For every $r \in [0,H(x)]$, we denote by $x_r \in T$ the unique ancestor of $x$ at height $r$. Furthermore, we define the subtree $T_{r,x}$ of $T$ above level $r$ containing $x$ as
\begin{equation}\label{eq: definition T_rx}
	T_{r,x} = \left\{y \in T \colon \, H(x \wedge y) \geqslant r\right\}.
\end{equation}
Equivalently, $T_{r,x}=\{y \in T\colon\, x_r \preccurlyeq y\}$ is the subtree of $T$ above $x_r$. Then $T_{r,x}$ can be naturally viewed as a weighted rooted real tree, rooted at $x_r$ and endowed with the distance $d$ and the measure $\mu_{| T_{r,x}}$. Note that $T_{0,x} = T$. We also define the subtree of $T$ above $x$ by $T_x\coloneqq T_{H(x),x}$. Denote by
\begin{equation}
	\sigma_{r,x}(T) = \mu(T_{r,x}) \quad \text{and} \quad \H_{r,x}(T) = \H(T_{r,x})
\end{equation}
the total mass and the height of $T_{r,x}$. For every $\alpha, \beta \geqslant 0$, we define
\begin{equation}
		Z^T_{\alpha,\beta}(x) = \int_0^{H(x)}\sigma_{r,x}(T)^\alpha \H_{r,x}(T)^\beta\, \dd r, \quad \forall x \in T.
\end{equation}
We shall omit the dependence on $T$ when there is no ambiguity, simply writing $\sigma_{r,x}$, $\H_{r,x}$ and $Z_{\alpha,\beta}(x)$. For every $0\leqslant r <r'\leqslant H(x)$, we also introduce the notation
\begin{equation}\label{eq: definiton Tint}
	T_{[r,r'),x}	= \left(T_{r,x}\setminus T_{r',x}\right) \cup \{x_{r'}\}=\{y\in T\colon \, r \leqslant H(x\wedge y)<r'\}\cup \{x_{r'}\},
\end{equation}
which defines a weighted rooted real tree, equipped with the distance and the measure it inherits from $T$ and naturally rooted at $x_{r}$.

The next lemma, whose proof is elementary, relates $\H_{r,x}(T)$, the height of the subtree of $T$ above level $r$ containing $x$, to the total height $\H(T)$.
\begin{lemma}\label{lemma: height is linear}
	Let $T$ be a compact real tree. For every $x\in T$ and $r \in [0,H(x)]$, we have
	\begin{equation}\label{eq: height is sublinear}
		\H(T) \geqslant \H_{r,x}(T)+r.
	\end{equation}
	Furthermore, if $x^* \in T$ is such that $H(x^*) = \H(T)$, then for every $r \in [0,H(x\wedge x^*)]$, we have
	\begin{equation}\label{eq: height is linear}
		\H(T) = \H_{r,x}(T) + r.
	\end{equation}
\end{lemma}

\subsection{The Gromov-Hausdorff-Prokhorov topology}
We denote by $\T$ the set of (measure-preserving, root-preserving isometry classes of) compact real trees. We will often identify a class with an element of this class. So we shall write $(T,\root, d, \mu)\in \T$ for a weighted rooted compact real tree.

Let us define the Gromov-Hausdorff-Prokhorov (GHP) topology on $\T$. Let $(T,\root,d, \mu),\allowbreak(T',\root',d',\mu')\in \T$ be two compact real trees. Recall that a correspondence between $T$ and $T'$ is a subset $\mathcal{R} \subset T \times T'$ such that for every $x \in T$, there exists $x' \in T'$ such that $(x,x') \in \mathcal{R}$, and conversely, for every $x'\in T'$, there exists $x \in T$ such that $(x,x') \in \mathcal{R}$. In other words, if we denote by $p \colon T \times T'\to T$ (resp. $p'\colon T \times T' \to T'$) the canonical projection on $T$ (resp. on $T'$), a correspondence is a subset $\mathcal{R} \subset T \times T'$ such that $p(\mathcal{R}) = T$ and $p'(\mathcal{R}) = T'$. If $\mathcal{R}$ is a correspondence between $T$ and $T'$, its distortion is defined by
\[\operatorname{dis}(\mathcal{R}) = \sup \left\{\left|d(x,y) - d'(x',y')\right| \colon \, (x,x'), (y,y') \in \mathcal{R}\right\}.\]
Next, for any nonnegative finite measure $m$ on $T \times T'$, we define its discrepancy with respect to $\mu$ and $\mu'$ by
\[\operatorname{D}(m; \mu, \mu') = d_{\mathrm{TV}}(m\circ p^{-1} ,\mu) + d_{\mathrm{TV}}(m \circ {p'}^{-1} ,\mu'),\]
where $d_{\mathrm{TV}}$ denotes the total variation distance. Then the GHP distance between $T$ and $T'$ is defined as
\begin{equation}\label{eq: definition GHP} \ghp(T,T') = \inf\left\{\frac{1}{2}\operatorname{dis}(\mathcal{R})\vee \operatorname{D}(m; \mu,\mu') \vee m(\mathcal{R}^c)\right\},
\end{equation}
where the infimum is taken over all correspondences $\mathcal{R}$ between $T$ and $T'$ such that $(\root, \root') \in \mathcal{R}$ and all nonnegative finite measures $m$ on $T \times T'$. It can be verified that $\ghp$ is indeed a distance on $\T$ and that the space $(\T,\ghp)$ is complete and separable, see e.g. \cite{addario2017scaling}. 

The next lemma gives an upper bound for the GHP distance between a tree $(T,\root,d,\mu) \in \T$ and the tree $(T,\root,ad,b\mu)$ obtained from $T$ by multiplying all distances by $a>0$ and the measure $\mu$ by $b>0$. The proof is elementary and is left to the reader.
\begin{lemma}\label{lemma: ghp dilatation}
	For every $T\in \T$ and $a,b >0$, we have
	\begin{equation}
		\ghp\left((T,\root,d,\mu),(T,\root,ad,b\mu)\right) \leqslant 2|a-1|\H(T) + |b-1| \mu(T).
	\end{equation}
\end{lemma}

\section{The stable tree}\label{section: stable tree}
Here, we define the stable tree and recall some of its properties. We refer to \cite{duquesne2005probabilistic} for background. We shall work with the stable tree $\rdtree$ with branching mechanism $\psi(\lambda) = \lambda^\gamma$ where $\gamma\in (1,2]$ under its excursion measure $\n$: more explicitly, using the coding of compact real trees by height functions, one can define a $\sigma$-finite measure $\n$ on $\T$ with the following properties. 
\begin{enumerate}[label=(\roman*)]
	\item \textbf{Mass measure.} $\n$-a.e. the mass measure $\mu$ is supported by the set of leaves $\operatorname{Lf}(\rdtree) \coloneqq \{x \in \rdtree\colon \, \rdtree \setminus \{x\} \text{ is connected}\}$ and the distribution on $(0,\infty)$ of the total mass $\sigma \coloneqq \mu(\rdtree)$ is given by
	\begin{equation*}
	\n[\sigma \in \dd a] = \frac{1}{\gamma\Gamma(1-1/\gamma)}\,\frac{\dd a}{a^{1+1/\gamma}}\cdot
	\end{equation*}
	\item \textbf{Height.} $\n$-a.e. there exists a unique leaf $x^* \in \rdtree$ realizing the height, that is $H(x^*) = \H(\rdtree)$, and the distribution on $(0,\infty)$ of the height $\H \coloneqq \H(\rdtree)$ is given by
	\begin{equation*}
		\n[\H \in \dd a] = (\gamma-1)^{-\gamma/(\gamma-1)}\,\frac{\dd a}{a^{\gamma/(\gamma-1)}}\cdot
	\end{equation*}
\end{enumerate}

We will make extensive use of the scaling property of the stable tree under $\n$. Recall from \eqref{eq: definition R} the definition of $R_\gamma$ and note that if $T$ has total mass $\sigma$ and height $\H$ then $R_\gamma(T,a)$ has total mass $a^{\gamma/(\gamma-1)}\sigma$ and height $a \H$. Furthermore,
it is straightforward to show that for all $x \in T$, $r \in [0,H(x)]$ and $a>0$:
\begin{align}\label{eq: scaling Z}
	\sigma_{ar,x}(R_\gamma(T,a)) &= a^{\gamma/(\gamma-1)} \sigma_{r,x}(T), \nonumber\\
	\H_{ar,x}(R_\gamma(T,a)) &= a\H_{r,x}(T),\nonumber\\
	Z_{\alpha,\beta}^{R_\gamma(T,a)}(x) &= a^{\alpha\gamma/(\gamma-1)+\beta+1} Z_{\alpha,\beta}^T(x).
\end{align}

The scaling property of the stable tree can be written as follows:
\begin{equation}\label{eq: scaling property under N}
	R_\gamma(\rdtree,a) \quad \text{under} \ \n \quad \lawd \quad \rdtree \quad \text{under} \ a^{1/(\gamma-1)} \n,
\end{equation}
see e.g. \cite[Eq.~(40)]{duquesne2017decomposition}. Using this, one can define a regular conditional probability measure $\excm{a} = \n[\bullet |\sigma=a]$ such that $\excm{a}$-a.s. $\sigma =a$ and
\begin{equation*}
	\n[\bullet] = \frac{1}{\gamma\Gamma(1-1/\gamma)}\int_0^\infty \excm{a}[\bullet] \,\frac{\dd a}{a^{1+1/\gamma}}\cdot
\end{equation*}
Informally, $\excm{a}$ can be seen as the distribution of the stable tree $\rdtree$ with total mass $a$.

The next result is a restatement of \cite[Proposition 5.7]{goldschmidt2010behavior} in terms of trees which gives a version of the scaling property for the stable tree conditioned on its total mass. Recall from \eqref{eq: definition norma} the definition of $\norma$.
\begin{lemma}\label{lemma: sigma = 1 ou > 1}
	Let $\rdtree$ be the stable tree with branching mechanism $\psi(\lambda) = \lambda^\gamma$ where $\gamma \in (1,2]$. 
	\begin{enumerate}[label=(\roman*),leftmargin=*]
		\item For every measurable function $F \colon \T \to [0,\infty]$, we have
		\begin{equation*}
			\excm{1}\left[F(\rdtree)\right] =  \Gamma(1-1/\gamma)\n \left[\ind{\sigma >1}F(\norma(\rdtree))\right].
		\end{equation*}
	\item Under $\excm{a}$, the random tree $\rdtree$ is distributed as $R_\gamma(\rdtree,a^{1-1/\gamma})$ under $\excm{1}$ for every $a>0$.
	\end{enumerate}
\end{lemma}

We shall need Bismut's decomposition of the stable tree on several occasions. This is a decomposition of the tree along the ancestral line of a uniformly chosen leaf. We refer the reader to \cite[Theorem 4.5]{duquesne2005probabilistic} and \cite[Theorem 2.1]{abraham2013forest} for more details. Although in this paper we are only interested in the stable case $\psi(\lambda) = \lambda^\gamma$, we state the next two results in the general Lévy case. Let $\rdtree$ denote a Lévy tree under its excursion measure $\n$ associated with a branching mechanism 
\begin{equation}\label{eq: general branching}
	\psi(\lambda) = a\lambda + b\lambda^2 +\int_0^\infty (\e^{-\lambda r }-1+\lambda r)\, \pi(\dd r)
\end{equation}
	 where $a,b\geqslant 0$ and $\pi$ is a $\sigma$-finite measure on $(0,\infty)$ satisfying $\int_0^\infty (r\wedge r^2) \,\pi(\dd r)< \infty$. We further assume that $\int^\infty \dd \lambda/\psi(\lambda)<\infty$ so that the Lévy tree is compact. Notice that the Brownian case $\gamma =2$ corresponds to $a=0$, $b=1$ and $\pi = 0$ while the non-Brownian stable case $\gamma\in(1,2)$ corresponds to $a=b=0$ and
\begin{equation}\label{eq: Levy measure}
	\pi(\dd r) = \frac{\gamma(\gamma-1)}{\Gamma(2-\gamma)}  \,\frac{ \dd r}{r^{1+\gamma}}\cdot
\end{equation}
We will also need the probability measure $\for{r}$ on $\T$ which is the distribution of the Lévy tree starting from $r>0$ individuals. More precisely, take $\sum_{i\in I} \delta_{\mathcal{T}_i}$ a  Poisson point measure on $\T$ with intensity $r \n$ and define $\for{r}$ as the distribution of the random tree $\mathcal{T}$ obtained by gluing together the trees $\mathcal{T}_i$ at their root. See \cite[Section 2.6]{abraham2013forest} for further details.

Before stating the result, we first introduce some notations. Let $(T,\root,d,\mu)$ be a (class representative of a)compact real tree and let $x \in T$. Denote by $(x_i, \, i \in I_x)$ the branching points of $T$ which lie on the branch $\llbracket \root,x\rrbracket$, that is those points $y\in \llbracket \root,x\rrbracket$ such that $T\setminus \{y\}$ has at least three connected components. For every $i \in I_x$, define the tree grafted on the branch $\llbracket \root,x\rrbracket$ at $x_{i}$ by $T_{i} = \{y \in T\colon x\wedge y = x_{i}\}$. We consider $T_{i}$ as an element of $\T$ in the obvious way. Let $h_{i} = H(x_{i})$ and define a point measure on $[0,\infty) \times \T$ by
\[\M_x^T = \sum_{i \in I_x} \delta_{(h_{i},T_{i})}. \]

We can now state Bismut's decomposition, see \cite[Theorem 4.5]{duquesne2005probabilistic} or \cite[Theorem 2.1]{abraham2013forest}.
\begin{thm}\label{thm: Bismut}
	Let $\rdtree$ be the Lévy tree with a general branching mechanism \eqref{eq: general branching} satisfying the Grey condition $\int^\infty \dd \lambda/\psi(\lambda) <\infty$ under its excursion measure $\n$. For every $\lambda \geqslant 0$ and every nonnegative measurable function $\Phi$ on $[0,\infty)\times \T$, we have
	\begin{equation}
	\n\left[\int_\rdtree \mu(\dd x) \e^{-\lambda H(x) - \langle \M_x^\rdtree, \Phi \rangle}\right] = \int_0^\infty \dd t \,\e^{-(\lambda+a) t} \ex{\e^{-\sum_{0\leqslant s \leqslant t} \Phi(s,\tree_s)}},
	\end{equation}
	where $(\tree_s, \,  0\leqslant s \leqslant t)$ is a Poisson point process with intensity $\m[\dd \rdtree]= 2b\n[\dd \rdtree]+\int_0^\infty r\pi(\dd r) \for{r}(\dd \rdtree)$.
\end{thm}
\begin{remark}
	Bismut's decomposition states the following: let $\rdtree$ be the Lévy tree under its excursion measure $\n$ and, conditionally on $\rdtree$, let $U$ be a leaf chosen uniformly at random, \emph{i.e.} according to the distribution $\sigma^{-1}\mu$. Then, under $\n[\sigma \bullet]$, the random variable $H(U)$ has “distribution” $\e^{-at}\, \dd t$ on $(0,\infty)$ and, conditionally on $H(U) = t$, the point measure $\M_{U}^\rdtree$ is distributed as $\sum_{s\leqslant t} \delta_{(s,\tree_s)}$. One can make this claim rigorous by introducing the space of compact weighted rooted real trees with an additional marked vertex and considering the semidirect product measure $\n \times \sigma^{-1}\mu$ on it which corresponds to the distribution of the pair $(\rdtree,U)$. Under this measure, the distribution of the random pair $(H(U), \M_U^\rdtree)$ does not depend on the particular choice of representative in the class of $\rdtree$.
\end{remark}
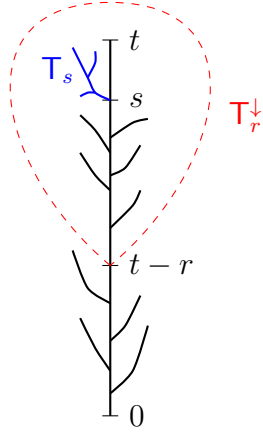
\begin{figure}[h]
	\centering
	\begin{tikzpicture}
		\draw[thick] (0,0) -- (0,2) -- (0,5);
		\draw (-3pt,0)  -- (3pt,0) node[right]{$0$};
		\draw (-3pt,2) -- (3pt,2) node[right]{$t-r$} ;
		\draw (-3pt,5) -- (3pt,5) node[right]{$t$} ;
		\draw (-3pt,4.2) -- (3pt,4.2) node[right]{$s$} ;
		\mybox[dashed,red]{dir={(0,2) to (0,5.5)}, left=50:90:2cm, right=50:90:2cm};
		\node[red] at (1.8,4) {$\Tdown_r$};
		\draw[thick] plot [smooth] coordinates {(0,2.5) (0.3,2.8) (0.4,3)};
		\draw[thick] plot [smooth] coordinates {(0,3) (-0.2,3.2) (-0.4,3.5)};
		\draw[thick] plot [smooth] coordinates {(0,3.2) (0.2,3.3) (0.4,3.6)};
		\draw[thick] plot [smooth] coordinates {(0,3.5) (-0.2,3.8) (-0.4,4)};
		\draw[thick] plot [smooth] coordinates {(0,3.7) (0.3,3.9) (0.5,3.95)};
		\draw[blue,thick] plot [smooth] coordinates {(0,4.2) (-0.2, 4.3) (-0.3,4.5) (-0.5,4.9)};
		\draw[blue,thick] plot [smooth] coordinates {(-0.3,4.5) (-0.2,4.7) (-0.2, 4.85)};
		\draw[blue,thick] plot [smooth] coordinates { (-0.2, 4.3) (-0.3,4.3) (-0.4,4.25)};
		\node[blue] at (-0.7,4.6) {$\tree_s$};
		\draw[thick] plot [smooth] coordinates {(0,1) (0.2,1.2) (0.4,1.6)};
		\draw[thick] plot [smooth] coordinates {(0,0.3) (0.3,0.6) (0.5,1.2)};
		\draw[thick] plot [smooth] coordinates {(0,0.6) (-0.2,0.9) (-0.4,1.3)};
		\draw[thick] plot [smooth] coordinates {(0,1.5) (-0.3,1.7) (-0.5,2.2)};
	\end{tikzpicture}
	\caption{The real tree $\Tdown_r$ obtained by grafting the atoms $\tree_s$ of a Poisson point process on a branch $[t-r,t]$ at height $s$.}
	\label{fig: bismut}
\end{figure}
Let $(\tree_s, 0\leqslant  s \leqslant t)$ be a Poisson point process as in Theorem \ref{thm: Bismut} and denote by
\begin{equation}\label{eq: definition Tdown}
	\Tdown_r \coloneqq [t-r,t] \circledast_{t-r\leqslant s\leqslant t} \left(\tree_s,s\right), \quad \forall 0\leqslant r \leqslant t
\end{equation}
the random real tree obtained by grafting $\tree_s$ on a branch $[t-r,t]$ at height $s$ for every $t-r \leqslant s \leqslant t$ and rooted at $t-r$, see Figure \ref{fig: bismut}. We refer the reader to \cite[Section 2.4]{abraham2013forest} for a precise definition of the grafting procedure. Let
\begin{equation}\label{eq: definition tau and eta}
	\mass_r \coloneqq \mu(\Tdown_r) = \sum_{t-r\leqslant s \leqslant t} \mu(\tree_s) \quad \text{and} \quad
	\height_r \coloneqq \H(\Tdown_r) = \max_{t-r\leqslant s \leqslant t}\left(\H(\tree_s) + s-(t-r)\right)
\end{equation} 
denote its mass and height. Finally, let 
\begin{equation}\label{eq: definition S}
	S_r \coloneqq \sum_{s\leqslant r} \mu(\tree_s).
\end{equation} 
 It is shown in the proof of \cite[Lemma 4.6]{delmas2018}, see Section 8.6 and more precisely (8.20) therein, that in the stable case $\psi(\lambda) =\lambda^\gamma$, both $\tau$ and $S$ are subordinators defined on $[0,t]$ with Laplace exponent 
\begin{equation}\label{eq: laplace exponent}
\phi(\lambda) = \gamma \lambda^{1-1/\gamma}.
\end{equation}
In particular, thanks to \cite[Section 4]{wolfe1975onmoments} or \cite[Eq.~(2.1.8)]{zolotarev1986one}, we have for every $p \in (-\infty, 1-1/\gamma)$, 
\begin{equation}\label{eq: tau finite moment}
	\ex{\mass_1^p} < \infty.
\end{equation}

We now give the following form of Bismut's decomposition which we will use throughout the paper. Denote by $\D$ the space of cadlag functions on $[0,\infty)$ endowed with the Skorokhod $J1$ topology. For every measurable function $F \colon [0,\infty)^3 \times \T\times \D^2\to [0,\infty]$, we have 
\begin{multline}\label{eq: Bismut}
	\n\left[\int_\rdtree \mu(\dd x)F\left(H(x), \sigma,\H, \rdtree, \left(\sigma_{H(x)-r,x},\, 0\leqslant r \leqslant H(x)\right),\left(\H_{H(x)-r,x},\, 0\leqslant r \leqslant H(x)\right)\right)\right]\\
	\begin{aligned}[b]
	&= \int_0^\infty \dd t \ex{F\left(t,\mass_t,\height_t,\Tdown_t,\left(\mass_r, \, 0\leqslant r \leqslant t\right),\left(\height_r, \, 0\leqslant r \leqslant t\right)\right)}.
	\end{aligned}
\end{multline}
Notice that by definition $\mass_t = S_t$ and $S_{r-}=\mass_t - \mass_{t-r}$ for every $r\in[0,t]$. This will be used implicitly in the sequel. In particular, the following computation will be useful
\begin{equation}\label{eq: tau is integrable}
	\int_0^\infty \ex{\frac{1}{S_t}\ind{S_t >1}}\, \dd t= \int_0^\infty \ex{\frac{1}{\mass_t}\ind{\mass_t >1}} \, \dd t =\n\left[\sigma >1\right]  = \frac{1}{\Gamma(1-1/\gamma)},
\end{equation}
where in the last equality we used Lemma \ref{lemma: sigma = 1 ou > 1}-(i) with $F\equiv 1$.

As a first application of Bismut's decomposition, we give a decomposition of the stable tree into $n+1$ subtrees which generalizes \cite[Lemma 6.1]{abraham2020global}.
\begin{thm}\label{thm: decomposition n+1}
	Let $\rdtree$ be the Lévy tree with a general branching mechanism \eqref{eq: general branching} under its excursion measure $\n$. Then for every $n\geqslant 1$ and every nonnegative measurable functions $f_i, \ 1\leqslant i \leqslant n+1$ defined on $[0,\infty)\times\T$, we have with $r_0 = 0$ and $r_{n+1} = H(x)$
	\begin{equation}
		\n\left[\int_\rdtree \mu(\dd x) \int_{0< r_1< \ldots < r_n < H(x)} \prod_{i=1}^{n+1} f_i \left(r_i-r_{i-1},\treeint{r_{i-1}}{r_{i}}{x}\right) \,\prod_{i=1}^n \dd r_i\right]=\prod_{i=1}^{n+1} \n\left[\int_\rdtree \mu(\dd x) f_i(H(x),\rdtree)\right].
	\end{equation}
\end{thm}
\begin{proof}
	Recall from \eqref{eq: definition Tdown} the definition of $\Tdown$. By Theorem \ref{thm: Bismut}, we have
	\begin{multline*}
		\n\left[\int_\rdtree \mu(\dd x) \int_{0< r_1< \ldots < r_n < H(x)} \prod_{i=1}^{n+1} f_i \left(r_i-r_{i-1},\treeint{r_{i-1}}{r_{i}}{x}\right) \, \prod_{i=1}^n \dd r_i\right]\\
		=\int_0^\infty \dd r_{n+1}\,\e^{-ar_{n+1}}\ex{\int_{0< r_1< \ldots < r_n < r_{n+1}}\prod_{i=1}^{n+1} f_i\left(r_i-r_{i-1},\Tint{r_{i-1}}{r_{i}}\right) \, \prod_{i=1}^n \dd r_i},
	\end{multline*}
where we set $\Tint{r}{r'} = (\Tdown_{t-r} \setminus \Tdown_{t-r'})\cup \{t-r'\}$ for every $0< r < r' < t$. Since $(\tree_s, \, 0\leqslant s \leqslant t)$ is a Poisson point process, we get that the $\Tint{r_{i-1}}{r_{i}}$ are independent and distributed as $\Tint{0}{r_{i}-r_{i-1}}$. We deduce that
\begin{multline*}
		\n\left[\int_\rdtree \mu(\dd x) \int_{0< r_1< \ldots < r_n < H(x)} \prod_{i=1}^{n+1} f_i \left(r_i-r_{i-1},\treeint{r_{i-1}}{r_{i}}{x}\right) \, \prod_{i=1}^n \dd r_i\right]\\
		\begin{aligned}
			&= \int_{0< r_1< \ldots < r_n < r_{n+1}} \prod_{i=1}^{n+1} \e^{-a (r_i - r_{i-1})} \ex{f_i\left(r_i-r_{i-1},\Tint{0}{r_i-r_{i-1}}\right)}\, \dd r_i \\
			&=\int_{[0,\infty)^{n+1}} \prod_{i=1}^{n+1}\e^{-a s_i} \ex{f_i\left(s_i,\Tint{0}{s_i}\right)}\, \dd s_i\\
			&=\prod_{i=1}^{n+1} \n\left[\int_\rdtree \mu(\dd x) f_i(H(x),\rdtree)\right],
		\end{aligned}
\end{multline*}
where we made the change of variables $(s_1,s_2,\ldots,s_{n+1}) = (r_1,r_2-r_1,\ldots, r_{n+1}-r_n)$ for the second equality and used Bismut's decomposition \eqref{eq: Bismut} together with the fact that $\Tint{0}{t} = \Tdown_t$ $\mathbb{P}$-a.s. for the last.
\end{proof}

From now on, we restrict ourselves to the stable case $\psi(\lambda) = \lambda^\gamma$ where $\gamma \in (1,2]$. For functions $f,g$ defined on $(0,\infty)$, we denote by $f\ast g$ their convolution defined by
\begin{equation*}
	f\ast g (t) = \int_0^t f(s) g(t-s) \, \dd s, \quad \forall t >0.
\end{equation*}
\begin{proposition}\label{prop: n+1 stable}
	Let $\rdtree$ be the stable tree with branching mechanism $\psi(\lambda) = \lambda^{\gamma}$ where $\gamma \in (1,2]$. For every $n \geqslant 1$ and every nonnegative measurable functions $f_i, \ 1\leqslant i \leqslant n+1$ defined on $[0,\infty)\times\T$, we have with $r_0 = 0$ and $r_{n+1} = H(x)$
	\begin{multline}\label{eq: n+1 general}
		\excm{1}\left[\int_\rdtree \mu(\dd x) \int_{0< r_1< \ldots < r_n < H(x)}\prod_{i=1}^{n+1} f_i\left(r_i-r_{i-1},\treeint{r_{i-1}}{r_{i}}{x}\right)\, \prod_{i=1}^n \dd r_i\right] \\
		= \frac{1}{\gamma^n \Gamma(1-1/\gamma)^n} F_1\ast \cdots \ast F_{n+1}(1),
	\end{multline}
where $R_\gamma$ is defined in \eqref{eq: definition R} and
\begin{equation*}
	F_i(a) = a^{-1/\gamma} \excm{1}\left[\int_\rdtree \mu(\dd x) f_i \left(a^{1-1/\gamma} H(x), R_\gamma\left(\rdtree,a^{1-1/\gamma}\right)\right)\right], \quad \forall a >0.
\end{equation*}
In particular, for every $n \geqslant 1$ and every nonnegative measurable functions $g_i, \ 1\leqslant i \leqslant n+1$ defined on $[0,\infty)\times[0,1]$, we have
\begin{multline}\label{eq: n+1 mass}
		\excm{1}\left[\int_\rdtree \mu(\dd x) \int_{0< r_1< \ldots < r_n < H(x)}\prod_{i=1}^{n+1} g_i\left(r_i-r_{i-1},\sigma_{r_{i-1},x}-\sigma_{r_i,x}\right)\, \prod_{i=1}^n \dd r_i\right] \\
		= \frac{1}{\gamma^n \Gamma(1-1/\gamma)^n} G_1\ast \cdots \ast G_{n+1}(1),
\end{multline}
where
\begin{equation*}
	G_i(a) = a^{-1/\gamma} \excm{1}\left[\int_\rdtree \mu(\dd x) g_i\left(a^{1-1/\gamma}H(x),a\right)\right], \quad \forall a >0.
\end{equation*}
\end{proposition}
\begin{proof}
	Let $f_i \colon [0,\infty)\times\T \to \real$ be continuous and bounded for $1\leqslant i\leqslant n+1$. By Theorem \ref{thm: decomposition n+1}, we have for $\lambda >0$
	\begin{multline}\label{eq: LT equality}
			\prod_{i=1}^{n+1} \n\left[\e^{-\lambda \sigma}\int_\rdtree \mu(\dd x) f_i(H(x),\rdtree)\right]\\
			\begin{aligned}[b]
			&=\n\left[\int_\rdtree \mu( \dd x) \int_{0< r_1< \ldots < r_n < H(x)} \prod_{i=1}^{n+1} \e^{-\lambda\mu\left(\treeint{r_{i-1}}{r_{i}}{x}\right)}f_i \left(r_i - r_{i-1},\treeint{r_{i-1}}{r_{i}}{x}\right) \,\prod_{i=1}^n \dd r_i\right]\\
			&= \n\left[\e^{-\lambda \sigma}\int_\rdtree \mu(\dd x) \int_{0< r_1< \ldots < r_n < H(x)} \prod_{i=1}^{n+1} f_i \left(r_i - r_{i-1},\treeint{r_{i-1}}{r_{i}}{x}\right) \,\prod_{i=1}^n \dd r_i\right].
			\end{aligned}
	\end{multline}

Disintegrating with respect to $\sigma$ and using the scaling property from Lemma \ref{lemma: sigma = 1 ou > 1}-(ii), we have
\begin{align}\label{eq: step 1}
	\n\left[\e^{-\lambda \sigma}\int_\rdtree \mu(\dd x) f_i(H(x),\rdtree)\right] &= \frac{1}{\gamma\Gamma(1-1/\gamma)} \int_0^\infty  \e^{-\lambda a} \excm{a} \left[\int_\rdtree \mu(\dd x)f_i(H(x),\rdtree)\right]\, \frac{\dd a}{a^{1+1/\gamma}} \nonumber\\
	&= \frac{1}{\gamma\Gamma(1-1/\gamma)} \mathcal{L}F_i(\lambda),
\end{align}
where $\mathcal{L}$ denotes the Laplace transform on $[0,\infty)$.

On the other hand, again disintegrating with respect to $\sigma$, we have
\begin{multline}\label{eq: step 2}
	\gamma \Gamma(1-1/\gamma)\n\left[\e^{-\lambda \sigma}\int_\rdtree \mu(\dd x) \int_{0< r_1< \ldots < r_n < H(x)} \prod_{i=1}^{n+1} f_i \left(r_i-r_{i-1},\treeint{r_{i-1}}{r_{i}}{x}\right) \,\prod_{i=1}^n \dd r_i\right] \\
	\begin{aligned}[b]
		&=\int_0^\infty \, \frac{\dd a}{a^{1+1/\gamma}}\e^{-\lambda a}\excm{a}\left[\int_\rdtree \mu(\dd x) \int_{0< r_1< \ldots < r_n < H(x)} \prod_{i=1}^{n+1} f_i \left(r_i-r_{i-1},\treeint{r_{i-1}}{r_{i}}{x}\right) \,\prod_{i=1}^n \dd r_i\right]\\
		&=\int_0^\infty \dd a \, a^{(n+1)(1-1/\gamma)-1}\e^{-\lambda a} F(a),
	\end{aligned}
\end{multline}
where we set
\begin{equation*}
	F(a) = \excm{1}\left[\int_\rdtree \mu(\dd x) \int_{0< r_1< \ldots < r_n < H(x)} \prod_{i=1}^{n+1} f_i\left(a^{1-1/\gamma}(r_i- r_{i-1}),R_\gamma\left(\rdtree_{[r_{i-1},r_i),x},a^{1-1/\gamma}\right)\right) \,\prod_{i=1}^n \dd r_i\right].
\end{equation*}
%Now notice that
%\begin{multline*}
%	\excm{1}\left[\int_\rdtree \mu(\dd x) \int_{0< r_1< \ldots < r_n < a^{1-1/\gamma}H(x)} \prod_{i=1}^{n+1} f_i \left(R_\gamma(\rdtree,a^{1-1/\gamma})_{[r_{i-1},r_i),x}\right) \,\prod_{i=1}^n \dd r_i\right] \\
%	\begin{aligned}
%		&= a^{n(1-1/\gamma)}\excm{1}\left[\int_\rdtree \mu(\dd x) \int_{0< r_1< \ldots < r_n < H(x)} \prod_{i=1}^{n+1} f_i \left(R_\gamma(\rdtree,a^{1-1/\gamma})_{[a^{1-1/\gamma}r_{i-1},a^{1-1/\gamma}r_i),x}\right) \,\prod_{i=1}^n \dd r_i\right]\\
%		&=a^{n(1-1/\gamma)}\excm{1}\left[\int_\rdtree \mu(\dd x) \int_{0< r_1< \ldots < r_n < H(x)} \prod_{i=1}^{n+1} f_i \circ R_\gamma\left(\rdtree_{[r_{i-1},r_i),x},a^{1-1/\gamma}\right) \,\prod_{i=1}^n \dd r_i\right].
%	\end{aligned}
%\end{multline*} 

Putting together \cref{eq: LT equality,eq: step 1,eq: step 2} yields
\begin{align*}
	\frac{1}{\gamma^n\Gamma(1-1/\gamma)^n}\mathcal{L}(F_1\ast \ldots \ast F_{n+1})(\lambda) &= \frac{1}{\gamma^n \Gamma(1-1/\gamma)^n}\prod_{i=1}^{n+1}\mathcal{L}F_i(\lambda)\\
	&= \int_0^\infty \dd a \, a^{(n+1)(1-1/\gamma)-1}\e^{-\lambda a} F(a).
\end{align*}
Since this holds for every $\lambda >0$, we deduce that $\dd a$-a.e. on $(0,\infty)$,
\begin{equation}\label{eq: ae equality}
	\frac{1}{\gamma^n\Gamma(1-1/\gamma)^n}F_1\ast \ldots \ast F_{n+1}(a)
	= a^{(n+1)(1-1/\gamma)-1}F(a).
\end{equation}

Thanks to Lemma \ref{lemma: ghp dilatation}, the mapping $a \mapsto R_\gamma(T,a^{1-1/\gamma})$ is continuous on $(0,\infty)$ for every $T\in \T$. We deduce from the dominated convergence theorem that the $F_i$ are continuous on $(0,\infty)$ and thus $F_1\ast \ldots \ast F_{n+1}$ too. Similarly, the right-hand side of \eqref{eq: ae equality} is continuous with respect to $a$. Therefore the equality holds for every $a \in(0,\infty)$.  In particular, taking $a=1$ proves \eqref{eq: n+1 general} for continuous bounded functions $f_i \colon [0,\infty)\times\T \to \real$. This extends to measurable functions $f_i \colon [0,\infty)\times\T \to \real$ thanks to the monotone class theorem. Finally, \eqref{eq: n+1 mass} is a direct consequence of \eqref{eq: n+1 general}.
\end{proof}

In particular, the following corollary will be useful.
\begin{corollary}\label{lemma: second moment}
We have
\begin{equation}\label{eq: bounded in L2}
	\sup_{\alpha \geqslant 0} \alpha^{2-2/\gamma} \excm{1}\left[\int_\rdtree \mu(\dd x)\left(\int_0^{H(x)}\sigma_{r,x}^\alpha \, \dd r\right)^2\right] < \infty.
\end{equation}
\end{corollary}
\begin{proof}
		Applying \eqref{eq: n+1 mass} with $n=2$, $g_1(r,a) = g(1-a)$, $g_2 (a)= 1$ and $g_3(r,a)= g(a)$ yields, for every measurable function $g \colon [0,1] \to [0,\infty]$,
		\begin{multline}\label{eq: second moment}
			\excm{1}\left[\int_\rdtree \mu(\dd x)\left(\int_0^{H(x)} g(\sigma_{r,x})\, \dd r\right)^2 \right] \\
			= \frac{2}{\gamma^2\Gamma(1-1/\gamma)^2}\int_0^1 g(y)(1-y)^{-1/\gamma} \, \dd y \int_0^y g(z) z^{-1/\gamma}(y-z)^{-1/\gamma}\, \dd z.
		\end{multline}
		Taking $g(a) = a^\alpha$, we get
	\begin{multline*}
		\alpha^{2-2/\gamma}\excm{1}\left[\int_\rdtree \mu(\dd x) \left(\int_0^{H(x)} \sigma_{r,x}^\alpha\, \dd r\right)^2\right]\\
		\begin{aligned}[b]
			&= \frac{2\alpha^{2-2/\gamma}}{\gamma^2\Gamma(1-1/\gamma)^2} \int_0^1 y^\alpha (1-y)^{-1/\gamma}\, \dd y \int_0^y z^{\alpha-1/\gamma} (y-z)^{-1/\gamma}\, \dd z\\
			&= \frac{2\alpha^{2-2/\gamma}}{\gamma^2 \Gamma(1-1/\gamma)^2}\mathrm{B}\left(2\alpha+2-2/\gamma,1-1/\gamma\right)\mathrm{B}\left(\alpha+1-1/\gamma,1-1/\gamma\right),
		\end{aligned}
	\end{multline*}
	where $\mathrm{B}$ is the Beta function. Using that $\mathrm{B}(x,1-1/\gamma)\sim \Gamma(1-1/\gamma)x^{-1+1/\gamma}$ as $x\to \infty$, \eqref{eq: bounded in L2} readily follows.
\end{proof}

As a consequence of Proposition \ref{prop: n+1 stable}, we are able to compute the intensity measure of the random measure $\Psi_\rdtree$ appearing in \cite{abraham2020global}, see Proposition 6.3 therein.
\begin{corollary}\label{cor: first moment Psi}
	Let $\rdtree$ be the normalized stable tree with branching mechanism $\psi(\lambda) = \lambda^\gamma$ where $\gamma \in (1,2]$. Let $f$ and $g$ be nonnegative measurable functions defined on $\T$ and $[0,\infty)$ respectively. We have
\begin{multline}\label{eq: first moment Psi}
	 \gamma\Gamma(1-1/\gamma)\excm{1}\left[\int_\rdtree \mu(\dd x)\int_0^{H(x)} f(\rdtree_{r,x})g(r)\, \dd r\right]\\
	 =\int_0^1 \dd a \, a^{-1/\gamma} (1-a)^{-1/\gamma}\excm{1}\left[f\circ R_\gamma \left(\rdtree, a^{1-1/\gamma}\right)\right]
	 \excm{1}\left[\int_\rdtree \mu(\dd x) g\left((1-a)^{1-1/\gamma} H(x)\right)\right].
\end{multline}
\end{corollary}

Another application of Theorem \ref{thm: Bismut} is the following result giving the moments of the height $H(U)$ of a uniformly distributed leaf $U\in \rdtree$ (\emph{i.e.} according to $\mu$) under $\excm{1}$. In particular, this allows to give a nontrivial upper bound for the size of the ball with radius $\epsilon >0$ centered around the root of the normalized stable tree. Let us mention that this result is not new since the distribution of $H(U)$ under $\excm{1}$ is known: in the Brownian case $\gamma =2$, $H$ is distributed as $\sqrt{2}e$ where $e$ is the Brownian excursion so $\sqrt{2}H(U)$ has Rayleigh distribution; in the case $\gamma \in (1,2)$, 
$H(U)$ is distributed as a multiple of the local time at $0$ of the Bessel bridge of dimension $2/\gamma$, see \cite[Corollary 10]{haas2009spinal}.
\begin{lemma}\label{lemma: height integrable}
	Let $\rdtree$ be the normalized stable tree with branching mechanism $\psi(\lambda) = \lambda^\gamma$ where $\gamma \in (1,2]$. For every $p\in(-\infty,2)$, we have
	\begin{equation}\label{eq: height integrable}
		\excm{1}\left[\int_\rdtree H(x)^{-p}\, \mu(\dd x)\right] = \frac{(\gamma-1)\gamma^{p-1}\Gamma(1-1/\gamma)\Gamma(2-p)}{\Gamma(1-(p-1)(1-1/\gamma))}< \infty.
	\end{equation}
\end{lemma}
\begin{proof}
	Using Bismut's decomposition \eqref{eq: Bismut}, we have for every $\lambda > 0$
	\begin{equation*}
		\n\left[\sigma\e^{-\lambda \sigma} \int_\rdtree H(x)^{-p}\, \mu(\dd x)\right] = \int_0^\infty t^{-p}\ex{\mass_t\e^{-\lambda \mass_t}}\, \dd t = \phi'(\lambda)\int_0^\infty t^{1-p}\e^{-t\phi(\lambda)}\, \dd t.
	\end{equation*}
	
	On the other hand, disintegrating with respect to $\sigma$ and using Lemma \ref{lemma: sigma = 1 ou > 1}-(ii), we have
	\begin{align*}
		\n\left[\sigma\e^{-\lambda \sigma} \int_\rdtree H(x)^{-p}\, \mu(\dd x)\right] 
		&= \frac{1}{\gamma \Gamma(1-1/\gamma)} \int_0^\infty a\e^{-\lambda a} \excm{a}\left[\int_\rdtree H(x)^{-p}\, \mu(\dd x)\right]\, \frac{\dd a}{a^{1+1/\gamma}}\\
		&= \frac{1}{\gamma \Gamma(1-1/\gamma)} \int_0^\infty \e^{-\lambda a} \, \frac{\dd a}{a^{(p-1)(1-1/\gamma)}} \excm{1}\left[\int_\rdtree H(x)^{-p}\, \mu(\dd x)\right] \\
		&= \frac{\Gamma(1-(p-1)(1-1/\gamma))}{\gamma \Gamma(1-1/\gamma)\lambda^{1-(p-1)(1-1/\gamma)}} \excm{1}\left[\int_\rdtree H(x)^{-p} \, \mu(\dd x)\right].
	\end{align*}
	Using \eqref{eq: laplace exponent}, it follows that
	\begin{align*}
		\excm{1}\left[\int_\rdtree H(x)^{-p} \, \mu(\dd x)\right] &= \frac{\gamma \Gamma(1-1/\gamma)\lambda^{1-(p-1)(1-1/\gamma)}\phi'(\lambda)}{\Gamma(1-(p-1)(1-1/\gamma))} \int_0^\infty t^{1-p} \e^{-t\phi(\lambda)}\, \dd t \\
		&=\frac{(\gamma-1)\gamma^{p-1}\Gamma(1-1/\gamma)\Gamma(2-p)}{\Gamma(1-(p-1)(1-1/\gamma))}\cdot
	\end{align*}
\end{proof}
\begin{remark}
	Conditionally on $\rdtree$, let $U\in \rdtree$ be a uniformly distributed leaf. Then we can rewrite \eqref{eq: height integrable} as follows:
	\begin{equation}
		\frac{1}{c_\gamma}\excm{1}\left[\frac{1}{H(U)}\left(\gamma H(U)\right)^p\right] = \frac{\Gamma(p+1)}{\Gamma(p(1-1/\gamma)+1)}, \quad \forall p > -1,
	\end{equation}
	where $c_\gamma = (\gamma-1)\Gamma(1-1/\gamma)$.
	This implies that, under the probability measure $c_\gamma^{-1}\excm{1}[H(U)^{-1}\bullet]$, the random variable $\gamma H(U)$ has Mittag-Leffler distribution with index $1-1/\gamma$, see \cite[Eq.~(0.42)]{pitman2006combinatorial}.
\end{remark}

\section{Zooming in at the root of the stable tree}\label{section: zooming}
In this section, we study the shape of the stable tree in a small neighborhood of its root. The main result, Theorem \ref{thm: zooming}, states that after zooming in and rescaling, one sees a branch on which trees are grafted according to a Poisson point process on $\T$ with intensity $\m$ given by
\begin{equation}\label{eq: intensity measure}
	\m[\dd \rdtree] = \begin{cases}
		2 \n[\dd \rdtree] & \text{if } \gamma =2,\\
		\int_0^\infty r \pi(\dd r) \for{r}(\dd \rdtree) & \text{if } \gamma \in (1,2),
	\end{cases}
\end{equation}
where we recall from Section \ref{section: stable tree} that $\pi$ is given by \eqref{eq: Levy measure} and $\for{r}$ is the distribution of the random tree $\rdtree$ obtained by gluing together at their roots a family of trees distributed according to a Poisson point measure with intensity $r\n$. 

We start with the following result giving the scaling property of the stable tree under $\m$.
\begin{lemma}\label{lemma: scaling property under M}
	The following identity holds for every $a>0$
	\begin{equation}\label{eq: scaling property under M}
		R_\gamma(\rdtree,a) \quad \text{under} \ \m \quad \lawd \quad \rdtree \quad \text{under} \ a\m.
	\end{equation}
\end{lemma}

\begin{proof}
	 The case $\gamma =2$ reduces to the scaling property \eqref{eq: scaling property under N} so we only need to prove the case $\gamma \in (1,2)$. Thanks to \eqref{eq: scaling property under N}, we deduce that $R_\gamma(\rdtree,a)$ under $\for{r}$ has distribution $\for{a^{1/(\gamma-1)}r}$. It follows from \eqref{eq: Levy measure} that under $\m$, $R_\gamma(\rdtree,a)$ has distribution
\begin{equation*}
	\int_0^\infty r \pi(\dd r) \for{a^{1/(\gamma-1)}r}(\dd \rdtree) = a\int_0^\infty s \pi(\dd s) \for{s}(\dd \rdtree) = a\m[\dd \rdtree].
\end{equation*}
\end{proof}

Let $(T,\root,d,\mu)$ be a compact real tree and let $x \in T$. Recall from Section \ref{section: stable tree} that $T_i, \, i \in I_x$ are the trees grafted on the branch $\llbracket \root, x\rrbracket$, each one at height $h_i$. Fix $\f\colon (0,\infty)\to (0,\infty)$ and define for every $\epsilon >0$ a point measure on $[0,\infty)^2\times \T$ by
\begin{equation}
	\mathcal{N}_\epsilon^\f(x) = \sum_{h_i \leqslant \f(\epsilon) H(x)} \delta_{\left(\epsilon^{-1} h_i,\epsilon^{-\gamma/(\gamma-1)}\sigma_i, \norma(T_i)\right)}.
\end{equation}
We are now in a position to give the main result of this section. 
\begin{thm}\label{thm: zooming} Let $\rdtree$ be the normalized stable tree with branching mechanism $\psi(\lambda) = \lambda^\gamma$ where $\gamma \in (1,2]$. Conditionally on $\rdtree$, let $U$ be a $\rdtree$-valued random variable with distribution $\mu$ under $\excm{1}$. Let $(\tree_s', \, s\geqslant 0)$ be a Poisson point process with intensity $\m$, independent of $(\rdtree,H(U))$. Let $\Phi \colon [0,\infty)^2\times \T\to [0,\infty)$ be a measurable function such that there exists $C>0$ such that for every $h\geqslant 0$ and $T\in \T$, we have
\begin{equation}\label{eq: hypothesis Phi}
	\left|\Phi(h,b,T)-\Phi(h,a,T)\right|\leqslant C|b-a|.
\end{equation}
\begin{enumerate}[label=(\roman*),leftmargin=*]
	\item If $\lim_{\epsilon \to 0}\epsilon^{-1/2}\f(\epsilon) = 0$ and $\lim_{\epsilon \to 0}\epsilon^{-1}\f(\epsilon)=\infty$, then we have the following convergence in distribution
	\begin{equation}
		\left(\rdtree, H(U), \langle\mathcal{N}_\epsilon^\f(U),\Phi\rangle\right) \xrightarrow[\epsilon \to 0]{(d)} \left(\rdtree, H(U), \sum_{s\geqslant 0} \Phi\left(s,\mu(\tree_s'),\norma\left(\tree_s'\right)\right)\right),
	\end{equation}
	in the space $\T\times [0,\infty)\times [0,\infty]$.
\item If $\f(\epsilon) = \epsilon$, then we have the following convergence in distribution
\begin{equation}
	\left(\rdtree, H(U), \langle\mathcal{N}_\epsilon^\f(U),\Phi\rangle\right) \xrightarrow[\epsilon \to 0]{(d)} \left(\rdtree, H(U), \sum_{s\leqslant H(U)} \Phi\left(s,\mu(\tree_s'),\norma\left(\tree_s'\right)\right)\right)
\end{equation}
in the space $\T\times [0,\infty)\times [0,\infty]$.
\end{enumerate}
\end{thm}

\begin{proof}
	 We only prove (i), the proof of (ii) being similar. Let $f \colon \T \to \real$ and $g \colon [0,\infty) \to \real$ be Lipschitz-continuous and bounded and assume that $\Phi \colon [0,\infty)^2 \times\T \to [0,\infty)$ is measurable and satisfies \eqref{eq: hypothesis Phi}. We shall consider the following modification of the measure $\mathcal{N}_\epsilon^\f(U)$:
	 \begin{equation*}
	 	\widehat{\mathcal{N}}_\epsilon^\f(U)\coloneqq\sum_{h_i \leqslant \f(\epsilon) H(U)} \delta_{\left(\epsilon^{-1} h_i/H(U),\epsilon^{-\gamma/(\gamma-1)}\sigma_i, \norma(T_i)\right)}.
	 \end{equation*}

	\noindent
	\textbf{Step 1.} Set
	\begin{align*}
		F(\epsilon) &\coloneqq\excm{1}\left[f(\rdtree)g(H(U))\expp{-\left\langle \widehat{\mathcal{N}}_\epsilon^\f(U), \Phi\right\rangle}\right]\nonumber \\
		&=\excm{1}\left[\int_\rdtree \mu(\dd x) f(\rdtree)g(H(x))\expp{-\sum_{h_i \leqslant \f(\epsilon)H(x)} \Phi\left(\epsilon^{-1}h_i/H(x),\epsilon^{-\gamma/(\gamma-1)}\sigma_i, \norma(\rdtree_i)\right)}\right].
	\end{align*}
	Using Lemma \ref{lemma: sigma = 1 ou > 1}-(i) and Theorem \ref{thm: Bismut}, we have
	\begin{align}\label{eq: expression F}
		\Gamma(1-1/\gamma)^{-1}F(\epsilon)&=
		\begin{multlined}[t]
			\n\left[\frac{1}{\sigma}\ind{\sigma>1} \int_\rdtree \mu(\dd x) f\circ \norma\left(\rdtree\right) g\left(\sigma^{-1+1/\gamma}H(x)\right)\right. \\
			\left.\times\expp{-\sum_{h_i \leqslant \f(\epsilon) H(x)} \Phi\left(\epsilon^{-1}h_i/H(x),\epsilon^{-\gamma/(\gamma-1)}\sigma^{-1}\sigma_i,\norma\left(\rdtree_i\right)\right)}\right]
		\end{multlined}\nonumber\\
		&= 
		\begin{multlined}[t]
			\int_0^\infty \dd t \operatorname{\mathbb{E}}\left[\frac{1}{\mass_t}\ind{\mass_t >1} f\circ \norma\left(\Tdown_t\right) g\left(\mass_t^{-1+1/\gamma} t\right)\right.\\
			\left. \times\expp{-\sum_{s\leqslant \f(\epsilon) t} \Phi\left(\epsilon^{-1} s/t,\epsilon^{-\gamma/(\gamma-1)}\mass_t^{-1}\mu(\tree_s),\norma\left(\tree_s\right)\right)}\right].
		\end{multlined}
	\end{align}

	\noindent
	\textbf{Step 2.} The proof of the following lemma is postponed to Section \ref{section: approx point measure}. To simplify notation, we introduce $\g(\epsilon) = 1-\f(\epsilon)$.
	\begin{lemma}\label{lemma: approx point measure}
		Assume that $\lim_{\epsilon \to 0} \epsilon^{-1/2} \f(\epsilon) = 0$. Let $f \colon \T \to \real$ and $g \colon [0,\infty) \to \real$ be Lipschitz-continuous and bounded and assume that $\Phi \colon [0,\infty)^2 \times\T \to [0,\infty)$ is measurable and satisfies \eqref{eq: hypothesis Phi}. We have
		\begin{multline*}
			\lim_{\epsilon \to 0}\Gamma(1-1/\gamma)^{-1}F(\epsilon) - \int_0^\infty \dd t \operatorname{\mathbb{E}}\left[\frac{1}{\mass_{\g(\epsilon) t}}\ind{\mass_{\g(\epsilon)t} >1} f\circ\norma\left(\Tdown_{\g(\epsilon)t}\right) g\left(\mass_{\g(\epsilon)t}^{-1+1/\gamma} t\right)\right.\\
			\left. \times\expp{-\sum_{s\leqslant \f(\epsilon) t} \Phi\left(\epsilon^{-1} s/t,\epsilon^{-\gamma/(\gamma-1)}\mass_{\g(\epsilon)t}^{-1}\mu(\tree_s),\norma\left(\tree_s\right)\right)}\right]=0.
		\end{multline*}
	\end{lemma}
	Since $(\tree_s, \, 0\leqslant s \leqslant t)$ is a Poisson point process, it follows from the definition of $\Tdown_{\g(\epsilon)t}$ that $(\tree_s, \, 0\leqslant s \leqslant \f(\epsilon) t)$ is independent of $\Tdown_{\g(\epsilon)t}$. Thus, denoting by $(\tree_s', \, s\geqslant 0)$ a Poisson point process with intensity $\m$ which is independent of $\Tdown_{\g(\epsilon)t}$, recalling that $\mass_{\g(\epsilon)t}$ is a measurable function of $\Tdown_{\g(\epsilon)t}$ and making the change of variable $u=\g(\epsilon)t$, we have 
	\begin{equation}\label{eq: approx point measure}\lim_{\epsilon\to 0}\left|\Gamma(1-1/\gamma)^{-1}F(\epsilon) - \g(\epsilon)^{-1}\int_0^\infty \dd u\ex{Y_\epsilon(u)}\right|=0,
	\end{equation} where
	\begin{multline}
		Y_\epsilon(u)= \frac{1}{\mass_{u}}\ind{\mass_{u} >1} f\circ \norma\left(\Tdown_{u}\right) g\left(\g(\epsilon)^{-1}\mass_{u}^{-1+1/\gamma} u\right)\\
		\times \condex{\expp{-\sum_{s\leqslant \f(\epsilon)\g(\epsilon)^{-1} u} \Phi\left(\epsilon^{-1}\g(\epsilon) s/u, \epsilon^{-\gamma/(\gamma-1)}\mass_{u}^{-1}\mu(\tree_s'),\norma\left(\tree_s' \right)\right)}}{\Tdown_{u}}.
	\end{multline}
	
	\noindent
	\textbf{Step 3.} For fixed $\lambda >0$, we have
	\begin{multline}
		\ex{\expp{-\sum_{s\leqslant \f(\epsilon)\g(\epsilon)^{-1} u} \Phi\left(\epsilon^{-1}\g(\epsilon) s/u,\epsilon^{-\gamma/(\gamma-1)}\lambda^{-1}\mu(\tree_s'),\norma\left(\tree_s'\right)\right)}}\\
		\begin{aligned}
			&= \expp{-\int_0^{\f(\epsilon)\g(\epsilon)^{-1} u}  \dd s \m\left[1-\e^{-\Phi\left(\epsilon^{-1}\g(\epsilon)s/u,\epsilon^{-\gamma/(\gamma-1)}\lambda^{-1}\sigma,\norma(\rdtree)\right)}\right]} \nonumber\\
			&=\expp{-\g(\epsilon)^{-1}\int_0^{\epsilon^{-1}\f(\epsilon)\lambda^{-1+1/\gamma}u}\dd r \m\left[1-\e^{-\Phi\left(\lambda^{1-1/\gamma}r/u,\sigma,\norma(\rdtree)\right)}\right]},
		\end{aligned}
	\end{multline}
	where we made the change of variable $r = \epsilon^{-1}\g(\epsilon)\lambda^{-1+1/\gamma}s$ and used Lemma \ref{lemma: scaling property under M} with $a=\epsilon\lambda^{1-1/\gamma}$. (Notice that $\norma(\rdtree)$ has the same distribution under $a\m$ for every $a>0$). Thus, we deduce that a.s. for every $u>0$
	\begin{multline}
		\lim_{\epsilon\to 0}\condex{\expp{-\sum_{s\leqslant \f(\epsilon)\g(\epsilon)^{-1} u} \Phi\left(\epsilon^{-1}\g(\epsilon) s/u,\epsilon^{-\gamma/(\gamma-1)}\mass_{u}^{-1}\mu(\tree_s'),\norma\left(\tree_s'\right)\right)}}{\Tdown_{u}} \\
		\begin{aligned}
			&= \lim_{\epsilon \to 0} \expp{-\g(\epsilon)^{-1}\int_0^{\epsilon^{-1}\f(\epsilon) \lambda^{-1+1/\gamma}u}  \dd r \m\left[1-\e^{-\Phi\left(\lambda^{1-1/\gamma}r/u,\sigma,\norma( \rdtree)\right)}\right]_{\left|\lambda=\mass_{u}\right.}}\nonumber\\
			&=\expp{-\int_0^{\infty}  \dd r \m\left[1-\e^{-\Phi\left(\lambda^{1-1/\gamma}r/u,\sigma,\norma( \rdtree)\right)}\right]_{\left|\lambda=\mass_{u}\right.}}\nonumber\\
			&= \condex{\expp{-\sum_{s\geqslant 0} \Phi\left(\mass_u^{1-1/\gamma}s/u,\mu(\tree_s'),\norma\left(\tree_s'\right)\right)}}{\Tdown_u}.
		\end{aligned}
	\end{multline}
	
	\noindent
	\textbf{Step 4.} We deduce that a.s. for every $u>0$
	\begin{multline}
		\lim_{\epsilon \to 0}Y_\epsilon(u)= \frac{1}{\mass_u}\ind{\mass_u>1} f\circ \norma \left(\Tdown_u\right) g\left(\mass_u^{-1+1/\gamma}u\right)\\
		\times\condex{\expp{-\sum_{s\geqslant 0} \Phi\left(\mass_u^{1-1/\gamma}s/u,\mu(\tree_s'),\norma\left(\tree_s'\right)\right)}}{\Tdown_u}.
	\end{multline}
	Since $|Y_\epsilon(u)|\leqslant \norm{f}_\infty \norm{g}_\infty \mass_u^{-1}\ind{\mass_u>1}$ where the right-hand side is integrable with respect to $\mathbf{1}_{(0,\infty)}(u)\, \dd u \otimes \mathbb{P}$ thanks to \eqref{eq: tau is integrable}, it follows by dominated convergence that
	\begin{multline}
		\lim_{\epsilon \to 0}\int_0^\infty \dd u\ex{Y_\epsilon(u)}
		= \int_0^\infty \dd u\operatorname{\mathbb{E}}\left[\frac{1}{\mass_u}\ind{\mass_u>1} f\circ \norma \left(\Tdown_u\right) g\left(\mass_u^{-1+1/\gamma}u\right)\right.\\
		\left.\times\expp{-\sum_{s\geqslant 0} \Phi\left(\mass_u^{1-1/\gamma}s/u,\mu(\tree_s'),\norma\left(\tree_s'\right)\right)}\right].
	\end{multline}
	
	\noindent
	\textbf{Step 5.} Using Theorem \ref{thm: Bismut} and Lemma \ref{lemma: sigma = 1 ou > 1}-(i) again, we get that
	\begin{align*}
		\lim_{\epsilon \to 0}F(\epsilon) &= 
		\begin{multlined}[t]
			\Gamma(1-1/\gamma)\int_0^\infty \dd u\operatorname{\mathbb{E}}\left[\frac{1}{\mass_u}\ind{\mass_u>1} f\circ \norma \left(\Tdown_u\right) g\left(\mass_u^{-1+1/\gamma}u\right) \right.\\
			\left.\times\expp{-\sum_{s\geqslant 0} \Phi\left(\mass_u^{1-1/\gamma}s/u,\mu(\tree_s'),\norma\left(\tree_s'\right)\right)}\right]
		\end{multlined}\\
		&= 
		\begin{multlined}[t]
			\Gamma(1-1/\gamma)\n\left[\frac{1}{\sigma}\ind{\sigma>1}\int_\rdtree \mu(\dd x)f\circ \norma\left(\rdtree\right)g\left(\sigma^{-1+1/\gamma}H(x)\right)\right.\\
			\left.\times\expp{-\sum_{s\geqslant 0} \Phi\left(\sigma^{1-1/\gamma}s/H(x),\mu(\tree_s'),\norma\left(\tree_s'\right)\right)}\right]
		\end{multlined}\\
		&=\excm{1}\left[\int_{\rdtree}\mu(\dd x)f(\rdtree)g(H(x))\expp{-\sum_{s\geqslant 	0}\Phi\left(s/H(x),\mu(\tree_s'),\norma\left(\tree_s'\right)\right)}\right]\\
		&=\excm{1}\left[f(\rdtree)g(H(U))\expp{-\sum_{s\geqslant 	0}\Phi\left(s/H(U),\mu(\tree_s'),\norma\left(\tree_s'\right)\right)}\right],
	\end{align*}
	where, with a slight abuse of notation, we denote by $(\tree_s', \, s \geqslant 0)$ a Poisson point process with intensity $\m$ under $\excm{1}$, independent of $(\rdtree, H(U))$. Since $H(U)$ and $(\tree_s', \, s \geqslant 0)$ are independent, this concludes the proof.
\end{proof}

As a consequence of Theorem \ref{thm: zooming}, the next result gives the asymptotic behavior of the total mass of the subtrees grafted near the root of the stable tree.
\begin{corollary}\label{corollary: zooming mass}
	Let $\rdtree$ be the normalized stable tree with branching mechanism $\psi(\lambda) = \lambda^\gamma$ where $\gamma \in (1,2]$. Conditionally on $\rdtree$, let $U$ be $\rdtree$-valued random variable with distribution $\mu$ under $\excm{1}$. Assume that $\lim_{\epsilon \to 0}\epsilon^{-1/2}\f(\epsilon) =0$ and $\lim_{\epsilon \to 0} \epsilon^{-1}\f(\epsilon)=\infty$. Define a process $S^\epsilon$ by
	\begin{equation*}
		S_t^\epsilon \coloneqq \sum_{ h_i\leqslant \epsilon t \wedge \f(\epsilon)H(U)}\epsilon^{-\gamma/(\gamma-1)} \sigma_i, \quad t \geqslant 0.
	\end{equation*}
	Then we have the following convergence in distribution
	\begin{equation}\label{eq: cv to subordinator}
		\left(\rdtree,H(U),\left( S_t^\epsilon,\, t \geqslant 0\right)\right) \xrightarrow[\epsilon \to 0]{(d)} \left(\rdtree,H(U),\left(S_t,\, t\geqslant 0\right)\right)
	\end{equation}
in the space $\T \times \real \times D[0,\infty)$, where $S$ is a stable subordinator with Laplace exponent $\phi$ given by \eqref{eq: laplace exponent}, independent of $(\rdtree,H(U))$.
\end{corollary}
\begin{proof}
	We adapt the arguments of \cite[Chapter VII, Section 7.2]{resnick2007heavy}, see also Theorem 3.1 and Corollary 3.4 in \cite{tyrankaminska2010convergence}. Since the process $S$ has no fixed points of discontinuity, it is enough to show that the convergence \eqref{eq: cv to subordinator} holds in $\T \times \real \times D[0,r]$ for every $r >0$.
	
	Fix $r >0$ and let $\delta >0$. Define
	\begin{equation*}
		S_t^{\epsilon,\delta} \coloneqq \sum_{ h_i\leqslant \epsilon t \wedge \f(\epsilon)H(U)} \epsilon^{-\gamma/(\gamma-1)}\sigma_i \ind{\epsilon^{-\gamma/(\gamma-1)}\sigma_i >\delta}, \quad t \geqslant 0.
	\end{equation*}
 	Recall that for a metric space $X$, we denote by $\M_p(X)$ the space of point measures on $X$ equipped with the topology of vague convergence. It is known (see \cite[p. 215]{resnick2007heavy}) that the restriction mapping
	\begin{equation*}
		m \mapsto m_{\left|[0,\infty)\times (\delta,\infty)\right.}
	\end{equation*}
	is a.s. continuous from $\M_p([0,\infty)^2)$ to $\M_p([0,\infty)\times (\delta,\infty))$ with respect to the distribution of the Poisson random measure $\sum_{s\geqslant 0} \delta_{(s,\mu(\tree'_s))}$. Furthermore, the summation mapping
	\begin{equation*}
		m \mapsto \left(\int_{[0,t]\times (\delta,\infty)}x \, m(\dd s,\dd x),\, 0\leqslant t \leqslant r\right)
	\end{equation*}
	is a.s. continuous from $\M_p([0,\infty)\times (\delta,\infty))$ to $D[0,r]$ with respect to the same distribution. We deduce from Theorem \ref{thm: zooming}-(i) and the continuous mapping theorem the following convergence in distribution
	\begin{equation}\label{eq: second converging together 1}
			\left(\rdtree,H(U),\left( S_t^{\epsilon,\delta},\, 0\leqslant t \leqslant r\right)\right) \xrightarrow[\epsilon \to 0]{(d)} \left(\rdtree,H(U),\left(\sum_{s\leqslant t}\mu(\tree_s')\ind{\mu(\tree_s')>\delta},\, 0\leqslant t\leqslant r\right)\right)
	\end{equation}
in $\T\times \real \times D[0,r]$, where $(\tree_s', s\geqslant 0)$ is a Poisson point process with intensity $\m$, independent of $(\rdtree,H(U))$.

Furthermore, since $\sum_{s\leqslant r} \mu(\tree_s')$ is $\excm{1}$-a.s. finite, it is clear by the dominated convergence theorem that $\excm{1}$-a.s.
\begin{equation*}
	\lim_{\delta \to 0}\sup_{t \leqslant r} \left|\sum_{s\leqslant t} \mu(\tree_s') - \sum_{s\leqslant t}\mu(\tree_s')\ind{\mu(\tree_s')>\delta}\right| = \lim_{\delta \to 0}\sum_{s\leqslant r} \mu(\tree_s')\ind{\mu(\tree_s')\leqslant \delta} = 0.
\end{equation*}
Since uniform convergence on $[0,T]$ implies convergence for the Skorokhod $J1$ topology, we deduce that
\begin{equation}\label{eq: second converging together 2}
	\left(\rdtree,H(U),\left(\sum_{s\leqslant t}\mu(\tree_s')\ind{\mu(\tree_s')>\delta},\, 0\leqslant t\leqslant r\right)\right) \xrightarrow[\delta \to 0]{(d)} \left(\rdtree,H(U),\left(S_t,\, 0\leqslant t\leqslant r\right)\right),
\end{equation}
where $S_t = \sum_{s\leqslant t} \mu(\tree_s')$ is a stable subordinator with Laplace exponent $\phi$, independent of $(\rdtree,H(U))$.

Finally, we shall prove that for every $\eta >0$
\begin{equation}\label{eq: second converging together 3}
	\lim_{\delta \to 0} \limsup_{\epsilon\to 0} \excm{1}\left[\sup_{0\leqslant t\leqslant T}\left|S_t^\epsilon-S_t^{\epsilon,\delta}\right|\geqslant \eta\right] = 0.
\end{equation}
Let $f \colon [0,\infty) \to [0,\infty)$ be Lipschitz-continuous such that $x\mathbf{1}_{[0,\delta]}(x)\leqslant f (x)\leqslant x\mathbf{1}_{[0,2\delta]}(x)$. We have
\begin{align*}
	\sup_{0\leqslant t \leqslant r} \left|S_t^\epsilon-S_t^{\epsilon,\delta}\right| &= \sum_{h_i \leqslant \epsilon r \wedge \f(\epsilon)H(U)} \epsilon^{-\gamma/(\gamma-1)}\sigma_i \ind{\epsilon^{-\gamma/(\gamma-1)}\sigma_i\leqslant \delta}\\ 
	&\leqslant \sum_{h_i \leqslant \epsilon r \wedge \f(\epsilon)H(U)}  f\left(\epsilon^{-\gamma/(\gamma-1)}\sigma_i\right).
\end{align*}
It follows that
\begin{align}\label{eq: uniform estimate}
	\limsup_{\epsilon\to 0} \excm{1}\left[\sup_{0\leqslant t\leqslant r}\left|S_t^\epsilon-S_t^{\epsilon,\delta}\right|\geqslant \eta\right]&\leqslant \limsup_{\epsilon\to 0} \excm{1}\left[ \sum_{h_i \leqslant \epsilon r \wedge \f(\epsilon)H(U)}  f\left(\epsilon^{-\gamma/(\gamma-1)}\sigma_i\right)\geqslant \eta\right]\nonumber\\
	&\leqslant \excm{1}\left[\sum_{s\leqslant r} f\left(\mu(\tree_s')\right)\geqslant \eta\right]\nonumber\\
	&\leqslant \excm{1}\left[\sum_{s\leqslant r} \mu(\tree_s')\ind{\mu(\tree_s')\leqslant 2\delta}\geqslant \eta\right],
\end{align}
where in the second inequality we used the Portmanteau theorem together with the following convergence in distribution 
\begin{equation*}
	\sum_{h_i \leqslant \epsilon r \wedge \f(\epsilon)H(U)}  f\left(\epsilon^{-\gamma/(\gamma-1)}\sigma_i\right) \xrightarrow[\epsilon \to 0]{(d)}\sum_{s\leqslant r} f\left(\mu(\tree_s')\right),
\end{equation*}
which holds thanks to Theorem \ref{thm: zooming}-(i) applied with $\Phi(h,a,T) = \ind{h\leqslant r} f(a)$. But, by the dominated convergence theorem, we have that $\excm{1}$-a.s.
\begin{equation*}
	\lim_{\delta \to 0}\sum_{s\leqslant r} \mu(\tree_s')\ind{\mu(\tree_s')\leqslant 2\delta}=0.
\end{equation*}
Together with \eqref{eq: uniform estimate}, this implies \eqref{eq: second converging together 3}.

Putting together \cref{eq: second converging together 1,eq: second converging together 2,eq: second converging together 3}, it follows from the second converging together theorem, see e.g. \cite[Theorem 3.2]{billingsley1999convergence}, that
\begin{equation*}
		\left(\rdtree,H(U),\left( S_t^\epsilon,\, 0\leqslant t \leqslant r\right)\right) \xrightarrow[\epsilon \to 0]{(d)} \left(\rdtree,H(U),\left(S_t,\, 0\leqslant t\leqslant r\right)\right)
\end{equation*}
in $\T\times \real \times D[0,r]$. This finishes the proof.
\end{proof}

\begin{remark}\label{remark: haas}
	Let us comment on the connection between Theorem \ref{thm: zooming} and the small time asymptotics of the fragmentation at height of the stable tree $F^{-}$, see \cite[Section 4]{bertoin2002selfsimilar} for the Brownian case $\gamma=2$ and \cite{miermont2003fragmentations} for the case $\gamma \in (1,2)$. We briefly recall its definition. Consider the normalized stable tree $\rdtree$ and denote by $\left(\rdtree_j, \, j \in J_t\right)$ the connected components of the set $\{x \in \rdtree\colon \, H(x)>t\}$ obtained from $\rdtree$ by removing vertices located at height $\leqslant t$. Then $F^{-}(t) = (F_1^{-}(t),F_2^{-}(t), \ldots)$ is defined as the decreasing sequence of masses $\left(\mu(\rdtree_j),\, j \in J_t\right)$. In \cite[Section 5.1]{haas2007fragmentation}, Haas obtains the following functional convergence in distribution as a consequence of a more general result
	\begin{equation}\label{eq: haas}
	\epsilon^{-\gamma/(\gamma-1)} (1-F_1^{-}(\epsilon \cdot),(F_2^{-}(\epsilon \cdot), F_3^{-}(\epsilon \cdot), \ldots)) \xrightarrow[\epsilon \to 0]{(d)} (S, FI),
	\end{equation}
	where the convergence holds with respect to the Skorokhod $J1$ topology. Here $FI$ is a fragmentation process with immigration and $S$ is a stable subordinator with index $1-1/\gamma$ representing the total mass of immigrants.
	
	At least heuristically, this can be recovered from Theorem \ref{thm: zooming}. 
	Let $U \in \rdtree$ be a leaf chosen uniformly at random. It is not difficult to see that for $0\leqslant t \leqslant H(U)$, with high probability as $\epsilon \to 0$, the biggest fragment at time $\epsilon t$ is the one containing $U$. Thus we get $	1- F_1^{-}(\epsilon t) = \sum_{h_i \leqslant \epsilon t} \sigma_i$ and
	\begin{equation*}
	 (F_2^{-}(\epsilon t), F_3^{-}(\epsilon t), \ldots) = (\mu(\rdtree_i^{\geqslant \epsilon t - h_i}), \, h_i \leqslant \epsilon t)^{\downarrow}
	\end{equation*}
is the decreasing rearrangement of the masses of $\rdtree_i^{\geqslant \epsilon t -h_i}$ for the subtrees grafted at height $h_i \leqslant \epsilon t$. Here we denote by $T^{\geqslant r}= T\setminus T^{<r} = \{x \in T\colon \, H(x)\geqslant r\}$ the set of vertices of $T$ above height $r$. To recover \eqref{eq: haas}, we may prove the joint convergence of
\begin{equation}\label{eq: haas approx}
	\left(\sum_{h_i \leqslant \epsilon \cdot \wedge \epsilon H(U)}\epsilon^{-\gamma/(\gamma-1)} \sigma_i,\sum_{h_i \leqslant \epsilon H(U)} \delta_{\left(\ind{h_i \leqslant \epsilon t} \epsilon^{-\gamma/(\gamma-1)}\mu(\rdtree_i^{\geqslant \epsilon t-h_i}), \, t \geqslant 0\right)}\right),
\end{equation}
then argue that the convergence of the point measure in \eqref{eq: haas approx} implies that of the rearranged atoms. Notice that we may obtain the convergence of the first coordinate in \eqref{eq: haas approx} using Theorem \ref{thm: zooming}-(ii), similarly to how we proved Corollary \ref{corollary: zooming mass} using Theorem \ref{thm: zooming}-(i). For the convergence of the second coordinate, the idea is to consider
$	\Phi(h,a,T) = F\left((\ind{h\leqslant t} a\mu(T^{\geqslant a^{-1+1/\gamma}( t-h)}),\, t \geqslant 0)\right),$
where $F\colon D[0,\infty) \to [0,\infty)$ is Lipschitz-continuous with compact support. However, $\Phi$ is not Lipschitz-continuous with respect to $a$ so our result does not apply directly. Similarly, to get the convergence of the dust, notice that
\begin{equation*}
	\mu(\rdtree^{< \epsilon t}) = \sum_{h_i \leqslant \epsilon t} \mu(\rdtree_i^{< \epsilon t- h_i}).
\end{equation*}
Thus the idea is to apply Theorem \ref{thm: zooming}-(ii) with $\Phi(h,a,T) = \ind{h\leqslant t} a \mu(T^{<a^{-1+1/\gamma}(t-h)})$ which again does not satisfy the assumptions.
\end{remark}

\section{Asymptotic behavior of $\Z_{\alpha,\beta}$ in the case $\beta/\alpha^{1-1/\gamma} \to c \in [0,\infty)$}\label{section: subcritical}
We start by showing that if $U\in \rdtree$ is a leaf chosen uniformly at random, $Z_{\alpha,\beta}(U)$ defined in \eqref{eq: the functional on a Levy tree} converges in disrtibution after proper rescaling.
\begin{proposition}\label{prop: subcritical case}
	Assume that $\alpha \to \infty$, $\beta \geqslant 0$ and $\beta/\alpha^{1-1/\gamma} \to c \in [0,\infty)$. 	Let $\rdtree$ be the normalized stable tree with branching mechanism $\psi(\lambda) = \lambda^{\gamma}$ where $\gamma\in (1,2]$. Conditionally on $\rdtree$, let $U$ be a $\rdtree$-valued random variable with disribution $\mu$ under $\excm{1}$. Then we have the following convergence in distribution
	\begin{equation}\label{eq: convergence random leaf}
		\left(\rdtree,H(U),\alpha^{1-1/\gamma} \H^{-\beta} Z_{\alpha, \beta}(U)\right) \xrightarrow[\alpha \to \infty]{(d)} \left(\rdtree, H(U),\int_0^\infty \e^{-S_t-ct/\H}\,\dd t\right),
	\end{equation}
where $(S_t,\, t \geqslant 0)$ is a stable subordinator with Laplace exponent $\phi$ given by \eqref{eq: laplace exponent}, independent of $(\rdtree,H(U))$.
\end{proposition}

\begin{proof}
	Set 
	\begin{equation}\label{eq: definition epsilon}
		\epsilon = \epsilon(\alpha) \coloneqq \alpha^{(\delta-1)(1-1/\gamma)}	\end{equation} 
	with $\delta \in (0, 1/3)$ so that $\epsilon \to 0$ as $\alpha \to \infty$. Define
		\begin{equation}\label{eq: definition I}
		I_\alpha \coloneqq \alpha^{1-1/\gamma} \int_0^{\epsilon H(U)} \e^{-\alpha(1-\sigma_{r,U})-\beta r/\H}\, \dd r.
		\end{equation}
\begin{lemma}\label{lemma: convergence in probability J}
	We have the following convergence in $\excm{1}$-probability
	\begin{equation*}
		 \lim_{\alpha \to \infty}\left(\alpha^{1-1/\gamma}\H^{-\beta}Z_{\alpha,\beta}(U)-I_\alpha\right) = 0.
	\end{equation*}
\end{lemma}

The proof is postponed to Section \ref{section: remainder}. Using this together with Slutsky's theorem, it is clear that the proof of \eqref{eq: convergence random leaf} reduces to showing the following convergence in distribution
\begin{equation}\label{eq: convergence of I}
	\left(\rdtree,H(U),I_\alpha\right) \xrightarrow[\alpha \to \infty]{(d)} \left(\rdtree, H(U),\int_0^\infty \e^{-S_t-ct/\H}\,\dd t\right).
\end{equation}

Making the change of variable $t=\alpha^{1-1/\gamma}r$, notice that
\begin{equation}
	I_\alpha = \int_0^{ \alpha^{1-1/\gamma}\epsilon H(U)} \expp{-\alpha \left(1-\sigma_{\alpha^{-1+1/\gamma}t,U}\right)-\beta \alpha^{-1+1/\gamma} t/\H}\, \dd t,
\end{equation}
Let $A>0$. Notice that, applying Corollary \ref{corollary: zooming mass}, we get the following convergence in distribution
\begin{equation}\label{eq: truncated at A 1}
	\left(\rdtree, H(U), \left(\sum_{h_i \leqslant \alpha^{-1+1/\gamma} t \wedge \epsilon H(U)} \alpha\sigma_i, \, 0\leqslant t \leqslant A\right)\right) \xrightarrow[\alpha \to \infty]{(d)} \left(\rdtree,H(U),\left(S_t, \, 0\leqslant t \leqslant A\right)\right),
\end{equation}
where $S$ is a subordinator with Laplace exponent $\phi$, independent of $(\rdtree, H(U))$. Moreover, on the event $\Omega_\alpha \coloneqq \{\alpha^{-1+1/\gamma}A\leqslant \epsilon H(U) \}$, we have for every $t \in [0,A]$
\begin{equation}\label{eq: truncated at A 2}
	 \sum_{h_i \leqslant \alpha^{-1+1/\gamma} t \wedge \epsilon H(U)} \sigma_i =  \sum_{h_i \leqslant \alpha^{-1+1/\gamma} t} \sigma_i = 1- \sigma_{\alpha^{-1+1/\gamma} t, U}.
\end{equation}
Since $\alpha^{1-1/\gamma} \epsilon \to \infty$, it is clear that $\lim_{\alpha \to \infty}\excm{1}[\Omega_\alpha] = 1$. Thus, it follows from \eqref{eq: truncated at A 1} and \eqref{eq: truncated at A 2} that
\begin{equation*}
	\left(\rdtree,H(U),\left( \alpha \left(1- \sigma_{\alpha^{-1+1/\gamma} t, U}\right), \, 0\leqslant t \leqslant A\right)\right) \xrightarrow[\alpha \to \infty]{(d)} \left(\rdtree,H(U),\left(S_t, \, 0\leqslant t \leqslant A\right)\right).
\end{equation*}
Now a simple application of the continuous mapping theorem gives
\begin{equation}\label{eq: convergence of I truncated}
	\left(\rdtree,H(U),\int_0^{A}\expp{-\alpha \left(1-\sigma_{\alpha^{-1+1/\gamma}t,U}\right)-\beta \alpha^{-1+1/\gamma} t/\H}\, \dd t \right)  \xrightarrow[\alpha \to \infty]{(d)} \left(\rdtree,H(U),\int_0^{A}\e^{-S_t - ct/\H}\, \dd t\right).
\end{equation}

On the other hand, applying \eqref{eq: first moment Psi} with $f(T)= \e^{-\alpha(1-\mu(T))}$ and $g(r) = \ind{r\geqslant \alpha^{-1+1/\gamma}A}$, we get
\begin{multline*}
	\excm{1}\left[\int_A^{\alpha^{1-1/\gamma}\epsilon H(U)}\expp{-\alpha \left(1-\sigma_{\alpha^{-1+1/\gamma}t,U}\right)-\beta \alpha^{-1+1/\gamma} t/\H}\, \dd t \right]\\
	\begin{aligned}
		&\leqslant\alpha^{1-1/\gamma}\excm{1}\left[\int_{\alpha^{-1+1/\gamma}A}^{H(U)}\expp{-\alpha \left(1-\sigma_{r,U}\right)}\, \dd r \right] \\
		&= \frac{\alpha^{1-1/\gamma}}{\gamma\Gamma(1-1/\gamma)} \int_0^1 x^{-1/\gamma} (1-x)^{-1/\gamma} \e^{-\alpha x} \excm{1}\left[\left(\alpha x\right)^{1-1/\gamma} H(U)\geqslant A\right]\, \dd x\\
		&= \frac{1}{\gamma\Gamma(1-1/\gamma)} \int_0^\alpha y^{-1/\gamma} \left(1-\frac{y}{\alpha}\right)^{-1/\gamma} \e^{- y} \excm{1}\left[y^{1-1/\gamma} H(U)\geqslant A\right]\, \dd y.
	\end{aligned}
\end{multline*}

By the dominated convergence theorem, we have
\begin{multline*}
	\lim_{\alpha \to \infty} \int_0^{\alpha/2} y^{-1/\gamma} \left(1-\frac{y}{\alpha}\right)^{-1/\gamma} \e^{- y} \excm{1}\left[y^{1-1/\gamma} H(U)\geqslant A\right]\, \dd y \\
	= \int_0^\infty y^{-1/\gamma} \e^{- y} \excm{1}\left[y^{1-1/\gamma} H(U)\geqslant A\right]\, \dd y.
\end{multline*}
Moreover, we have
\begin{align*}
	\int_{\alpha/2}^\alpha y^{-1/\gamma} \left(1-\frac{y}{\alpha}\right)^{-1/\gamma} \e^{- y} \excm{1}\left[y^{1-1/\gamma} H(U)\geqslant A\right]\, \dd y &\leqslant \e^{-\alpha/2} \int_{\alpha/2}^\alpha y^{-1/\gamma} \left(1-\frac{y}{\alpha}\right)^{-1/\gamma} \, \dd y \\
	&= \alpha^{1-1/\gamma} \e^{-\alpha/2} \int_{1/2}^1 z^{-1/\gamma} (1-z)^{-1/\gamma} \, \dd z,
\end{align*}
where the last term converges to $0$ as $\alpha \to \infty$. We deduce that
\begin{multline*}
	\limsup_{\alpha \to \infty}\excm{1}\left[\int_A^{\alpha^{1-1/\gamma}\epsilon H(U)}\expp{-\alpha \left(1-\sigma_{\alpha^{-1+1/\gamma}t,U}\right)-\beta \alpha^{-1+1/\gamma} t/\H}\, \dd t \right]\\
		\leqslant\frac{1}{\gamma\Gamma(1-1/\gamma)} \int_0^\infty y^{-1/\gamma} \e^{- y} \excm{1}\left[y^{1-1/\gamma} H(U)\geqslant A\right]\, \dd y,
\end{multline*}
and, thanks to the dominated convergence theorem,
\begin{equation}\label{eq: remainder sequence}
	\lim_{A\to \infty}\limsup_{\alpha \to \infty}\excm{1}\left[\int_A^{\alpha^{1-1/\gamma}\epsilon H(U)}\expp{-\alpha \left(1-\sigma_{\alpha^{-1+1/\gamma}t,U}\right)-\beta \alpha^{-1+1/\gamma} t/\H}\, \dd t \right] = 0.
\end{equation}

Combining \eqref{eq: convergence of I truncated} and \eqref{eq: remainder sequence} and applying \cite[Theorem 3.2]{billingsley1999convergence}, \eqref{eq: convergence of I} readily follows. This finishes the proof.
\end{proof}

The next lemma, whose proof is postponed to Section \ref{section: convergence in proba}, states that taking a leaf uniformly at random or taking the average over all leaves yields the same limiting behavior for $Z_{\alpha,\beta}(x)$. Recall from \eqref{eq: the functional on a Levy tree} the definition of $\Z_{\alpha,\beta}$.
\begin{lemma}\label{lemma: convergence in proba}
	Under the assumptions of Theorem \ref{prop: subcritical case}, we have the convergence in $\excm{1}$-probability
	\begin{equation}
		\lim_{\alpha \to \infty}\alpha^{1-1/\gamma}\H^{-\beta} \left(Z_{\alpha,\beta}(U)-\Z_{\alpha,\beta}\right) =0.
	\end{equation}
\end{lemma}

Combining Proposition \ref{prop: subcritical case} and Lemma \ref{lemma: convergence in proba}, we get the following result using Slutsky's theorem.
\begin{thm}\label{thm: subcritical case}
	Assume that $\alpha \to \infty$, $\beta \geqslant 0$ and $\beta/\alpha^{1-1/\gamma} \to c \in [0,\infty)$. Let $\rdtree$ be the stable tree with branching mechanism $\psi(\lambda) = \lambda^\gamma$ where $\gamma \in (1,2]$. Conditionally on $\rdtree$, let $U$ be a $\rdtree$-valued random variable with distribution $\mu$ under $\excm{1}$. Then we have the following convergence in distribution
	\begin{multline}
		\left(\rdtree,H(U),\alpha^{1-1/\gamma} \H^{-\beta}Z_{\alpha,\beta}(U),\alpha^{1-1/\gamma}\H^{-\beta}\Z_{\alpha,\beta}\right) \\
		\xrightarrow[\alpha \to \infty]{(d)} \left(\rdtree, H(U),\int_0^\infty \e^{-S_t-ct/\H}\,\dd t,\int_0^\infty \e^{-S_t-ct/\H}\,\dd t\right),
	\end{multline}
where $S$ is a stable subordinator with Laplace exponent $\phi$ given by \eqref{eq: laplace exponent}, independent of $(\rdtree,H(U))$.
\end{thm}

\section{Asymptotic behavior of $\Z_{\alpha,\beta}$ in the case $\beta/\alpha^{1-1/\gamma} \to \infty$}
\label{section: supercritical}
We treat the case $\beta/\alpha^{1-1/\gamma} \to \infty$. Intuitively, this assumption guarantees that $\H_{r,x}^\beta$ dominates $\sigma_{r,x}^\alpha$, thus we get a different asymptotic behavior and there is no longer a subordinator in the limit.
\begin{thm}\label{thm: supercritical case}
	Assume that $\beta \to \infty$, $\alpha \geqslant 0$ and $\alpha^{1-1/\gamma}/\beta \to 0$. Let $\rdtree$ be the stable tree with branching mechanism $\psi(\lambda) = \lambda^\gamma$ where $\gamma \in (1,2]$. Then we have the following convergence in $\excm{1}$-probability
	\begin{equation}
		\lim_{\beta \to \infty} \beta\H^{-\beta}\Z_{\alpha,\beta}= \H.
	\end{equation}
Furthermore, if $\alpha^{1-1/\gamma}/\beta^{\rho}\to 0$ for some $\rho \in (0,1)$, then the convergence holds $\excm{1}$-almost surely.
\end{thm}

\begin{proof}We start by assuming that $\alpha \to \infty$ and $\alpha^{1-1/\gamma}/\beta \to 0$ (the case $\alpha$ bounded from above is covered by the second part of the theorem). Setting $\epsilon = (\alpha^{1-1/\gamma}\beta)^{-1/2}$, it is straightforward to check that $\epsilon \to 0$, $\beta \epsilon \to \infty$ and $\alpha^{1-1/\gamma} \epsilon \to 0$. Write 
	\begin{equation}\label{eq: decomposition E F}\beta\H^{-\beta}\Z_{\alpha,\beta}=E_\beta+\sum_{i=1}^4 F_\beta^{i}
	\end{equation} 
	where
	\begin{align*}
		F_\beta^1 &= \beta\int_\rdtree \ind{H(x)<2\epsilon}\,\mu(\dd x)\int_0^{H(x)} \sigma_{r,x}^\alpha \left(\frac{\H_{r,x}}{\H}\right)^\beta \, \dd r,\\
		F_\beta^2 &= \beta\int_\rdtree \ind{H(x)\geqslant 2\epsilon}\,\mu(\dd x) \int_\epsilon^{H(x)}\sigma_{r,x}^\alpha \left(\frac{\H_{r,x}}{\H}\right)^\beta \, \dd r,\\
		F_\beta^3 &= \beta\int_\rdtree \ind{H(x)\geqslant 2\epsilon}\,\mu(\dd x) \int_0^{\epsilon} \sigma_{r,x}^\alpha \left[\left(\frac{\H_{r,x}}{\H}\right)^\beta -\left(1-\frac{r}{\H}\right)^\beta \right]\, \dd r, \\
		F_\beta^4 &= \beta\int_\rdtree \ind{H(x)\geqslant 2\epsilon}\,\mu(\dd x) \int_0^{\epsilon} \sigma_{r,x}^\alpha \left[\left(1-\frac{r}{\H}\right)^\beta -\e^{-\beta r/\H}\right]\, \dd r,\\
		E_\beta &= \beta\int_\rdtree \ind{H(x)\geqslant 2\epsilon}\,\mu(\dd x) \int_0^{\epsilon}\sigma_{r,x}^\alpha \e^{-\beta r/\H}\, \dd r.
	\end{align*}
	We shall prove that $\lim_{\beta \to \infty}F_\beta^{i}= 0$ in $\excm{1}$-probability for every $i \in \{1,2,3,4\}$.
	
	Let $p\in (1,2)$. Using that $\sigma_{r,x} \leqslant 1$ and $\H_{r,x}\leqslant \H$ and applying the Markov inequality, it is clear that
	\begin{equation*}
		F_\beta^1  \leqslant 2\beta \epsilon \int_\rdtree \ind{H(x)< 2\epsilon}\, \mu(\dd x)\leqslant 2^{1+p}\beta\epsilon^{1+p} \int_\rdtree H(x)^{-p}\,\mu(\dd x).
	\end{equation*}
	Since the last integral has a finite first moment by Lemma \ref{lemma: height integrable} and $\beta\epsilon^{1+p}\to 0$, we deduce that $\excm{1}$-a.s. $\lim_{\beta \to \infty} F_\beta^1 = 0$.

	Next, using \eqref{eq: height is sublinear}, we get
	\begin{align}\label{eq: F2}
		F_\beta^2 &= \beta\int_\rdtree \ind{H(x)\geqslant 2\epsilon}\,\mu(\dd x) \int_\epsilon^{H(x)}\sigma_{r,x}^\alpha \left(\frac{\H_{r,x}}{\H}\right)^\beta \, \dd r \nonumber \\
		&\leqslant \beta \left(1-\frac{\epsilon}{\H}\right)^\beta\int_\rdtree \mu(\dd x)\int_0^{H(x)}\sigma_{r,x}^\alpha \, \dd r.
	\end{align}
	By \cite[Corollary 6.6]{abraham2020global}, we have
	\begin{equation*}
	\excm{1}\left[\int_\rdtree \mu(\dd x)\int_0^{H(x)}\sigma_{r,x}^\alpha \, \dd r\right] = \frac{1}{|\Gamma(-1/\gamma)|} \mathrm{B}\left(\alpha + 1-1/\gamma,1-1/\gamma\right),
	\end{equation*}
where $\mathrm{B}$ is the beta function. Using that $\mathrm{B}(x,1-1/\gamma)\sim \Gamma(1-1/\gamma)x^{-1+1/\gamma}$ as $x\to \infty$, we deduce that
\begin{equation}\label{eq: F2 bounded in L1}
\sup_{\alpha \geqslant 0 }\excm{1}\left[\alpha^{1-1/\gamma} \int_\rdtree \mu(\dd x)\int_0^{H(x)}\sigma_{r,x}^\alpha \, \dd r\right] < \infty.
\end{equation}
On the other hand, let $\theta >1$. Since the function $x \mapsto x^{1+\theta}\e^{-x}$ is bounded on $[0,\infty)$, it follows that
\begin{equation}\label{eq: F2 as convergence}
	 \frac{\beta}{\alpha^{1-1/\gamma}} \left(1-\frac{\epsilon}{\H}\right)^\beta\leqslant \frac{\beta}{\alpha^{1-1/\gamma}} \e^{-\beta \epsilon/\H}\leqslant C\frac{\H^{1+\theta}}{\beta ^\theta \epsilon^{1+\theta}\alpha^{1-1/\gamma}}
\end{equation}
for some constant $C>0$. Notice that $\beta ^\theta \epsilon^{1+\theta}\alpha^{1-1/\gamma}\to \infty$ since $\theta >1$. Thus the right-hand side of \eqref{eq: F2 as convergence} goes to $0$ almost surely. Now putting together \eqref{eq: F2}, \eqref{eq: F2 bounded in L1} and \eqref{eq: F2 as convergence}, we deduce that $\lim_{\beta \to \infty} F_\beta^2=0$ in $\excm{1}$-probability.
	
	Let $x \in \rdtree$. Recall from \eqref{eq: height is sublinear} and \eqref{eq: height is linear} that $\H_{r,x} \leqslant \H-r$ for every $r \in [0,H(x)]$ and that the equality holds for $r \in [0,H(x\wedge x^*)]$. Therefore, we get
	\begin{align*}
		|F_\beta^3|&= \beta \int_\rdtree \ind{H(x)\geqslant 2\epsilon, \, H(x\wedge x^*)<\epsilon} \, \mu(\dd x)  \int_{H(x\wedge x^*)}^\epsilon \sigma_{r,x}^\alpha \left[\left(1-\frac{r}{\H}\right)^\beta- \left(\frac{\H_{r,x}}{\H}\right)^\beta\right]\, \dd r \\
		&\leqslant \beta \int_\rdtree \ind{H(x)\geqslant 2\epsilon, \, H(x\wedge x^*)<\epsilon} \, \mu(\dd x)  \int_{H(x\wedge x^*)}^\epsilon \left(1-\frac{r}{\H}\right)^\beta\,\dd r \\
		&\leqslant \beta \int_\rdtree \ind{H(x)\geqslant 2\epsilon, \, H(x\wedge x^*)<\epsilon} \, \mu(\dd x)  \int_{H(x\wedge x^*)}^\epsilon \e^{-\beta r/\H}\,\dd r \\
		&\leqslant \H \int_\rdtree \e^{-\beta H(x\wedge x^*)/\H}\, \mu(\dd x).
	\end{align*}
Since $H(x\wedge x^*) >0$ for $\mu$-a.e. $x \in \rdtree$, a simple application of the dominated convergence theorem gives that $\excm{1}$-a.s. $\lim_{\beta\to \infty} F_\beta^3 = 0$.
	
	Furthermore, using the inequality $|\e^b-\e^{a}|\leqslant |b-a|\e^{b}$ for $a\leqslant b$ together with the fact that $j \colon y\mapsto -(y+\log(1-y))/y^2$ is increasing on $[0,1)$, we get for $r \in [0,\epsilon]$
	\begin{equation*}
		\left| \e^{-\beta r/\H} - \left(1-\frac{r}{\H}\right)^\beta\right| \leqslant \beta \left|\frac{r}{\H} + \log\left(1-\frac{r}{\H}\right)\right| \e^{-\beta r/\H} \leqslant \beta \left(\frac{r}{\H}\right)^2 \e^{-\beta r/\H}j\left(\frac{\epsilon}{\H}\right).
	\end{equation*}
	Therefore, we deduce that 
	\begin{equation*}
		|F_\beta^4| \leqslant j\left(\frac{\epsilon}{\H}\right)\int_\rdtree \ind{H(x)\geqslant 2\epsilon}\,\mu(\dd x)\int_0^\epsilon \left(\frac{\beta r}{\H}\right)^2 \e^{-\beta r/\H} \, \dd r\leqslant C j\left(\frac{\epsilon}{\H}\right)\epsilon,
	\end{equation*}
	where we used that $y \mapsto y^2 \e^{-y}$ is bounded on $[0,\infty)$ by some constant $C<\infty$ for the second inequality. Since $\lim_{y\to 0}j(y)=1/2$, we get $\excm{1}$-a.s. $\lim_{\beta \to \infty} F_\beta^4 = 0$.
	We deduce the following convergence in $\excm{1}$-probability
	\begin{equation}\label{eq: only I 2}
		\lim_{\beta \to \infty}\sum_{i=1}^4 F_\beta^{i} = 0.
	\end{equation}

	Notice that
	\begin{equation}\label{eq: I upper bound}
		E_\beta \leqslant \beta \int_\rdtree\ind{H(x)\geqslant 2\epsilon}\,\mu(\dd x) \int_0^\epsilon \e^{-\beta r/\H}\, \dd r = \H \left(1-\e^{- \beta\epsilon/\H}\right)\int_\rdtree\ind{H(x)\geqslant 2\epsilon} \,\mu(\dd x)\leqslant  \H.
	\end{equation}
	On the other hand, using that $\sigma_{r,x} \geqslant \sigma_{\epsilon,x}$ for every $x\in \rdtree$ such that $H(x)\geqslant 2\epsilon$ and every $r\in[0,\epsilon]$, we get
	\begin{equation}\label{eq: I lower bound}
		E_\beta \geqslant \H\left(1-\e^{- \beta \epsilon/\H}\right) \int_\rdtree \ind{H(x)\geqslant 2\epsilon}\sigma_{\epsilon,x}^\alpha\,\mu(\dd x).
	\end{equation}
	
	We now shall prove the following convergence in $\excm{1}$-probability
	\begin{equation}\label{eq: l1 cv}
		 \lim_{\beta \to \infty}\int_\rdtree \ind{H(x)\geqslant 2\epsilon}\sigma_{\epsilon,x}^\alpha\,\mu(\dd x) = 1. 
	\end{equation} 
	Using Lemma \ref{lemma: sigma = 1 ou > 1}-(i) and Bismut's decomposition \eqref{eq: Bismut}, we have
	\begin{align}\label{eq: l1 cv pre}
		\excm{1}\left[\int_\rdtree \ind{H(x)\geqslant 2\epsilon}\sigma_{\epsilon,x}^\alpha\,\mu(\dd x)\right]&=
		\Gamma(1-1/\gamma)\n\left[\frac{1}{\sigma}\ind{\sigma >1}\int_\rdtree \ind{\sigma^{-1+1/\gamma}H(x)\geqslant 2\epsilon}\left(\frac{\sigma_{\sigma^{1-1/\gamma}\epsilon,x}}{\sigma}\right)^\alpha\,\mu(\dd x)\right] \nonumber\\
		&=\Gamma(1-1/\gamma)\int_0^\infty \dd t \ex{\frac{1}{S_t}\ind{S_t >1, \, t \geqslant 2\epsilon S_t ^{1-1/\gamma}} \left(1-\frac{S_{\epsilon S_t^{1-1/\gamma}}}{S_t}\right)^\alpha }.
	\end{align}
Recall that $S$ is a stable subordinator with index $1-1/\gamma$. Thus we have the following identity in distribution for $c >0$
\begin{equation*}
	\left(S_{cr}, \, r\geqslant 0\right) \lawd \left(c^{\gamma/(\gamma-1)}S_r,\, r\geqslant 0\right).
\end{equation*}
Applying this, we get that
\begin{equation}\label{eq: subordinator identity}
\alpha S\left(\epsilon S^{1-1/\gamma}_t\right) \lawd S\left(\epsilon S^{1-1/\gamma}_{\alpha^{1-1/\gamma}t}\right).
\end{equation}
Now notice that
\begin{equation*}
\epsilon S_{\alpha^{1-1/\gamma}t}^{1-1/\gamma} \lawd \epsilon \alpha^{1-1/\gamma} S_t^{1-1/\gamma}.
\end{equation*}
Since $\epsilon \alpha^{1-1/\gamma} \to 0$, this clearly implies that $\epsilon S_{\alpha^{1-1/\gamma}t}^{1-1/\gamma} \to 0$ in probability. As $S$ is a.s. continuous at $0$, we deduce that $S\left(\epsilon S^{1-1/\gamma}_{\alpha^{1-1/\gamma}t}\right)\to 0$ in probability. Thus, it follows from \eqref{eq: subordinator identity} that $\alpha S\left(\epsilon S^{1-1/\gamma}_t\right) \to 0$ in probability for every $t >0$ and
\begin{equation*}
\alpha \log\left(1-\frac{S_{\epsilon S_t^{1-1/\gamma}}}{S_t}\right) \sim -\alpha \frac{S\left(\epsilon S^{1-1/\gamma}_t\right)}{S_t} \xrightarrow{\mathbb{P}} 0.
\end{equation*}
In particular, this implies the following convergence in probability for every $t>0$
\begin{equation*}
\frac{1}{S_t}\ind{S_t >1, \, t \geqslant 2\epsilon S_t ^{1-1/\gamma}} \left(1-\frac{S_{\epsilon S_t^{1-1/\gamma}}}{S_t}\right)^\alpha  \to \frac{1}{S_t}\ind{S_t >1}.
\end{equation*}
Since we have the inequality
\begin{equation*}
\frac{1}{S_t}\ind{S_t >1, \, t \geqslant 2\epsilon S_t ^{1-1/\gamma}} \left(1-\frac{S_{\epsilon S_t^{1-1/\gamma}}}{S_t}\right)^\alpha  \leqslant \frac{1}{S_t}\ind{S_t >1}
\end{equation*}
where the right-hand side is integrable with respect to $\mathbf{1}_{(0,\infty)}(t)\, \dd t \otimes \mathbb{P}$ thanks to \eqref{eq: tau is integrable}, the dominated convergence theorem yields
\begin{equation*}
\int_0^\infty \dd t \ex{\frac{1}{S_t}\ind{S_t >1, \, t \geqslant 2\epsilon S_t ^{1-1/\gamma}} \left(1-\frac{S_{\epsilon S_t^{1-1/\gamma}}}{S_t}\right)^\alpha } \to \int_0^\infty \dd t \ex{\frac{1}{S_t} \ind{S_t >1}} = \frac{1}{\Gamma(1-1/\gamma)}\cdot
\end{equation*}
Together with \eqref{eq: l1 cv pre} and the fact that
\begin{equation*}
\int_\rdtree \ind{H(x)\geqslant 2\epsilon}\sigma_{\epsilon,x}^\alpha\,\mu(\dd x) \leqslant 1,
\end{equation*}
this proves \eqref{eq: l1 cv}.

Finally, since $\beta \epsilon \to \infty$, it is clear that $\H(1-\e^{-\beta\epsilon/\H}) \to \H$ almost surely. In conjunction with \eqref{eq: l1 cv}, this gives the following convergence in $\excm{1}$-probability
\begin{equation*}
\H\left(1-\e^{- \beta \epsilon/\H}\right) \int_\rdtree \ind{H(x)\geqslant 2\epsilon}\sigma_{\epsilon,x}^\alpha\,\mu(\dd x) \to \H.
\end{equation*}
Thus, using this together \eqref{eq: I upper bound} and \eqref{eq: I lower bound} yields $\lim_{\beta \to \infty}E_\beta= \H$ in $\excm{1}$-probability. It follows from \eqref{eq: decomposition E F} and \eqref{eq: only I 2} that $\lim_{\beta\to \infty} \beta \H^{-\beta} \Z_{\alpha,\beta} = \H$ in $\excm{1}$-probability. This proves the first part of the theorem.

Next, we treat the case $\alpha^{1-1/\gamma}/\beta^\rho \to 0$ for some $\rho \in (0,1)$. The proof is similar and we only highlight the differences. Notice that there exists $p,q \in (0,1)$ and $\theta \in (0,\gamma/(\gamma-1))$ such that $(1+p)q >1$ and $q\theta >\rho\gamma/(\gamma-1)$.  Taking $\epsilon = \beta^{-q}$, it is straighforward to check that $\epsilon \to 0$, $\beta \epsilon \to \infty$, $\beta \epsilon^{1+p} \to 0$ and $ \alpha\epsilon^\theta\to 0$. As in the first part, we have that $\excm{1}$-a.s. $\lim_{\beta \to \infty} F_\beta^1 + F_\beta^3+F_\beta^4 = 0$.

Furthermore, using that $\sigma_{r,x}\leqslant 1$, it follows from \eqref{eq: F2} that
\begin{equation*}
F_\beta^2\leqslant\beta \left(1-\frac{\epsilon}{\H}\right)^\beta \H \leqslant \beta \e^{-\beta \epsilon/\H}\H = \beta \e^{-\beta^{1-q}/\H}.
\end{equation*}
This proves that $\excm{1}$-a.s. $\lim_{\beta \to \infty} F_\beta^2 = 0$.

Now we shall prove that $\excm{1}$-a.s. $\mu(\dd x)$-a.s.
\begin{equation}\label{eq: as cv}
\lim_{\beta \to \infty} \ind{H(x)\geqslant 2\epsilon}\sigma_{\epsilon,x}^\alpha = 1.
\end{equation}
Using the same computation as in \eqref{eq: l1 cv pre}, we have the following identity in distribution
\begin{multline}\label{eq: identity process epsilon}
	\left(\ind{H(x)\geqslant 2\epsilon}\sigma_{\epsilon,x}^\alpha, \, \epsilon >0\right) \quad \text{under } \excm{1} \\ \lawd \quad \left(\ind{t \geqslant 2 \epsilon S_t^{1-1/\gamma}}\left(1-\frac{S\left(\epsilon S_t^{1-1/\gamma}\right)}{S_t}\right)^\alpha,\, \epsilon >0\right) \quad \text{under } \int_0^\infty \dd t \ex{\frac{1}{S_t}\ind{S_t >1}\bullet}.
\end{multline}
Since $\theta < \gamma/(\gamma-1)$, \cite[Chapter III, Theorem 9]{bertoin1996levy} guarantees that $\mathbb{P}$-a.s. $\limsup_{r \to 0} r^{-\theta} S_r = 0$. By composition, it follows that $\mathbb{P}$-a.s. for every $t>0$, $\lim_{\epsilon \to 0} \epsilon^{-\theta} S\left(\epsilon S_t^{1-1/\gamma}\right) =0$. Thus we deduce that
\begin{equation*}
\alpha \log\left(1-\frac{S\left(\epsilon S_t^{1-1/\gamma}\right)}{S_t}\right)\sim -\alpha S\left(\epsilon S_t^{1-1/\gamma}\right) = -\alpha \epsilon^\theta \epsilon^{-\theta}S\left(\epsilon S_t^{1-1/\gamma}\right)\to 0
\end{equation*}
since $\alpha \epsilon^\theta \to 0$. This proves that the process in the right-hand side of \eqref{eq: identity process epsilon} goes to $1$ $\mathbb{P}$-a.s. as $\epsilon \to 0$, thus \eqref{eq: as cv} follows. 

Thanks to \eqref{eq: as cv}, since $\sigma_{\epsilon,x}\leqslant 1$, a simple application of the dominated convergence theorem gives that $\excm{1}$-a.s.
\begin{equation*}
	\lim_{\beta \to \infty}\int_\rdtree \ind{H(x)\geqslant 2\epsilon}\sigma_{\epsilon,x}^\alpha\,\mu(\dd x) = 1.
\end{equation*}
This, together with the estimates \eqref{eq: I upper bound} and \eqref{eq: I lower bound} yields the $\excm{1}$-a.s. convergence $\lim_{\beta \to \infty} E_\beta = \H$ which concludes the proof of the second part of the theorem.
\end{proof}

%Using Bismut's decomposition \eqref{eq: Bismut}, we have
%\begin{equation}\label{eq: as cv pre}
%	\n\left[\int_\rdtree \ind{\liminf_{\beta \to \infty}\left(\frac{\sigma_{\epsilon,x}}{\sigma}\right)^\alpha < 1}\, \mu(\dd x)\right] =\int_0^\infty \pr{\liminf_{\beta \to \infty}\left(1-\frac{S_\epsilon}{\mass_t}\right)^\alpha < 1}\, \dd t.
%\end{equation}
%But we have as $\beta \to \infty$
%\begin{equation*}
%	\alpha \log\left(1-\frac{S_\epsilon}{\mass_t}\right)\sim -\alpha \frac{S_\epsilon}{\mass_t}= -\alpha\epsilon^\theta \frac{\epsilon^{-\theta}S_\epsilon}{\mass_t}.
%\end{equation*}
%Recall that $\alpha \epsilon^\theta \to 0$. Furthermore, since $\theta < \gamma/(\gamma-1)$, \cite[Chapter III, Theorem 9]{bertoin1996levy} guarantees that $\mathbb{P}$-a.s. $\limsup_{\epsilon\to 0}\epsilon^{-\theta}S_\epsilon = 0$. It follows that $\mathbb{P}$-a.s.
%\begin{equation*}
%	\lim_{\beta \to \infty} \left(1-\frac{S_\epsilon}{S_t}\right)^\alpha = \lim_{\beta \to \infty} \exp\left(\alpha \log\left(1-\frac{S_\epsilon}{S_t}\right)\right)=1.
%\end{equation*}
%Thus, \eqref{eq: as cv pre} gives that $\n$-a.e. $\mu(\dd x)$-a.e.
%\begin{equation*}
%	\lim_{\beta \to \infty}\left(\frac{\sigma_{\epsilon,x}}{\sigma}\right)^\alpha =1.
%\end{equation*}
%Now a simple disintegration argument yields \eqref{eq: as cv}.

\section{Technical lemmas}\label{section: technical}
\subsection{Proof of Lemma \ref{lemma: approx point measure}}\label{section: approx point measure}
Recall that $\g(\epsilon) = 1-\f(\epsilon)$. Using the expression of $F(\epsilon)$ from \eqref{eq: expression F}, we write
\begin{multline}\label{eq: decomposition N}
	\Gamma(1-1/\gamma)^{-1}F(\epsilon) - \int_0^\infty \dd t \operatorname{\mathbb{E}}\left[\frac{1}{\mass_{\g(\epsilon)t}}\ind{\mass_{\g(\epsilon)t} >1} f\circ \norma\left(\Tdown_{\g(\epsilon)t}\right) g\left(\mass_{\g(\epsilon)t}^{-1+1/\gamma} t\right)\right.\\
	\left. \times\expp{-\sum_{s\leqslant \f(\epsilon) t} \Phi\left(\epsilon^{-1} s/t,\epsilon^{-\gamma/(\gamma-1)}\mass_{\g(\epsilon)t}^{-1}\mu(\tree_s),\norma\left(\tree_s\right)\right)}\right] =\sum_{i=1}^4 \int_0^\infty\dd t \ex{N_\epsilon^{i}(t)},
\end{multline}
where
\begin{align*}
	N_\epsilon^{1}(t)&=
	\begin{multlined}[t]
		\frac{1}{\mass_t}\ind{\mass_t >1} \left\lbrace f\circ \norma\left(\Tdown_t\right) -f\circ \norma\left(\Tdown_{\g(\epsilon)t}\right)\right\rbrace g\left(\mass_t^{-1+1/\gamma} t\right)\\
		\times\expp{-\sum_{s\leqslant \f(\epsilon) t} \Phi\left(\epsilon^{-1} s/t,\epsilon^{-\gamma/(\gamma-1)}\mass_t^{-1}\mu(\tree_s),\norma\left(\tree_s\right)\right)},
	\end{multlined}\\
	N_\epsilon^{2}(t)&=
	\begin{multlined}[t]
		\frac{1}{\mass_t}\ind{\mass_t >1} f\circ 	\norma\left(\Tdown_{\g(\epsilon)t}\right)\left\lbrace g\left(\mass_t^{-1+1/\gamma} t\right)-g\left(\mass_{\g(\epsilon)t}^{-1+1/\gamma} t\right)\right\rbrace\\
		\times\expp{-\sum_{s\leqslant \f(\epsilon) t} \Phi\left(\epsilon^{-1} s/t,\epsilon^{-\gamma/(\gamma-1)}\mass_t^{-1}\mu(\tree_s),\norma\left(\tree_s\right)\right)},
	\end{multlined}\\
	N_\epsilon^{3}(t)&=
	\begin{multlined}[t]
		\frac{1}{\mass_{t}}\ind{\mass_{t}>1} f\circ \norma\left(\Tdown_{\g(\epsilon)t}\right) g\left(\mass_{\g(\epsilon)t}^{-1+1/\gamma} t\right)\\
		\begin{multlined}[t]
			\times\left[\expp{-\sum_{s\leqslant \f(\epsilon) t} \Phi\left(\epsilon^{-1} s/t,\epsilon^{-\gamma/(\gamma-1)}\mass_t^{-1}\mu(\tree_s),\norma\left(\tree_s\right)\right)}\right.\\
			\left. -\expp{-\sum_{s\leqslant \f(\epsilon) t} \Phi\left(\epsilon^{-1} s/t,\epsilon^{-\gamma/(\gamma-1)}\mass_{\g(\epsilon)t}^{-1}\mu(\tree_s),\norma\left(\tree_s\right)\right)}\right],
		\end{multlined}
	\end{multlined}\\
	N_\epsilon^{4}(t)&=
	\begin{multlined}[t]
		\left\lbrace\frac{1}{\mass_t}\ind{\mass_t >1}-\frac{1}{\mass_{\g(\epsilon)t}}\ind{\mass_{\g(\epsilon)t}>1}\right\rbrace f\circ \norma\left(\Tdown_{\g(\epsilon)t}\right) g\left(\mass_{\g(\epsilon)t}^{-1+1/\gamma} t\right)\\
		\times\expp{-\sum_{s\leqslant \f(\epsilon) t} \Phi\left(\epsilon^{-1} s/t,\epsilon^{-\gamma/(\gamma-1)}\mass_{\g(\epsilon)t}^{-1}\mu(\tree_s),\norma\left(\tree_s\right)\right)}.
	\end{multlined}
\end{align*}

Recall from \eqref{eq: definition norma} the definition of $\norma$ and notice that since the total mass of $\Tdown_t$ is $\mass_t$, we have $\norma(\Tdown_t) = R_\gamma(\Tdown_t, \mass_t^{-
	1+1/\gamma})$. It follows that
\begin{align*}
	\left|N_\epsilon^1(t)\right|&\leqslant \norm{f}_\mathrm{L}\norm{g}_\infty \frac{1}{\mass_t}\ind{\mass_t>1} \ghp\left(\norma\left( \Tdown_t\right),\norma\left(\Tdown_{\g(\epsilon)t}\right)\right)\\
	& \leqslant \begin{multlined}[t]
		\norm{f}_\mathrm{L}\norm{g}_\infty\ind{\mass_t>1}\left[ \ghp\left(R_\gamma\left(\Tdown_t,\mass_t^{-1+1/\gamma}\right), R_\gamma\left(\Tdown_{\g(\epsilon)t},\mass_t^{-1+1/\gamma}\right)\right)\right.\\
		+\left.\ghp\left(R_\gamma\left(\Tdown_{\g(\epsilon)t} ,\mass_t^{-1+1/\gamma}\right),R_\gamma\left(\Tdown_{\g(\epsilon)t},\mass_{\g(\epsilon)t}^{-1+1/\gamma}\right)\right)\right],
	\end{multlined}
\end{align*}
where $\norm{f}_\mathrm{L}$ denotes the Lipschitz constant of $f$. Notice that, by construction, the tree $\Tdown_t$ is obtained from $\Tdown_{\g(\epsilon)t}$ by adding to the root a branch $[0,\f(\epsilon)t)$ onto which we graft $\tree_s$ at height $0\leqslant s<\f(\epsilon)t$. It is clear that the added part has mass $\sum_{s<\f(\epsilon)t}\mu(\tree_s)=S_{f(\epsilon)t-}$  and height at most $\max_{s<\f(\epsilon)t} \H(\tree_s) +\f(\epsilon)t$. Thus, by definition \eqref{eq: definition R} of the mapping $R_\gamma$, we deduce that
\begin{equation}\label{eq: ghp1}
	\ghp\left(R_\gamma\left(\Tdown_t,\mass_t^{-1+1/\gamma}\right), R_\gamma\left(\Tdown_{\g(\epsilon)t},\mass_t^{-1+1/\gamma}\right)\right)\leqslant \mass_t^{-1} S_{\f(\epsilon)t-} + \mass_t^{-1+1/\gamma}\left(\max_{s<\f(\epsilon)t}\H(\tree_s)+\f(\epsilon)t\right).
\end{equation}
Moreover, using Lemma \ref{lemma: ghp dilatation} and again the definition of $R_\gamma$, we get
\begin{multline}\label{eq: ghp2}
	\ghp\left(R_\gamma\left(\Tdown_{\g(\epsilon)t} ,\mass_t^{-1+1/\gamma}\right),R_\gamma\left(\Tdown_{\g(\epsilon)t},\mass_{\g(\epsilon)t}^{-1+1/\gamma}\right)\right)\\
	\leqslant 2\left(\mass_{\g(\epsilon)t}^{-1+1/\gamma}-\mass_t^{-1+1/\gamma}\right)\H\left(\Tdown_{\g(\epsilon)t}\right)+\left(\mass_{\g(\epsilon)t}^{-1}-\mass_t^{-1}\right)\mu\left(\Tdown_{\g(\epsilon)t}\right).
\end{multline}
From \eqref{eq: ghp1} and \eqref{eq: ghp2}, we deduce that
\begin{multline*}
	\left|N_\epsilon^1(t)\right|\leqslant \norm{f}_\mathrm{L}\norm{g}_\infty\left[S_{\f(\epsilon)t-} + \max_{s<\f(\epsilon)t}\H(\tree_s)+\f(\epsilon)t\right.\\
	\left.+2\left(\mass_{\g(\epsilon)t}^{-1+1/\gamma}-\mass_t^{-1+1/\gamma}\right)\H\left(\Tdown_{\g(\epsilon)t}\right)+\left(\mass_{\g(\epsilon)t}^{-1}-\mass_t^{-1}\right)\mu\left(\Tdown_{\g(\epsilon)t}\right)\right].
\end{multline*}
Therefore it follows that for every $t>0$ $\mathbb{P}$-a.s. 
\begin{equation}\label{eq: N1}
	\lim_{\epsilon \to 0}N_\epsilon^{1}(t) = 0.
\end{equation}

Furthermore, it is clear that
\begin{equation*}
	\left|N_\epsilon^2(t)\right|\leqslant \norm{f}_\infty \norm{g}_\mathrm{L}t \left|\mass_t^{-1+1/\gamma} - \mass_{\g(\epsilon)t}^{-1+1/\gamma}\right|.
\end{equation*}
Thus, we have for every $t>0$ $\mathbb{P}$-a.s.
\begin{equation}\label{eq: N2}
	\lim_{\epsilon \to 0}N_\epsilon^{2}(t) = 0.
\end{equation}
Since
\begin{equation*}
	\left|N_\epsilon^1(t)+N_\epsilon^2(t)\right|\leqslant 4\norm{f}_\infty \norm{g}_\infty \frac{1}{\mass_t}\ind{\mass_t >1}
\end{equation*}
where the right-hand side is integrable with respect to $\mathbf{1}_{(0,\infty)}(t)\, \dd t \otimes \operatorname{\mathbb{P}}$ thanks to \eqref{eq: tau is integrable}, it follows from \eqref{eq: N1} and \eqref{eq: N2} that
\begin{equation}\label{eq: N1-2}
	\lim_{\epsilon\to 0}\int_0^\infty \dd t \ex{N_\epsilon^1(t)+N_\epsilon^2(t)} = 0.
\end{equation}

Using the inequality $|\e^b-\e^{a}|\leqslant 1\wedge|b-a|$ for $a\leqslant b \leqslant 0$, we have
\begin{align}\label{eq: N3 pre}
	\left|N_\epsilon^3(t)\right|&\leqslant 
	\begin{multlined}[t]\norm{f}_\infty \norm{g}_\infty\frac{1}{\mass_{t}}\ind{\mass_{t} >1}\left(1\wedge\sum_{s\leqslant \f(\epsilon)t} \left|\Phi\left(\epsilon^{-1}s/t,\epsilon^{-\gamma/(\gamma-1)}\mass_t^{-1}\mu(\tree_s),\norma\left(\tree_s\right)\right)\right.\right.\\
		\left.\left.-\Phi\left(\epsilon^{-1}s/t,\epsilon^{-\gamma/(\gamma-1)}\mass_{\g(\epsilon)t}^{-1}\mu(\tree_s),\norma\left(\tree_s\right)\right)\right|\right)
	\end{multlined}\nonumber\\
	&\leqslant \norm{f}_\infty \norm{g}_\infty \frac{1}{\mass_{t}}\ind{\mass_{t} >1}\left(1\wedge C\epsilon^{-\gamma/(\gamma-1)} \left|\mass_{t}^{-1}-\mass_{\g(\epsilon)t}^{-1}\right|\sum_{s\leqslant \f(\epsilon) t}\mu(\tree_s)\right)\nonumber\\
	&= \norm{f}_\infty \norm{g}_\infty \frac{1}{\mass_{t}}\ind{\mass_{t} >1}\left(1\wedge C\epsilon^{-\gamma/(\gamma-1)} \frac{\left(\mass_t-\mass_{\g(\epsilon)t}\right)^2}{\mass_t \mass_{\g(\epsilon)t}}\right).
\end{align}
Since $\mass$ is a stable subordinator with index $1-1/\gamma$, we get that
\begin{equation*}
	\epsilon^{-\gamma/(\gamma-1)}\left(\mass_t-\mass_{\g(\epsilon)t}\right)^2 \lawd \epsilon^{-\gamma/(\gamma-1)}\mass_{\f(\epsilon)t}^2 \lawd \left(\epsilon^{-1}\f(\epsilon)^2\right)^{\gamma/(\gamma-1)} \mass_t^2 \xrightarrow[\epsilon \to 0]{(d)} 0
\end{equation*}
as $\epsilon^{-1}\f(\epsilon)^2 \to 0$. We deduce the following convergence in $\mathbb{P}$-probability
\begin{equation*}
	\lim_{\epsilon \to 0} \frac{1}{\mass_{t}}\ind{\mass_{t} >1}\left(1\wedge C\epsilon^{-\gamma/(\gamma-1)} \frac{\left(\mass_t-\mass_{\g(\epsilon)t}\right)^2}{\mass_t \mass_{\g(\epsilon)t}}\right) = 0.
\end{equation*}
Thanks to \eqref{eq: tau is integrable}, it follows from the dominated convergence theorem that
\begin{equation*}
	\lim_{\epsilon \to 0} \int_0^\infty \dd t \ex{\frac{1}{\mass_{t}}\ind{\mass_{t} >1}\left(1\wedge C\epsilon^{-\gamma/(\gamma-1)} \frac{\left(\mass_t-\mass_{\g(\epsilon)t}\right)^2}{\mass_t \mass_{\g(\epsilon)t}}\right)} = 0.
\end{equation*}
Together with \eqref{eq: N3 pre}, this gives
\begin{equation}\label{eq: N3}
	\lim_{\epsilon \to 0} \int_0^\infty \dd t\ex{N_\epsilon^3(t)} = 0.
\end{equation}

Finally, notice that
\begin{equation}\label{eq: N4 pre}
	\left|\int_0^\infty\dd t\ex{N_\epsilon^4(t)}\right|
	\leqslant \norm{f}_\infty\norm{g}_\infty \int_0^\infty \dd t\ex{\frac{1}{\mass_t} \ind{\mass_{\g(\epsilon)t} \leqslant 1 < \mass_t} + \frac{\mass_t-\mass_{\g(\epsilon)t} }{\mass_t\mass_{\g(\epsilon)t} }\ind{\mass_{\g(\epsilon)t} >1}}.
\end{equation}
Thanks to \eqref{eq: tau is integrable} and the dominated convergence theorem, it is clear that
\begin{equation}\label{eq: N4 part 1}
	\lim_{\epsilon \to 0} \int_0^\infty \dd t \ex{\frac{1}{\mass_t}\ind{\mass_{\g(\epsilon)t} \leqslant 1 < \mass_t}}=0
\end{equation}
as the process $\mass$ is a.s. continuous at $t$. On the other hand, using the inequality
\begin{equation*}
	\frac{\mass_t-\mass_{\g(\epsilon)t} }{\mass_t\mass_{\g(\epsilon)t} }\ind{\mass_{\g(\epsilon)t} >1} \leqslant \left(\frac{\mass_t-\mass_{\g(\epsilon)t} }{\mass_t}\right)^{1-q} \frac{\left(\mass_t-\mass_{\g(\epsilon)t} \right)^q}{\mass_{\g(\epsilon)t} ^{1+q}} \ind{\mass_{\g(\epsilon)t} >1} \leqslant \frac{\left(\mass_t-\mass_{\g(\epsilon)t} \right)^q}{\mass_{\g(\epsilon)t} ^{1+q}}\ind{\mass_{\g(\epsilon)t} >1}
\end{equation*}
where $q \in (0,1-1/\gamma)$, we get that
\begin{align}\label{eq: eta moment}
	\int_0^\infty \dd t\ex{\frac{\mass_t-\mass_{\g(\epsilon)t} }{\mass_t\mass_{\g(\epsilon)t} }\ind{\mass_{\g(\epsilon)t}  >1}}&\leqslant \int_0^\infty \dd t\ex{\frac{\left(\mass_t-\mass_{\g(\epsilon)t} \right)^q}{\mass_{\g(\epsilon)t} ^{1+q}}\ind{\mass_{\g(\epsilon)t} >1}} \nonumber\\
	&=  \int_0^\infty\dd t \ex{\mass_{\f(\epsilon)t}^q} \ex{\frac{1}{\mass_{\g(\epsilon)t}^{1+q}}\ind{\mass_{\g(\epsilon)t}>1}}\nonumber\\
	&= \f(\epsilon)^{q\gamma/(\gamma-1)}\g(\epsilon)^{-1-q\gamma/(\gamma-1)}\ex{\mass_1^q} \int_0^\infty \dd r\, r^{q\gamma/(\gamma-1)} \ex{\frac{1}{\mass_r^{1+q}}\ind{\mass_r>1}}\nonumber\\
	&=\f(\epsilon)^{q\gamma/(\gamma-1)}\g(\epsilon)^{-1-q\gamma/(\gamma-1)}\ex{\mass_1^q}\ex{\frac{1}{\mass_1^{1+q}}\int_{\mass_1^{-1+1/\gamma}}^\infty \dd r\, r^{-\gamma/(\gamma-1)}}\nonumber\\
	&= \f(\epsilon)^{q\gamma/(\gamma-1)}\g(\epsilon)^{-1-q\gamma/(\gamma-1)}\ex{\mass_1^q}\ex{\mass_1^{-1-q+1/\gamma }},
\end{align}
where we used that $\mass_t-\mass_{\g(\epsilon)t}$ is independent of $\mass_{\g(\epsilon)t}$ and is distributed as $\mass_{\f(\epsilon)t}$ for the first equality and that $\mass_{t} \law t^{\gamma/(\gamma-1)}\mass_1$ for the second. Thanks to \eqref{eq: tau finite moment}, we have $\ex{\mass_1^q} < \infty$ and $\ex{\mass_1^{-1+1/\gamma-q}}< \infty$. Thus, it follows from \eqref{eq: eta moment} that
\begin{equation}\label{eq: N4 part 2}
	\lim_{\alpha \to \infty} 	\int_\epsilon^\infty \ex{\frac{\mass_t-\mass_{t-\epsilon} }{\mass_t\mass_{t-\epsilon} }\ind{\mass_{t-\epsilon}  >1}}\, \dd t  = 0.
\end{equation}
Combining \eqref{eq: N4 pre}, \eqref{eq:  N4 part 1} and \eqref{eq: N4 part 2}, we deduce that
\begin{equation}\label{eq: N4}
	\lim_{\epsilon \to 0}\int_0^\infty \dd t \ex{\left|N_\epsilon^4(t)\right|} =0.
\end{equation}

It follows from \cref{eq: decomposition N,eq: N1-2,eq: N3,eq: N4} that
\begin{multline*}
	\lim_{\epsilon \to 0}\Gamma(1-1/\gamma)^{-1}F(\epsilon) - \int_0^\infty \dd t \operatorname{\mathbb{E}}\left[\frac{1}{\mass_{\g(\epsilon)t}}\ind{\mass_{\g(\epsilon)t} >1} f\circ R\left(\Tdown_{\g(\epsilon)t}, \mass_{\g(\epsilon)t}^{-1}\right) g\left(\mass_{\g(\epsilon)t}^{-1+1/\gamma} t\right)\right.\\
	\left. \times\expp{-\sum_{s\leqslant \f(\epsilon) t} \Phi\left(\epsilon^{-1} s/t,\epsilon^{-\gamma/(\gamma-1)}\mass_{\g(\epsilon)t}^{-1}\mu(\tree_s),R\left(\tree_s, \mu(\tree_s)^{-1}\right)\right)}\right] = 0.
\end{multline*}

\subsection{Proof of Lemma \ref{lemma: convergence in probability J}}\label{section: remainder}
Recall from \eqref{eq: definition I} the definition of $I_\alpha$. Write $\alpha^{1-1/\gamma}\H^{-\beta} Z_{\alpha,\beta}(U) - I_\alpha = \sum_{i=1}^4 J_\alpha^{i}$ where
	\begin{align*}
		J^1_\alpha &= \alpha^{1-1/\gamma}\H^{-\beta} \int_{\epsilon H(U)}^{H(U)}\sigma_{r,U}^\alpha \H_{r,U}^\beta\, \dd r, \\
		J^2_\alpha &= \alpha^{1-1/\gamma} \H^{-\beta} \int_0^{\epsilon H(U)} \sigma_{r,U}^\alpha \left\{\left(\frac{\H_{r,U}}{\H}\right)^\beta-\left(1-\frac{r}{\H}\right)^\beta\right\}\, \dd r,\\
		J^3_\alpha &= \alpha^{1-1/\gamma}  \int_0^{\epsilon H(U)} \sigma_{r,U}^\alpha \left\{\left(1-\frac{r}{\H}\right)^\beta-\e^{-\beta r/\H}\right\}\, \dd r,\\
		J^4_\alpha &= \alpha^{1-1/\gamma}  \int_0^{\epsilon H(U)} \left\{\sigma_{r,U}^\alpha-\e^{-\alpha\left(1-\sigma_{r,U}\right)}\right\} \e^{-\beta r/\H}\, \dd r.
	\end{align*}
	We shall prove that for every $1\leqslant i \leqslant 4$, $\lim_{\alpha \to \infty} J_\alpha^{i} = 0$ in $\excm{1}$-probability.
	
	We start by showing that $\excm{1}$-a.s. $\mu(\dd x)$-a.s.
	\begin{equation}\label{eq: simple desintigration argument 0}
			\lim_{\alpha \to \infty} \alpha^{1-1/\gamma} \int_{\epsilon H(x)}^{H(x)} \sigma_{r,x}^\alpha \, \dd r =0.
	\end{equation}
	Recall from \eqref{eq: definition S} the definition of $S$. Using Lemma \ref{lemma: sigma = 1 ou > 1}-(i) and Bismut's decomposition \eqref{eq: Bismut}, we have
	\begin{multline}\label{eq: identity in distribution subordinator}
		\Gamma(1-1/\gamma)^{-1}\excm{1}\left[ \mu\left(x\in \rdtree \colon\,\limsup_{\alpha \to \infty}\alpha^{1-1/\gamma}\int_{\epsilon H(x)}^{H(x)} \sigma_{r,x}^\alpha\, \dd r>0\right)\right]\\
		\begin{aligned}[b]
			&= \n\left[\frac{1}{\sigma}\ind{\sigma>1} \mu\left(x\in \rdtree\colon\, \limsup_{\alpha \to \infty}\left(\frac{\alpha}{\sigma}\right)^{1-1/\gamma}\int_{\epsilon H(x)}^{H(x)} \left(\frac{\sigma_{r,x}}{\sigma}\right)^\alpha\, \dd r>0\right)\right]\\
			&= \int_0^\infty \dd t \ex{\frac{1}{\mass_t}\ind{\mass_t>1}; \limsup_{\alpha \to \infty} \left(\frac{\alpha}{\mass_t}\right)^{1-1/\gamma} \int_{\epsilon t}^t \left(1-\frac{S_r}{\mass_t}\right)^\alpha \, \dd r>0}.
			\end{aligned}
	\end{multline}
Let $t >0$. It is clear that
\begin{equation}\label{eq: inequality subordinator}
 \int_{\epsilon t}^t \left(1-\frac{S_r}{\mass_t}\right)^\alpha \, \dd r\leqslant \int_{\epsilon t}^t \e^{-\alpha S_r/\mass_t}\, \dd r 
	\leqslant t\e^{-\alpha S_{\epsilon t}/\mass_t}.
\end{equation}
 According to \cite[Chapter III, Theorem 11]{bertoin1996levy}, we have that $\mathbb{P}$-a.s.
	\begin{equation*}
		\liminf_{\epsilon\to 0} \frac{S_{\epsilon t}}{h(\epsilon t)} = \gamma - 1>0,
	\end{equation*}
	where $h(r) = r^{\gamma/(\gamma-1)}\log\left(\left|\log r \right|\right)^{-1/(\gamma-1)}$. As a consequence, there exist a positive random variable $\rho = \rho(\omega)$ and a constant $c>0$ such that $\mathbb{P}$-a.s. $S_{\epsilon t} \geqslant ch(\epsilon t)$ for every $\epsilon \in (0,\rho)$. We deduce that for every $t>0$, $\mathbb{P}$-a.s.
	\begin{align*}
		\limsup_{\alpha \to \infty}\alpha^{1-1/\gamma} \e^{-\alpha S_{\epsilon t}/\mass_t} &\leqslant\limsup_{\alpha \to \infty}\alpha^{1-1/\gamma} \e^{-c\alpha h(\epsilon t)/\mass_t}\\
		&= \limsup_{\alpha \to \infty}\alpha^{1-1/\gamma} \e^{-c t^{\gamma/(\gamma-1)}\alpha^\delta \log(|\log\left(\epsilon t\right)|)^{-1}/\mass_t} = 0,
	\end{align*}
where in the last equality we used \eqref{eq: definition epsilon}. In conjunction with \eqref{eq: identity in distribution subordinator} and \eqref{eq: inequality subordinator}, this yields \eqref{eq: simple desintigration argument 0}.

Let $\eta >0$. Using that $\H_{r,U} \leqslant \H$, we have
\begin{align*}
	\limsup_{\alpha \to \infty}\excm{1}\left[J_\alpha^1>\eta\right] &\leqslant \limsup_{\alpha \to \infty}\excm{1}\left[\alpha^{1-1/\gamma}\int_{\epsilon H(U)}^{H(U)} \sigma_{r,U}^\alpha \, \dd r>\eta\right]\\
	&= \limsup_{\alpha \to \infty}\excm{1}\left[\mu\left(x \in \rdtree \colon \, \alpha^{1-1/\gamma}\int_{\epsilon H(x)}^{H(x)} \sigma_{r,x}^\alpha \, \dd r >\eta\right)\right],
\end{align*}
where the last term vanishes thanks to \eqref{eq: simple desintigration argument 0} and the dominated convergence theorem. This gives that $\lim_{\alpha \to \infty} J_\alpha^1 = 0$ in $\excm{1}$-probability.

	Under $\excm{1}$, let $x^*$ be the unique leaf realizing the total height, that is the unique $x \in \rdtree$ such that  $H(x) = \H$. Then $\excm{1}$-a.s. we have $H(U\wedge x^*) >0$ and, thanks to  \eqref{eq: height is linear}, $\H_{r,U}= \H-r$ for every $r \in [0,\epsilon H(U)]$ if $\epsilon >0$ is small enough (more precisely for $\epsilon \leqslant H(U\wedge x^*)/H(U)$). In particular, this implies that $\excm{1}$-a.s. $\lim_{\alpha \to \infty}J_\alpha^2 =0$.
	
	Next, we have
	\begin{align*}
		|J_\alpha^{3}|
		&\leqslant \alpha^{1-1/\gamma} \int_0^{\epsilon H(U)} \sigma_{r,U}^\alpha \left|\left(1-\frac{r}{\H}\right)^\beta-\e^{-\beta r/\H}\right|\, \dd r \\
		&\leqslant\alpha^{1-1/\gamma} \beta \int_0^{\epsilon H(U)} \sigma_{r,U}^\alpha\left|\log\left(1-\frac{r}{\H}\right)+\frac{r}{\H}\right| \e^{-\beta r/\H} \, \dd r \\
		&\leqslant  \alpha^{1-1/\gamma}\beta j\left(\frac{\epsilon H(U)}{\H}\right)  \int_0^{\epsilon H(U)} \sigma_{r,U}^\alpha\frac{r^2}{\H^2}\e^{-\beta r/\H}\, \dd r\\
		&\leqslant C H(U) j\left(\epsilon\right)\epsilon^3\alpha^{2(1-1/\gamma)},
	\end{align*} 
	where we used that $|\e^b - \e^{a}| \leqslant |b-a| \e^b$ for $a\leqslant b$ for the second inequality, that the function $j \colon y \mapsto -(y+\log(1-y))/y^2$ is increasing on $[0,1)$ for the third and the fact that $H(U)\leqslant \H$ and $\beta/\alpha^{1-1/\gamma}$ is bounded by some constant $C>0$ for the last. Using \eqref{eq: definition epsilon}, notice that $\epsilon^3 \alpha^{2(1-1/\gamma)}=\alpha^{(3\delta-1)(1-1/\gamma)} \to 0$ as $\delta < 1/3$. Since $\lim_{y\to 0} j(y) =1/2$, we deduce that $\excm{1}$-a.s. $\lim_{\alpha \to \infty}J_\alpha^3 = 0$.
	
	Finally, we have
	\begin{align*}
		|J_\alpha^4| &\leqslant \alpha^{2-1/\gamma} \int_0^{\epsilon H(U)} \left|\log\left(\sigma_{r,U}\right)+ 1-\sigma_{r,U}\right|\e^{-\alpha\left(1-\sigma_{r,U}\right)}\, \dd r\\
		&\leqslant j\left(1-\sigma_{\epsilon H(U)}\right)\alpha^{2-1/\gamma}  \int_0^{\epsilon H(U)} \left(1-\sigma_{r,U}\right)^2 \e^{-\alpha\left(1-\sigma_{r,U}\right)}\, \dd r \\
		&\leqslant CH(U) j\left(1-\sigma_{\epsilon H(U)}\right)\alpha^{-1/\gamma} \epsilon,
	\end{align*}
	where we used that $|\e^b-\e^{a}| \leqslant |b-a|\e^{b}$ for $a\leqslant b$ for the first inequality, that the function $j \colon x \mapsto -(x+\log(1-x))/x^2$ is increasing on $[0,1)$  for the second and that the function $x\mapsto x^2 \e^{-x}$ is bounded on $[0,\infty)$ for the last. Since $\lim_{x \to 0} j(x) = 1/2$, $\lim_{\epsilon \to 0}\sigma_{\epsilon H(U)} =1$ and $\alpha^{-1/\gamma} \epsilon \to 0$, we deduce that $\excm{1}$-a.s. $\lim_{\alpha \to \infty}J_\alpha^{4} = 0$.

\subsection{Proof of Lemma \ref{lemma: convergence in proba}}\label{section: convergence in proba}
It is enough to show that for every Lipschitz-continuous and bounded function $f \colon [0,\infty) \to \real$
\begin{equation*}
	\lim_{\alpha \to \infty} \excm{1}\left[\int_\rdtree \mu(\dd x) f\left(\alpha^{1-1/\gamma} \H^{-\beta}\left(Z_{\alpha,\beta}(x) - \Z_{\alpha,\beta}\right)\right)\right] = f(0).
\end{equation*}

Let $\epsilon = \alpha^{(\delta-1)(1-1/\gamma)}$ with $\delta \in (0,1/2)$. For every $x\in \rdtree$ such that $H(x)\geqslant \epsilon$, set
\begin{equation*}
	Z_{\alpha,\beta}^\epsilon(x)= \int_0^{\epsilon} \sigma_{r,x}^\alpha \H_{r,x}^\beta \, \dd r \quad \text{and} \quad \Z_{\alpha,\beta}^\epsilon = \int_\rdtree \ind{H(x)\geqslant \epsilon} Z_{\alpha,\beta}^\epsilon(x)\, \mu(\dd x).
\end{equation*}
Let $x^*\in \rdtree$ be the unique leaf realizing the height, that is $H(x^*) = \H$. Using that $\H\geqslant H(x\wedge x^*) $ and that $Z_{\alpha,\beta}^\epsilon(x) = Z_{\alpha,\beta}^\epsilon(x^*)$ if $\epsilon \leqslant H(x\wedge x^*)$, write
\begin{equation*}
\int_\rdtree \mu(\dd x) f\left(\alpha^{1-1/\gamma} \H^{-\beta}\left(Z_{\alpha,\beta}(x) - \Z_{\alpha,\beta}\right)\right) = \sum_{i=1}^{4} A_\alpha^{i}+B_\alpha,
\end{equation*}
where
\begin{align*}
	A_\alpha^{1} &= \int_\rdtree \mu(\dd x)\ind{H(x\wedge x^*)<\epsilon} f\left(\alpha^{1-1/\gamma} \H^{-\beta}\left(Z_{\alpha,\beta}(x) - \Z_{\alpha,\beta}\right)\right),\\
	A_\alpha^2 &= \int_\rdtree \mu(\dd x)\ind{H(x\wedge x^*)\geqslant \epsilon} \left\{f\left(\alpha^{1-1/\gamma} \H^{-\beta}\left(Z_{\alpha,\beta}(x) - \Z_{\alpha,\beta}\right)\right)-f\left(\alpha^{1-1/\gamma} \H^{-\beta}\left(Z^\epsilon_{\alpha,\beta}(x) - \Z_{\alpha,\beta}\right)\right)\right\},\\
	A_\alpha^3 &= \int_\rdtree \mu(\dd x)\ind{H(x\wedge x^*)\geqslant \epsilon} \left\{f\left(\alpha^{1-1/\gamma} \H^{-\beta}\left(Z_{\alpha,\beta}^\epsilon(x) - \Z_{\alpha,\beta}\right)\right)-f\left(\alpha^{1-1/\gamma} \H^{-\beta}\left(Z^\epsilon_{\alpha,\beta}(x) - \Z_{\alpha,\beta}^\epsilon\right)\right)\right\},\\
	A_\alpha^4 &=- \mu\left(\left\{x\in \rdtree \colon \, H(x\wedge x^*) <\epsilon\right\}\right)f\left(\ind{\H\geqslant \epsilon}\alpha^{1-1/\gamma} \H^{-\beta}\left(Z^\epsilon_{\alpha,\beta}(x^*) - \Z_{\alpha,\beta}^\epsilon\right)\right),\\
	B_\alpha &= f\left(\ind{\H\geqslant \epsilon}\alpha^{1-1/\gamma} \H^{-\beta}\left(Z^\epsilon_{\alpha,\beta}(x^*) - \Z^\epsilon_{\alpha,\beta}\right)\right).
\end{align*}

Thanks to the dominated convergence theorem, we have
\begin{equation}\label{eq: A1}
	\lim_{\alpha \to \infty}\excm{1}[|A_\alpha^1+A_\alpha^4|]\leqslant 2\norm{f}_\infty \lim_{\alpha \to \infty}\excm{1}\left[\int_\rdtree \mu(\dd x)\ind{H(x\wedge x^*)<\epsilon}\right]=0.
\end{equation}

Next, notice that
\begin{align}\label{eq: A2 pre}
	\excm{1}[|A_\alpha^2|]&\leqslant \norm{f}_{\mathrm{L}}\excm{1}\left[\alpha^{1-1/\gamma}\H^{-\beta}\int_\rdtree \mu(\dd x) \ind{H(x\wedge x^*)\geqslant \epsilon)} \left(Z_{\alpha,\beta}(x)-Z_{\alpha,\beta}^\epsilon(x)\right)\right]\nonumber\\
	&\leqslant \norm{f}_{\mathrm{L}} \excm{1}\left[\alpha^{1-1/\gamma}\int_\rdtree \mu(\dd x) \ind{H(x)\geqslant \epsilon} \int_\epsilon^{H(x)}\sigma_{r,x}^\alpha\, \dd r\right],
\end{align}
where we used that $H(x\wedge x^*)\leqslant H(x)$ and $\H_{r,x}\leqslant \H$ for the second inequality. Now similarly to \eqref{eq: simple desintigration argument 0}, we have $\excm{1}$-a.s. $\mu(\dd x)$-a.s.
\begin{equation}\label{eq: simple disintegration argument}
	\lim_{\alpha \to \infty} \alpha^{1-1/\gamma}\ind{H(x)\geqslant \epsilon} \int_\epsilon^{H(x)} \sigma_{r,x}^\alpha \, \dd r = 0.
\end{equation}
Furthermore, applying Corollary \ref{lemma: second moment}, we have
\begin{multline*}
	\sup_{\alpha \geqslant 0}\alpha^{2-2/\gamma}\excm{1}\left[\int_\rdtree \mu(\dd x)\left(\ind{H(x)\geqslant \epsilon}\int_\epsilon^{H(x)} \sigma_{r,x}^\alpha \, \dd r\right)^2\right]\\
	\leqslant \sup_{\alpha \geqslant 0} \alpha^{2-2/\gamma}\excm{1}\left[\int_\rdtree \mu(\dd x) \left(\int_0^{H(x)}\sigma_{r,x}^\alpha\, \dd r\right)^2 \right] < \infty.
\end{multline*}
We deduce that the family
\begin{equation*}
	\left(\alpha^{1-1/\gamma}\ind{H(x)\geqslant \epsilon}\int_\epsilon^{H(x)} \sigma_{r,x}^\alpha \, \dd r, \, \alpha \geqslant 0\right)
\end{equation*}
is uniformly integrable under the measure $\excm{1}[\dd \rdtree] \mu(\dd x)$. In conjunction with \eqref{eq: simple disintegration argument}, this gives
\begin{equation}\label{eq: L1 convergence}
	\lim_{\alpha \to \infty}\excm{1}\left[\alpha^{1-1/\gamma} \int_\rdtree \ind{H(x)\geqslant \epsilon}\,\mu(\dd x)\int_\epsilon^{H(x)} \sigma_{r,x}^\alpha\, \dd r\right]=0,
\end{equation}
which, thanks to \eqref{eq: A2 pre}, implies that
\begin{equation}\label{eq: A2}
	\lim_{\alpha \to \infty}\excm{1}[|A_\alpha^2|]=0.
\end{equation}

We have
\begin{align}\label{eq: A3 pre}
	\excm{1}[|A_\alpha^3|]&\leqslant \norm{f}_{\mathrm{L}} \excm{1}\left[\alpha^{1-1/\gamma}\H^{-\beta}\int_\rdtree \mu(\dd x)\ind{H(x\wedge x^*)\geqslant \epsilon} \left(\Z_{\alpha,\beta}-\Z_{\alpha,\beta}^\epsilon\right)\right]\nonumber\\
	&\leqslant \norm{f}_\mathrm{L} \excm{1}\left[\alpha^{1-1/\gamma}\H^{-\beta} \left(\Z_{\alpha,\beta}-\Z_{\alpha,\beta}^\epsilon\right)\right]\nonumber\\
	& \begin{multlined}[b]
		 \leqslant\norm{f}_{\mathrm{L}}\excm{1}\left[\alpha^{1-1/\gamma} \int_\rdtree \ind{H(x)\geqslant \epsilon}\, \mu(\dd x)\int_\epsilon^{H(x)}\sigma_{r,x}^\alpha\, \dd r\right]\\
		 \quad+\norm{f}_{\mathrm{L}}\excm{1}\left[\alpha^{1-1/\gamma}\int_\rdtree \ind{H(x)<\epsilon} \, \mu(\dd x)\int_0^{H(x)}\sigma_{r,x}^\alpha\, \dd r\right],
		\end{multlined}
\end{align}
where we used that $\H_{r,x}\leqslant \H$ for the last inequality. Let $p\in (1,2)$ and notice that $\epsilon^{1+p}\alpha^{1-1/\gamma} \to 0$. Using that $\sigma_{r,x}\leqslant 1$ together with the Markov inequality, we get
\begin{align*}
	\excm{1}\left[\alpha^{1-1/\gamma}\int_\rdtree \ind{H(x)<\epsilon} \, \mu(\dd x)\int_0^{H(x)}\sigma_{r,x}^\alpha\, \dd r\right] &\leqslant \excm{1}\left[\epsilon\alpha^{1-1/\gamma}\int_\rdtree \ind{H(x)<\epsilon} \, \mu(\dd x)\right]\\
	&\leqslant \epsilon^{1+p} \alpha^{1-1/\gamma} \excm{1}\left[\int_\rdtree H(x)^{-p}\, \mu(\dd x)\right].
\end{align*}
By Lemma \ref{lemma: height integrable}, the last term is finite. This, in conjunction with \eqref{eq: L1 convergence} and \eqref{eq: A3 pre}, implies that
\begin{equation}\label{eq: A3}
	\lim_{\alpha \to \infty}\excm{1}[|A_\alpha^3|]=0.
\end{equation}

It remains to show that $\lim_{\alpha \to \infty}\excm{1}[B_\alpha] = f(0)$, which is equivalent to the following convergence in $\excm{1}$-probability
\begin{equation}\label{eq: uniform leaf convergence in proba}
	\lim_{\alpha \to \infty}\ind{\H\geqslant \epsilon}\alpha^{1-1/\gamma} \H^{-\beta} \left(Z_{\alpha,\beta}^\epsilon(x^*) - \Z_{\alpha,\beta}^\epsilon\right) =0.
\end{equation}
Again using that $Z_{\alpha,\beta}^\epsilon(x) = Z_{\alpha,\beta}^\epsilon(x^*)$ if $\epsilon\leqslant H(x\wedge x^*)$, we write
\begin{equation*}
	\ind{\H\geqslant \epsilon}\alpha^{1-1/\gamma} \H^{-\beta} \left(Z_{\alpha,\beta}^\epsilon(x^*) - \Z_{\alpha,\beta}^\epsilon\right) = B_\alpha^1 +B_\alpha^2,
\end{equation*}
where
\begin{align*}
	B_\alpha^1 &=  \alpha^{1-1/\gamma}\H^{-\beta}\left(\ind{\H\geqslant\epsilon}Z_{\alpha,\beta}^\epsilon(x^*) -\int_\rdtree \mu(\dd x) \ind{H(x\wedge x^*)\geqslant \epsilon}Z_{\alpha,\beta}^\epsilon(x^*)\right),\\
	B_\alpha^2 &= \alpha^{1-1/\gamma} \H^{-\beta}\left(\int_\rdtree \mu(\dd x) \ind{H(x\wedge x^*)\geqslant \epsilon}Z_{\alpha,\beta}^\epsilon(x)-\ind{\H\geqslant \epsilon} \Z_{\alpha,\beta}^\epsilon\right).
\end{align*}

Recall that $\epsilon = \alpha^{(\delta-1)(1-1/\gamma)}\to 0$ as $\alpha \to \infty$. Fix $\eta >0$ and let $\alpha_0>0$ be large enough so that for every $\alpha \geqslant \alpha_0$
\begin{equation*}
	\excm{1} \left[\int_\rdtree \mu(\dd x)\ind{H(x\wedge x^*)<\epsilon}\right] \leqslant \eta.
\end{equation*}
Then we have for every $\alpha \geqslant \alpha_0$ and $C>0$
\begin{multline}\label{eq: tight 1}
	\excm{1}\left[\alpha^{1-1/\gamma}\H^{-\beta}Z_{\alpha,\beta}^\epsilon(x^*)\ind{\H\geqslant \epsilon}\geqslant C\right] \\
	\begin{aligned}[b]
		&\leqslant \excm{1}\left[\int_\rdtree \mu(\dd x) \ind{\alpha^{1-1/\gamma}\H^{-\beta}Z_{\alpha,\beta}^\epsilon(x)\geqslant C, \, H(x\wedge x^*)\geqslant \epsilon}\right] + \excm{1}\left[\int_\rdtree \mu(\dd x)\ind{H(x\wedge x^*)< \epsilon}\right] \\
		&\leqslant \frac{\alpha^{2-2/\gamma}}{C^2}\excm{1}\left[\int_\rdtree \mu(\dd x)\ind{H(x\wedge x^*) \geqslant \epsilon}\left(\H^{-\beta}Z_{\alpha,\beta}^\epsilon(x)\right)^2\right] + \eta\\
		&\leqslant \frac{\alpha^{2-2/\gamma}}{C^2}\excm{1}\left[\int_\rdtree \mu(\dd x)\left(\int_0^{H(x)}\sigma_{r,x}^\alpha \, \dd r\right)^2\right] + \eta \\
		&\leqslant \frac{M}{C^2}+ \eta
	\end{aligned}
\end{multline}
for some constant $M>0$, where we used that $Z_{\alpha,\beta}^\epsilon(x^*) = Z_{\alpha,\beta}^\epsilon(x)$ for every $x\in \rdtree$ such that $H(x\wedge x^*)\geqslant \epsilon$ for the first inequality, the Markov inequality for the second and Corollary \ref{lemma: second moment} for the last. Thus, we get that the family $\left(\ind{\H\geqslant \epsilon}\alpha^{1-1/\gamma}\H^{-\beta}Z_{\alpha,\beta}^\epsilon(x^*), \, \alpha\geqslant \alpha_0,\,\beta \geqslant 0\right)$ is tight. Since $\excm{1}$-a.s. 
\begin{equation*}\lim_{\alpha \to \infty} \int_\rdtree \mu(\dd x)\ind{H(x\wedge x^*)< \epsilon}= 0,
\end{equation*}
we deduce the following convergence in $\excm{1}$-probability
\begin{equation*}
	\lim_{\alpha \to \infty}B_\alpha^1 =\lim_{\alpha \to \infty}\ind{\H\geqslant \epsilon}\alpha^{1-1/\gamma}\H^{-\beta}Z_{\alpha,\beta}^\epsilon(x^*)\int_\rdtree \mu(\dd x)\ind{H(x\wedge x^*)<\epsilon}=0.
\end{equation*}

Furthermore, we have
\begin{align*}
	\excm{1}[|B_\alpha^2|] &= \alpha^{1-1/\gamma} \excm{1}\left[\int_\rdtree \mu(\dd x)\ind{H(x)\geqslant \epsilon, \, H(x\wedge x^*)<\epsilon}\H^{-\beta}Z_{\alpha,\beta}^\epsilon(x)\right]\\
	&\leqslant \alpha^{1-1/\gamma} \excm{1}\left[\int_\rdtree \mu(\dd x) \left(\ind{H(x\wedge x^*)<\epsilon}\int_0^{H(x)}\sigma_{r,x}^\alpha\, \dd r\right)\right]\\
	&\leqslant \alpha^{1-1/\gamma}  \excm{1}\left[\int_\rdtree\mu(\dd x)\left(\int_0^{H(x)}\sigma_{r,x}^\alpha \, \dd r\right)^2\right]^{1/2}\excm{1}\left[\int_\rdtree \mu(\dd x)\ind{H(x \wedge x^*)<\epsilon}\right]^{1/2}\\
	&\leqslant C\excm{1}\left[\int_\rdtree \mu(\dd x)\ind{H(x \wedge x^*)<\epsilon}\right]^{1/2}
\end{align*}
for some constant $C>0$, where we used the Cauchy-Schwarz inequality for the second inequality and Corollary \ref{lemma: second moment} for the last. It follows from the dominated convergence theorem that $\lim_{\alpha \to \infty} \excm{1}[|B_\alpha^2|]=0$. This finishes the proof of \eqref{eq: uniform leaf convergence in proba}.

\medskip	
\bibliographystyle{amsplain}
\bibliography{costfunctionals}

\begin{bibdiv}
\begin{biblist}

\bib{abraham2013forest}{article}{
      author={Abraham, Romain},
      author={Delmas, Jean-Fran{\c{c}}ois},
       title={The forest associated with the record process on a {L{\'e}vy}
  tree},
        date={2013},
     journal={Stochastic Process. Appl.},
      volume={123},
      number={9},
       pages={3497\ndash 3517},
      review={\MR{3071387}},
}

\bib{abraham2020global}{misc}{
      author={Abraham, Romain},
      author={Delmas, Jean-François},
      author={Nassif, Michel},
       title={Global regime for general additive functionals of conditioned
  {Bienaymé}-{Galton}-{Watson} trees},
        date={2020},
}

\bib{addario2017scaling}{article}{
      author={Addario-Berry, Louigi},
      author={Broutin, Nicolas},
      author={Goldschmidt, Christina},
      author={Miermont, Gr{\'e}gory},
       title={The scaling limit of the minimum spanning tree of the complete
  graph},
        date={2017},
     journal={Ann. Probab.},
      volume={45},
      number={5},
       pages={3075\ndash 3144},
      review={\MR{3706739}},
}

\bib{aldous1991continuum}{article}{
      author={Aldous, David},
       title={The continuum random tree. {I}},
        date={1991},
     journal={Ann. Probab.},
       pages={1\ndash 28},
      review={\MR{1085326}},
}

\bib{berestycki2007beta}{article}{
      author={Berestycki, Julien},
      author={Berestycki, Nathana\"{e}l},
      author={Schweinsberg, Jason},
       title={Beta-coalescents and continuous stable random trees},
        date={2007},
        ISSN={0091-1798},
     journal={Ann. Probab.},
      volume={35},
      number={5},
       pages={1835\ndash 1887},
         url={https://doi-org.extranet.enpc.fr/10.1214/009117906000001114},
      review={\MR{2349577}},
}

\bib{bertoin1996levy}{book}{
      author={Bertoin, Jean},
       title={L{\'e}vy processes},
   publisher={Cambridge University Press, Cambridge},
        date={1996},
      volume={121},
      review={\MR{1406564}},
}

\bib{bertoin2002selfsimilar}{article}{
      author={Bertoin, Jean},
       title={Self-similar fragmentations},
        date={2002},
        ISSN={0246-0203},
     journal={Ann. Inst. H. Poincar\'{e} Probab. Statist.},
      volume={38},
      number={3},
       pages={319\ndash 340},
         url={https://doi-org.extranet.enpc.fr/10.1016/S0246-0203(00)01073-6},
      review={\MR{1899456}},
}

\bib{billingsley1999convergence}{book}{
      author={Billingsley, Patrick},
       title={Convergence of probability measures},
     edition={Second},
      series={Wiley Series in Probability and Statistics: Probability and
  Statistics},
   publisher={John Wiley \& Sons, Inc., New York},
        date={1999},
        ISBN={0-471-19745-9},
         url={https://doi.org/10.1002/9780470316962},
        note={A Wiley-Interscience Publication},
      review={\MR{1700749}},
}

\bib{delmas2018}{article}{
      author={Delmas, Jean-François},
      author={Dhersin, Jean-Stéphane},
      author={Sciauveau, Marion},
       title={Cost functionals for large (uniform and simply generated) random
  trees},
        date={2018},
     journal={Electron. J. Probab.},
      volume={23},
       pages={36 pp.},
         url={https://doi.org/10.1214/18-EJP213},
      review={\MR{3858915}},
}

\bib{duquesne2003limit}{article}{
      author={Duquesne, Thomas},
       title={A limit theorem for the contour process of condidtioned
  {Galton}-{Watson} trees},
        date={2003},
     journal={Ann. Probab.},
      volume={31},
      number={2},
       pages={996\ndash 1027},
      review={\MR{1964956}},
}

\bib{duquesne2002random}{book}{
      author={Duquesne, Thomas},
      author={Le~Gall, Jean-Fran\c{c}ois},
       title={Random trees, {L}{\'e}vy processes and spatial branching
  processes},
   publisher={Soci{\'e}t{\'e} math{\'e}matique de France},
        date={2002},
      volume={281},
      review={\MR{1954248}},
}

\bib{duquesne2005probabilistic}{article}{
      author={Duquesne, Thomas},
      author={Le~Gall, Jean-Fran\c{c}ois},
       title={Probabilistic and fractal aspects of {L}\'{e}vy trees},
        date={2005},
        ISSN={0178-8051},
     journal={Probab. Theory Related Fields},
      volume={131},
      number={4},
       pages={553\ndash 603},
         url={https://doi.org/10.1007/s00440-004-0385-4},
      review={\MR{2147221}},
}

\bib{duquesne2017decomposition}{article}{
      author={Duquesne, Thomas},
      author={Wang, Minmin},
       title={Decomposition of {L}\'{e}vy trees along their diameter},
        date={2017},
        ISSN={0246-0203},
     journal={Ann. Inst. Henri Poincar\'{e} Probab. Stat.},
      volume={53},
      number={2},
       pages={539\ndash 593},
         url={https://doi.org/10.1214/15-AIHP725},
      review={\MR{3634265}},
}

\bib{duquesne2007growth}{article}{
      author={Duquesne, Thomas},
      author={Winkel, Matthias},
       title={Growth of {L}\'{e}vy trees},
        date={2007},
        ISSN={0178-8051},
     journal={Probab. Theory Related Fields},
      volume={139},
      number={3-4},
       pages={313\ndash 371},
         url={https://doi-org.extranet.enpc.fr/10.1007/s00440-007-0064-3},
      review={\MR{2322700}},
}

\bib{evans2007probability}{book}{
      author={Evans, Steven~N.},
       title={Probability and real trees},
      series={Lecture Notes in Mathematics},
   publisher={Springer, Berlin},
        date={2008},
      volume={1920},
        ISBN={978-3-540-74797-0},
         url={https://doi-org.extranet.enpc.fr/10.1007/978-3-540-74798-7},
        note={Lectures from the 35th Summer School on Probability Theory held
  in Saint-Flour, July 6--23, 2005},
      review={\MR{2351587}},
}

\bib{fillsum}{article}{
      author={Fill, James~Allen},
      author={Janson, Svante},
       title={The sum of powers of subtree sizes for conditioned
  {Galton}--{Watson} trees},
}

\bib{goldschmidt2010behavior}{article}{
      author={Goldschmidt, Christina},
      author={Haas, B\'{e}n\'{e}dicte},
       title={Behavior near the extinction time in self-similar fragmentations.
  {I}. {T}he stable case},
        date={2010},
        ISSN={0246-0203},
     journal={Ann. Inst. Henri Poincar\'{e} Probab. Stat.},
      volume={46},
      number={2},
       pages={338\ndash 368},
         url={https://doi.org/10.1214/09-AIHP317},
      review={\MR{2667702}},
}

\bib{haas2015line}{article}{
      author={Goldschmidt, Christina},
      author={Haas, B\'{e}n\'{e}dicte},
       title={A line-breaking construction of the stable trees},
        date={2015},
     journal={Electron. J. Probab.},
      volume={20},
       pages={no. 16, 24},
         url={https://doi-org.extranet.enpc.fr/10.1214/EJP.v20-3690},
      review={\MR{3317158}},
}

\bib{haas2007fragmentation}{article}{
      author={Haas, B\'{e}n\'{e}dicte},
       title={Fragmentation processes with an initial mass converging to
  infinity},
        date={2007},
        ISSN={0894-9840},
     journal={J. Theoret. Probab.},
      volume={20},
      number={4},
       pages={721\ndash 758},
         url={https://doi-org.extranet.enpc.fr/10.1007/s10959-007-0120-z},
      review={\MR{2359053}},
}

\bib{haas2012scaling}{article}{
      author={Haas, B\'{e}n\'{e}dicte},
      author={Miermont, Gr\'{e}gory},
       title={Scaling limits of {M}arkov branching trees with applications to
  {G}alton-{W}atson and random unordered trees},
        date={2012},
        ISSN={0091-1798},
     journal={Ann. Probab.},
      volume={40},
      number={6},
       pages={2589\ndash 2666},
         url={https://doi.org/10.1214/11-AOP686},
      review={\MR{3050512}},
}

\bib{haas2009spinal}{article}{
      author={Haas, B\'{e}n\'{e}dicte},
      author={Pitman, Jim},
      author={Winkel, Matthias},
       title={Spinal partitions and invariance under re-rooting of continuum
  random trees},
        date={2009},
        ISSN={0091-1798},
     journal={Ann. Probab.},
      volume={37},
      number={4},
       pages={1381\ndash 1411},
         url={https://doi.org/10.1214/08-AOP434},
      review={\MR{2546748}},
}

\bib{kortchemski2013simple}{incollection}{
      author={Kortchemski, Igor},
       title={A simple proof of {D}uquesne's theorem on contour processes of
  conditioned {G}alton-{W}atson trees},
        date={2013},
   booktitle={S\'{e}minaire de {P}robabilit\'{e}s {XLV}},
      series={Lecture Notes in Math.},
      volume={2078},
   publisher={Springer, Cham},
       pages={537\ndash 558},
         url={https://doi-org.extranet.enpc.fr/10.1007/978-3-319-00321-4_20},
      review={\MR{3185928}},
}

\bib{le1998branching}{article}{
      author={Le~Gall, Jean-Fran\c{c}ois},
      author={Le~Jan, Yves},
       title={Branching processes in {L}\'{e}vy processes: {L}aplace
  functionals of snakes and superprocesses},
        date={1998},
        ISSN={0091-1798},
     journal={Ann. Probab.},
      volume={26},
      number={4},
       pages={1407\ndash 1432},
         url={https://doi.org/10.1214/aop/1022855868},
      review={\MR{1675019}},
}

\bib{marchal2008note}{incollection}{
      author={Marchal, Philippe},
       title={A note on the fragmentation of a stable tree},
        date={2008},
   booktitle={Fifth {C}olloquium on {M}athematics and {C}omputer {S}cience},
      series={Discrete Math. Theor. Comput. Sci. Proc., AI},
   publisher={Assoc. Discrete Math. Theor. Comput. Sci., Nancy},
       pages={489\ndash 499},
      review={\MR{2508809}},
}

\bib{miermont2003fragmentations}{article}{
      author={Miermont, Gr\'{e}gory},
       title={Self-similar fragmentations derived from the stable tree. {I}.
  {S}plitting at heights},
        date={2003},
        ISSN={0178-8051},
     journal={Probab. Theory Related Fields},
      volume={127},
      number={3},
       pages={423\ndash 454},
         url={https://doi-org.extranet.enpc.fr/10.1007/s00440-003-0295-x},
      review={\MR{2018924}},
}

\bib{miermont2005fragmentations}{article}{
      author={Miermont, Gr\'{e}gory},
       title={Self-similar fragmentations derived from the stable tree. {II}.
  {S}plitting at nodes},
        date={2005},
        ISSN={0178-8051},
     journal={Probab. Theory Related Fields},
      volume={131},
      number={3},
       pages={341\ndash 375},
         url={https://doi-org.extranet.enpc.fr/10.1007/s00440-004-0373-8},
      review={\MR{2123249}},
}

\bib{pitman2006combinatorial}{book}{
      author={Pitman, Jim},
       title={Combinatorial stochastic processes},
      series={Lecture Notes in Mathematics},
   publisher={Springer-Verlag, Berlin},
        date={2006},
      volume={1875},
        ISBN={978-3-540-30990-1; 3-540-30990-X},
        note={Lectures from the 32nd Summer School on Probability Theory held
  in Saint-Flour, July 7--24, 2002},
      review={\MR{2245368}},
}

\bib{resnick2007heavy}{book}{
      author={Resnick, Sidney~I.},
       title={Heavy-tail phenomena},
      series={Springer Series in Operations Research and Financial
  Engineering},
   publisher={Springer, New York},
        date={2007},
        ISBN={978-0-387-24272-9; 0-387-24272-4},
        note={Probabilistic and statistical modeling},
      review={\MR{2271424}},
}

\bib{tyrankaminska2010convergence}{article}{
      author={Tyran-Kami\'{n}ska, Marta},
       title={Convergence to {L}\'{e}vy stable processes under some weak
  dependence conditions},
        date={2010},
        ISSN={0304-4149},
     journal={Stochastic Process. Appl.},
      volume={120},
      number={9},
       pages={1629\ndash 1650},
         url={https://doi.org/10.1016/j.spa.2010.05.010},
      review={\MR{2673968}},
}

\bib{wolfe1975onmoments}{inproceedings}{
      author={Wolfe, Stephen~J.},
       title={On moments of probability distribution functions},
        date={1975},
   booktitle={Fractional calculus and its applications ({P}roc. {I}nternat.
  {C}onf., {U}niv. {N}ew {H}aven, {W}est {H}aven, {C}onn., 1974)},
       pages={306\ndash 316. Lecture Notes in Math., Vol. 457},
      review={\MR{0474463}},
}

\bib{zolotarev1986one}{book}{
      author={Zolotarev, V.~M.},
       title={One-dimensional stable distributions},
      series={Translations of Mathematical Monographs},
   publisher={American Mathematical Society, Providence, RI},
        date={1986},
      volume={65},
        ISBN={0-8218-4519-5},
         url={https://doi.org/10.1090/mmono/065},
        note={Translated from the Russian by H. H. McFaden, Translation edited
  by Ben Silver},
      review={\MR{854867}},
}

\end{biblist}
\end{bibdiv}

\end{document}